\documentclass[english,a4paper,12pt]{amsart}
\usepackage{amssymb}
\usepackage{amsfonts}
\usepackage{amscd}
\usepackage{longtable}
\usepackage{pb-diagram}
\usepackage[all]{xy}
\usepackage{color}
\usepackage{enumerate}
\usepackage{stmaryrd}
\usepackage[utf8]{inputenc}
\usepackage{mathtools}
\usepackage{tikz-cd}
\usepackage{todonotes}
\usepackage{enumitem}
\usepackage{cmap} 
\usepackage[url=false,maxbibnames=99,isbn=false,giveninits=true,doi=false]{biblatex} 

\usepackage[T1]{fontenc}
\RequirePackage[colorlinks=true, breaklinks=true, urlcolor= black, linkcolor= black, citecolor= black, bookmarksopen=true,bookmarksdepth=2,linktocpage=true,plainpages=false,pdfpagelabels]{hyperref}

\allowdisplaybreaks

\CompileMatrices

\swapnumbers

\def\rad{\operatorname{rad}}

\def\radop{\rad_{\mathrm{op}}}

\def\ac{{\rm ac}}

\def\val{{\mathrm{val}}}

\def\omin{{\mathrm{omin}}}  

\def\rank{{\rm rank}}

\def\11{{\mathbf 1}}

\def\CC{{\mathbb C}}

\def\LL{{\mathfrak L}}

\def\NN{{\mathbb N}}

\def\QQ{{\mathbb Q}}
\def\RR{{\mathbb R}}

\def\ZZ{{\mathbb Z}}

\def\cA{{\mathcal A}}
\def\cB{{\mathcal B}}
\def\cC{{\mathcal C}}

\def\cL{{\mathcal L}}
\def\cM{{\mathcal M}}

\def\cO{{\mathcal O}}

\def\cR{{\mathcal R}}
\def\cS{{\mathcal S}}
\def\cT{{\mathcal T}}

\mathchardef\alphag="7C0B \mathchardef\betag="7C0C
\mathchardef\gammag="7C0D \mathchardef\deltag="7C0E
\mathchardef\varepsilong="7C22 \mathchardef\varphig="7C27
\mathchardef\psig="7C20 \mathchardef\zetag="7C10
\mathchardef\epsilong="7C0F \mathchardef\rhog="7C1A
\mathchardef\taug="7C1C \mathchardef\upsilong="7C1D
\mathchardef\iotag="7C13 \mathchardef\thetag="7C12
\mathchardef\pig="7C19 \mathchardef\sigmag="7C1B
\mathchardef\etag="7C11 \mathchardef\omegag="7C21
\mathchardef\kappag="7C14 \mathchardef\lambdag="7C15
\mathchardef\mug="7C16 \mathchardef\xig="7C18
\mathchardef\chig="7C1F \mathchardef\nug="7C17
\mathchardef\varthetag="7C23 \mathchardef\varpig="7C24
\mathchardef\varrhog="7C25 \mathchardef\varsigmag="7C26
\mathchardef\Omegag="7C0A \mathchardef\Thetag="7C02
\mathchardef\Sigmag="7C06 \mathchardef\Deltag="7C01
\mathchardef\Phig="7C08 \mathchardef\Gammag="7C00
\mathchardef\Psig="7C09 \mathchardef\Lambdag="7C03
\mathchardef\Xig="7C04 \mathchardef\Pig="7C05
\mathchardef\Upsilong="7C07

\newtheorem{theorem}[subsubsection]{Theorem}
\newtheorem{thm}[subsubsection]{Theorem}
\newtheorem{lem}[subsubsection]{Lemma}
\newtheorem{cor}[subsubsection]{Corollary}
\newtheorem{prop}[subsubsection]{Proposition}

{Addendum}

\theoremstyle{definition}
\newtheorem{defn}[subsubsection]{Definition}

\newtheorem{notn}[subsubsection]{Notation}
\newtheorem{example}[subsubsection]{Example}

\newtheorem{def-prop}[subsubsection]{Proposition-Definition}
\newtheorem{def-theorem}[subsubsection]{Theorem-Definition}
\newtheorem{def-lem}[subsubsection]{Lemma-Definition}

\theoremstyle{remark}
\newtheorem{remark}[subsubsection]{Remark}

\newtheorem{question}[subsubsection]{Question}

\theoremstyle{plain}

\numberwithin{equation}{subsection}

\def\boxit#1#2{\setbox1=\hbox{\kern#1{#2}\kern#1}%
\dimen1=\ht1 \advance\dimen1 by #1 \dimen2=\dp1 \advance\dimen2 by
#1
\setbox1=\hbox{\vrule height\dimen1 depth\dimen2\box1\vrule}%
\setbox1=\vbox{\hrule\box1\hrule}%
\advance\dimen1 by .4pt \ht1=\dimen1 \advance\dimen2 by .4pt
\dp1=\dimen2 \box1\relax}

\renewcommand{\theequation}{\thesubsection.\arabic{equation}}

\mathchardef\alphag="7C0B \mathchardef\betag="7C0C
\mathchardef\gammag="7C0D \mathchardef\deltag="7C0E
\mathchardef\varepsilong="7C22 \mathchardef\varphig="7C27
\mathchardef\psig="7C20 \mathchardef\zetag="7C10
\mathchardef\epsilong="7C0F \mathchardef\rhog="7C1A
\mathchardef\taug="7C1C \mathchardef\upsilong="7C1D
\mathchardef\iotag="7C13 \mathchardef\thetag="7C12
\mathchardef\pig="7C19 \mathchardef\sigmag="7C1B
\mathchardef\etag="7C11 \mathchardef\omegag="7C21
\mathchardef\kappag="7C14 \mathchardef\lambdag="7C15
\mathchardef\mug="7C16 \mathchardef\xig="7C18
\mathchardef\chig="7C1F \mathchardef\nug="7C17
\mathchardef\varthetag="7C23 \mathchardef\varpig="7C24
\mathchardef\varrhog="7C25 \mathchardef\varsigmag="7C26
\mathchardef\Omegag="7C0A \mathchardef\Thetag="7C02
\mathchardef\Sigmag="7C06 \mathchardef\Deltag="7C01
\mathchardef\Phig="7C08 \mathchardef\Gammag="7C00
\mathchardef\Psig="7C09 \mathchardef\Lambdag="7C03
\mathchardef\Xig="7C04 \mathchardef\Pig="7C05
\mathchardef\Upsilong="7C07

\newcommand{\RV}{\mathrm{RV}}
\newcommand{\VF}{\mathrm{VF}}
\newcommand{\VG}{\mathrm{VG}}
\newcommand{\RF}{\mathrm{RF}}
\newcommand{\RRV}{\mathrm{RRV}}
\newcommand{\VFR}{\mathrm{VFR}}
\newcommand{\rv}{\operatorname{rv}}

\newcommand{\res}{\operatorname{res}}
\newcommand{\Th}{\operatorname{Th}}

\newcommand{\trop}{\mathrm{trop}}
\newcommand{\vol}{\mathrm{vol}}

\renewcommand{\sp}{\mathrm{sp}}

\DeclareMathOperator*{\Ob}{Ob}
\DeclareMathOperator*{\fibdim}{fibdim}
\DeclareMathOperator{\tr}{tr}
\DeclareMathOperator{\id}{id}

\DeclareMathOperator{\gr}{graph}

\DeclareMathOperator{\Fn}{Fn}

\newcommand{\alg}{{\mathrm{alg}}}

\newcommand{\rcl}{\mathrm{rcl}}

\newcommand{\an}{\mathrm{an}}
\newcommand{\ring}{\mathrm{ring}}
\newcommand{\intr}{\mathrm{intr}}

\DeclareMathOperator{\K}{K}

\def\Jac{\operatorname{Jac}}
\newcommand{\abs}[1]{\lvert#1\rvert}
\DeclarePairedDelimiter{\norm}{\lVert}{\rVert}

\newbox\removebox
\newcommand\remove[1]{%
\setbox\removebox=\ifmmode\hbox{$#1$}\else\hbox{#1}\fi%
\leavevmode
\rlap{\textcolor{blue}{\vrule height0.8ex depth-0.6ex width\wd\removebox}}%
\box\removebox
}
\long\def\bigremove#1{%
\par\setbox\removebox=\vbox{#1}%
\vbox{%
\vbox to0pt{\hbox{\tikz\draw[color=blue,thick] (0,0) -- (\wd\removebox,-\ht\removebox)  (\wd\removebox,0) -- (0,-\ht\removebox);}}
\box\removebox
}
}

\newcommand\acl{\mathrm{acl}}
\newcommand\dcl{\mathrm{dcl}}
\newcommand\bdd{\mathrm{bdd}}

\definecolor{orange}{rgb}{1,0.5,0}

\newcommand{\private}[1]{\leavevmode{\scriptsize\color{blue}\marginpar{{\scriptsize Private comment}}#1\par}}
\renewcommand{\private}[1]{}

\thanks{}

\title{Integration in Hensel minimal fields}

\author{Mathias Stout and Floris Vermeulen}


\begin{document}

\begin{abstract}
We develop a framework of motivic integration in the style of Hrushovski--Kazhdan in arbitrary Hensel minimal fields of equicharacteristic zero. 
Hence our work generalizes that of Hrushovski--Kazhdan and Yin, but applies more broadly to discretely valued fields, almost real closed fields with analytic structure, pseudo-local fields, and coarsenings.
In more detail, we obtain isomorphisms of Grothendieck rings of definable sets, with or without volume forms, in the valued field sort and in the leading term sort. 
Along the way we develop a theory of effective 1-h-minimal structures, where finite definable sets can be lifted from the leading term sort to the valued fields sort.
We show that many natural examples of 1-h-minimal structures are effective, and develop dimension theory and a theory of differentiation in $\RV$ for effective structures.
\end{abstract}

\maketitle

\setcounter{secnumdepth}{2}
\tableofcontents

\section{Introduction}\label{sec:intro}

%


In the nineties Kontsevich introduced motivic integration to prove that birational Calabi--Yau varieties have the same Hodge numbers, thereby generalizing a theorem due to Batyrev~\cite{batyrev-betti}.
The theory was subsequently heavily developed by Denef--Loeser~\cite{DLBarc, DLJAG, DL, DLinvent} in a geometric and arithmetic-geometric way, and later by Loeser--Nicaise--Sebag~\cite{LoeserSeb, NicaSeba} based on N\'eron models. 
Since then, motivic integration has grown to a vast subject with applications in singularity theory and birational geometry~\cite{Baty-Moreau, Batyrev, Veys:stringy, NicaSeba, MustJAMS}.

Later work by Cluckers--Loeser~\cite{CLoes} and Hrushovski--Kazhdan~\cite{HK} provide model-theoretic frameworks for motivic integration based on a deep understanding of definable sets in valued fields.
These approaches naturally lead to a theory of parametric integrals with Fubini type theorems, and are well-suited for the development of a motivic theory of distributions, Fourier transforms~\cite{CLexp, HKPoiss} and Mellin transforms~\cite{CLNV24}.
Moreover, the model-theoretic approach naturally gives transfer principles between characteristic zero and large positive characteristic, generalizing the classical Ax--Kochen--Ersov principle~\cite{CLoesAKE}.
These model-theoretic frameworks have seen applications to the Langlands program~\cite{CHL, YGordon, CGH2, Casselman-Cely-Hales, ForeyLoeserWyss} and birational geometry~\cite{HruLoeser:mono, NicaiseOttem, NicaiseOttem2, ForeyVirtual}.

In a sense, the theories of Cluckers--Loeser and Hrushovski--Kazhdan are somewhat orthogonal.
Cluckers--Loeser integration applies to discretely valued fields, since it relies crucially on summation of certain series over $\ZZ$.
On the other hand, the theory of Hrushovski--Kazhdan is more geometric and primarily applies to algebraically closed valued fields of equicharacteristic zero, with some tools to reduce to henselian subfields.  
These frameworks also differ on a more philosophical level.
For Cluckers--Loeser, the ring of motivic functions is given explicitly by a set of generators and relations.
The difficult part is then to show that these rings are closed under integration, which constitutes the majority of the work in~\cite{CLoes}.
Hrushovski--Kazhdan provide a universal ring of motivic functions, which is closed under integration essentially by definition.
However, it is not obvious what this ring of functions looks like, and much of the work in~\cite{HK} is devoted to understanding this ring in terms of the leading term sorts. 
Nevertheless, the ingredients in both frameworks are quite similar, and rely crucially on an analysis of definable sets and definable maps in valued fields.

The aim of this article is to develop a universal theory of motivic integration in arbitrary $1$-h-minimal theories of equicharacteristic zero, in the style of Hrushovski--Kazhdan.
As such, our theory generalizes that of Hrushovski--Kazhdan, but applies more broadly to discretely valued fields, $\cT_{\omin}$-convex valued fields (see also~\cite{Yin.tcon}), almost real closed fields with analytic structure, and pseudo-local fields.
Additionally, since our framework applies also to discretely valued fields, a specialization to Cluckers--Loeser motivic integration is possible, which we will realize in upcoming work.
From a different perspective, one can consider our main results on the universal Euler characteristic on valued fields as a generalization of the AKE principle to formulas rather than sentences.
Let us go more into details.

\subsection{Main results} 
Let $K$ be a valued field of equicharacteristic zero, considered in some language $\cL\supset \cL_{\val}= \{0,1,+,\cdot, \cO\}$ such that $\Th_{\cL}(K)$ is $1$-h-minimal. 
We denote the valuation ring by $\cO_K$ and the maximal ideal by $\cM_K$.
We denote by $\RV = K^\times / (1+\cM_K)\cup\{0\}$ the leading term sort, and by $\rv\colon K\to \RV$ the natural quotient map extended by $\rv(0)=0$.
Write $\RV^\times = \RV\setminus \{0\}$.
The value group of $K$ is denoted by $\Gamma_K^\times$, with multiplicative valuation $|\cdot |\colon K\to \Gamma_K = \Gamma_K^\times \cup\{0\}$.

In short, our main results concern a precise description of the Grothendieck ring of definable sets $K$ in terms of the Grothendieck ring of definable sets in $\RV$.
We provide such results both with and without measures.

Our main results relate the Grothendieck semiring $\K_+\VF$ to the graded Grothendieck semiring $\K_+\RV[*]$.
Essentially, these are generated by elements of the form $[X]$ and $[R]_n$ respectively, where $X \subset K^n$ and $R \subset (\RV^{\times})^n$ are $\emptyset$-definable. 
We refer to Section~\ref{sec:int} for precise definitions.
For now, we simply note that these semirings depend only on the theory of $K$ in $\cL$, and not on the specific model.

Our first goal is to provide a comparison between these semirings, which is achieved by the following theorem.
Denote $\RV^\times_{<1} = \{\xi\in \RV^\times \mid |\xi| < 1\}$.

\begin{thm}[{{{Theorem~\ref{thm:non.eff.int}}}}]\label{thm:main.not.eff}
Let $\cT$ be a $1$-h-minimal theory of equicharacteristic zero valued fields.
Then there is a natural surjective morphism of semirings
\[
\int\colon \K_+ \VF\to \K_+ \RV[*] / I_{\mathrm{sp}}
\]
where $I_{\mathrm{sp}} = ([1]_1, [\RV^\times_{<1}]_1 + [1]_0)$, such that for every $\emptyset$-definable $R\subset (\RV^\times)^n$ we have $\int([\rv^{-1}(R)]_n) = [R]_n$.
\end{thm}
The map $\int$ essentially boils down to cell decomposition.

The congruence $I_{\mathrm{sp}}$ is exactly the same one as in~\cite{HK}.
The fact that the congruence $I_{\mathrm{sp}}$ appears is not surprising. 
Indeed, $[1]_1$ corresponds to the definable set $1+\cM_K\subset K$, while $[\RV^\times_{<1}]_1 + [1]_0$ corresponds to $\cM_K\subset K$, and the classes $[1+\cM_K]$ and $[\cM_K]$ are obviously equal in $\K_+ \VF$.
What is more surprising is that this simple congruence is precisely the cokernel of $\int$.

In general, $\int$ is not injective. 
Indeed, a simple obstruction is given by the existence of a $\emptyset$-definable point $\xi \in \RV$ such that $\rv^{-1}(\xi)$ contains no $\emptyset$-definable point of $K$.
Essentially, this turns out to be the only obstruction: we show that $\int$ is a bijection if $\cT$ is \emph{effective}, similar to~\cite{HK}.
Effectivity means that for every model $K$ of $\cT$, every parameter set $A\subset K$ and every $A$-definable $\xi\in \RV$, there exists an $A$-definable $x\in K$ with $\rv(x) = \xi$.
In other words, elements from $\RV$ can be lifted to the valued field in a uniformly definable way.

In Section~\ref{sec:eff} we provide many natural examples of effective $1$-h-minimal theories, which can be summarized as follows.
\begin{enumerate}
\item (Corollary~\ref{cor:vf.effective}) The theory $\cT$ of an equicharacteristic zero valued field in the language $\cL_\val$ is effective if the residue field is algebraically bounded over the prime field.
In particular this applies when the residue field is algebraically closed, real closed, $p$-adically closed, or pseudofinite.
This remains true when adding analytic structure to the language as in~\cite{CLR}.
\item (Proposition~\ref{prop:t.convex.effective}) Let $\cT_\omin$ be a complete power-bounded theory of o-minimal fields. 
Then the corresponding $\cT_\omin$-convex theory is effective.
This is the setting studied by Yin in~\cite{Yin.tcon}.
\item (Proposition~\ref{prop:real.an.effective}) The theory $\cT$ of an ordered valued field with convex valuation ring equipped with analytic structure from a strong and rich real Weierstrass system as in~\cite{NSV24} is effective.
\item (Proposition~\ref{prop:coarsening.eff}) If $\cT$ is any $1$-h-minimal theory (possibly of mixed characteristic), then the theory of any proper nontrivial equicharacteristic zero coarsening becomes effective.
\end{enumerate}

In general 1-h-minimal theories, the structure $\RV$ can exhibit wild geometric behaviour.
One may think of effective 1-h-minimal theories as those in which the $\RV$-sort itself exhibits tame geometric behaviour.
In Section~\ref{sec:dim.RV} we develop dimension theory for definable sets in $\RV$ for effective theories, which interacts well with the dimension theory on the valued field.
The main properties of this dimension function are contained in Theorem~\ref{thm:dim.theory.RV}.
In Section~\ref{sec:intrinsic.jac} we develop a notion of derivatives for definable maps on $\RV$ for effective theories, which again interacts well with derivatives on $K$.

Non-examples of effective theories include the theory of $\QQ((t))$ in $\cL_\val$, or any theory in which an angular component map is definable, see Section~\ref{sec:non.examples}.

For effective theories, our main result then becomes the following.
This generalizes~\cite[Thm.\,8.8]{HK} and~\cite[Thm.\,5.40]{Yin.tcon}.

\begin{thm}[{{{Theorem~\ref{thm:integration.semiring}}}}]\label{thm:main.eff}
Let $\cT$ be an effective $1$-h-minimal theory of equicharacteristic zero valued fields.
Then the map $\int\colon \K_+ \VF\to \K_+ \RV[*]/I_\sp$ is an isomorphism of semirings.
\end{thm}

The proofs of Theorems~\ref{thm:main.not.eff} and~\ref{thm:main.eff} take up all of Section~\ref{sec:int}, and form the technical heart of this paper.

In Section~\ref{sec:measures}, we introduce measures and consider definable sets with volume forms.
Let $\K_+ \mu \VF[*]$ denote the Grothendieck semiring of $\emptyset$-definable sets $X\subset K^n$ with a volume form $\omega\colon X\to \RV^\times$.
Intuitively, two classes $[X,\omega]$ and $[Y,\rho]$ are considered the same if there is a measure-preserving bijection defined almost everywhere between $X$ and $Y$.

Similarly, we introduce the Grothendieck semiring of $\emptyset$-definable sets in $\RV$ with volume forms, denoted by $\K_+ \mu \RV[*]$.
The definition of this object is somewhat more complicated, since in general it is not clear what measure-preserving should mean for maps on $\RV$.
For many well-known theories there is a theory of differentiation on $\RV$, and in these cases the notion of measure-preserving is what one would expect.
For example, this is the case when the residue field of $K$ is algebraically closed, o-minimal, or itself $1$-h-minimal.
In Section~\ref{sec:intrinsic.jac} we also develop a notion of derivatives on algebraically bounded fields, which may be used to give a more intrinsic definition of the semiring $\K_+ \mu \RV[*]$.

The precise definitions of the semirings $\K_+ \mu \VF[*]$ and $\K_+ \mu \RV[*]$ are given in Section~\ref{ssec:volume.forms}, and we give several natural variants e.g.\ for $\Gamma_K$-valued volume forms, for bounded sets, or when only considering volumes of definable sets.

Our main result is a comparison between the semirings $\K_+ \mu \VF[*]$ and $\K_+ \mu \RV[*]$ via an integration map.
Under effectivity we obtain that this map is an isomorphism.
This generalizes~\cite[Thm.\,8.28]{HK} from V-minimal theories to 1-h-minimal theories.

\begin{thm}[{{{Theorem~\ref{thm:isom.K.mu}}}}]\label{thm:main.measures}
Let $\cT$ be a $1$-h-minimal theory of equicharacteristic zero valued fields.
Then there is a surjective morphism of graded semirings
\[
\int\colon \K_+ \mu \VF[*]\to \K_+ \mu \RV[*] / I_\mathrm{sp}^\mu
\]
where $I_\mathrm{sp}^\mu = ([1]_1, [\RV^\times_{<1}]_1)$.
Moreover, if $\cT$ is effective then this map is an isomorphism.
\end{thm}

For applications it is often useful to have further maps from the $\RV$-sort down to the residue field and value group.
This makes it easier to construct additive invariants of definable sets, and may help in computing Grothendieck rings of valued fields.
In upcoming work we study definable sets in short exact sequences and use this to define Grothendieck rings for the residue field $\K_+ \RF$ and for the value group $\K_+ \VG$.
This leads to several new computations of Grothendieck rings of valued fields, and to a comparison of our framework with that of Cluckers--Loeser~\cite{CLoes}.
In particular we will show that Cluckers--Loeser integration is not universal: the map from (a suitable variant of) $\K_+ \mu \VF[*]$ to the Cluckers--Loeser ring of motivic volumes has a nontrivial kernel.

\subsection{Further avenues} 

There are several natural further questions related to our work.

Firstly, the framework of Cluckers--Loeser has been generalized and expanded to include a theory of motivic distributions, Fourier transforms, and Mellin transforms.
This generalization is crucial for applications in representation theory and the Langlands program.
Similarly, the Hrushovski--Kazhdan framework already includes a treatment of distributions and Fourier transforms.
This has led to a motivic Poisson formula in~\cite{HKPoiss}.
We expect that motivic distributions, Fourier transforms, and Mellin transforms can be developed in the generality of the current paper, and we hope to study this question in a future work. 
Additionally, it would be interesting to develop the analogue of the bounded integral of Forey--Yin~\cite{ForeyYin}, which interpolates between the non-measured and measured case.

Secondly, we only deal with the equicharacteristic zero case in this paper, and so it is a natural question to wonder what happens in mixed characteristic.
With the correct definitions and set-up, we expect that the analogue of Theorem~\ref{thm:main.not.eff}, and Theorem~\ref{thm:main.measures} without the moreover part, are not too difficult to prove in mixed characteristic.
However, it is not at all clear to us how to adapt effectivity to mixed characteristic.
For example, if $K$ is a henselian valued field of mixed characteristic $(0,p)$ with infinite residue field, then $K$ is automatically not effective.
The reason is simply that the inverse Frobenius on the residue field, which is a $\emptyset$-definable map, cannot lift definably to the valued field.
As such, it seems highly nontrivial to us to find the correct definitions to even state an analogue of Theorem~\ref{thm:main.eff}, or the moreover part of Theorem~\ref{thm:main.measures}.
See~\cite{CH:GrQp} for a positive result in this direction, where it is proven that the universal theory of integration in $\QQ_p$ is simply $p$-adic integration. 
Note however that $\QQ_p$ is effective, because of definable choice.
Another interesting setting to study this problem would be perfectoid fields.
This would give a way to ``tilt'' motivic measure, and may lead to higher-dimensional generalizations of the Fontaine--Wintenberger theorem~\cite{FonWin}.

Thirdly, our hope is that the notion of effectivity may be useful in other places when dealing with valued fields.
While 1-h-minimality already implies many tame geometric consequences, it says nothing about the structure on $\RV$.
Under effectivity, also the structure on $\RV$ is tame, with a resulting dimension theory of definable sets and a notion of differentiation of definable maps.
For example, we believe that effectivity is the correct ingredient to obtain definability of risometries, see \cite[Question~3.2.8]{B-WH25}.
More generally, it may be interesting to investigate how stratifications interact with Theorems~\ref{thm:main.eff} and~\ref{thm:main.measures}, and see if one can compute generalized Euler characteristics using stratifications as in~\cite{ComteHalupczok}.

\subsection{Outline}  
We briefly summarize the structure of this paper. 
Each section is essentially self-contained and builds only on the main theorems of the previous sections and the consequences of 1-h-minimality (only in Section~\ref{sec:measures} do we build on the arguments of Section~\ref{sec:int}).
In particular, to understand the proof of our main theorem in Section~\ref{sec:int} only a working knowledge of the consequences of 1-h-minimality, the definition of effectivity, and the main theorem on $\RV$-dimension of Section~\ref{sec:dim.RV} is required.

\begin{itemize}
	\item Section~\ref{sec:prelims} contains some preliminaries.
	We discuss semirings and congruences, our conventions for valued fields, h-minimality and its consequences, and $\acl$-dimension.
	
	\item In Section~\ref{sec:eff} we introduce the notion of effectivity, which is used to lift definable maps from $\RV$ to the valued field.
	We give a general strategy to show that structures are effective via quantifier elimination, and use this to show that several natural $1$-h-minimal theories are effective.
	
	\item In Section~\ref{sec:dim.RV} we prove that for effective theories $\acl$ satisfies the exchange principle on $\RV$.
	This gives a dimension theory for definable sets of $\RV$, and we show that we have $\exists^\infty$-elimination.
	We also give several examples of non-effective theories, using dimension theory.
	
	\item Then in Section~\ref{sec:int} we define the categories $\VF$ and $\RV[*]$ and study their Grothendieck semirings $\K_+ \VF$ and $\K_+ \RV[*]$.
	In particular, we prove Theorems~\ref{thm:main.not.eff} and~\ref{thm:main.eff}.
	We start by showing that in effective theories, there is a well-defined lifting map  $\LL\colon \Ob \RV[*]\to \Ob \VF$, which will determine an inverse to $\int$ on the level of Grothendieck rings, once we verify that the latter is well-defined.
	We then construct the integration map in ambient dimension one first, and then proceed inductively, using a Fubini theorem and a reduction to unary maps.
	
	\item Finally, in Section~\ref{sec:measures} we construct the categories $\mu\VF[*]$ and $\mu \RV[*]$ of objects with volume forms, and follow a similar strategy to relate them using the lifting map $\LL$.
	We also develop a notion of derivatives on $\RV$ for effective theories, and show that it agrees with the usual notion of derivatives in many natural $1$-h-minimal theories.
\end{itemize}

\subsection{Acknowledgements} 
We would like to thank Raf Cluckers for his constant encouragement and continued interest in this project, and for the many discussion surrounding the topics of this paper.
We are grateful to Pierre Touchard for introducing us to the model theory of short exact sequences, and to Silvain Rideau--Kikuchi for help with Proposition~\ref{prop:t.convex.effective} and for various discussions on Hrushovski--Kazhdan integration.
We have also benefited from discussions with Arthur Forey, Immanuel Halupczok, Martin Hils, Franziska Jahnke and Mariana Vicaria.
The author M.S. would like to thank the University of Münster for its hospitality during a research stay in which part of this work was carried out.

The author M.S.\ is supported by McMaster University and the Fields Institute.
The author F.V.\ is supported by the Humboldt foundation.

\section{Preliminaries}\label{sec:prelims}

\subsection{Semigroups and semirings}\label{ssec:semirings}

Recall that a \emph{semigroup} $G$ is a set with operation $+: G^2\to G$ and an element $0$ which satisfies the usual group axioms except that elements need not have inverses. All semigroups in this paper will be abelian. Similarly, a \emph{semiring} $R$ is a set with operations $+: R^2\to R, \cdot: R^2\to R$ and elements $0,1$ which satisfy the usual ring axioms except that not every element needs to have an inverse for $+$, together with the axiom that $0 \cdot x = 0 = x \cdot 0$ for each $x \in R$. Morphisms of semigroups and semirings are defined as usual. Throughout this paper, all semirings will be commutative with $1$. 

If $G$ is a semigroup then one can add additive inverses for every element of $G$ to obtain the \emph{group associated to $G$} or \emph{groupification of $G$}, which we denote by $G^a$. Similarly, if $R$ is a semiring we can add additive inverses to obtain $R^a$, the \emph{ring associated to $R$}. There are natural morphisms $i\colon G\to G^a$, $i\colon R\to R^a$ of semigroups resp.\ semirings, which are in general not injective. The map $i$ is injective if and only if $G$ (or $R$) is \emph{cancellative}, meaning that whenever $a+c=b+c$ then $a=b$. Most of the semigroups and semirings we construct in this paper will not be cancellative, and so one loses information when moving the associated group or ring. 

A \emph{congruence} $I$ of a semiring $R$ is an equivalence relation on $R$ which is a subsemiring of $R^2$. 
Equivalently, it is an equivalence relation such that for all $a,b,c\in R$ if $(a,b)\in I$ then also $(a+c,b+c)\in I$ and $(ac, bc)\in I$. 
Note that this description uses our running assumption that $R$ is commutative.
Congruences play the role of ideals in semirings. 
In particular, the semiring operations on $R$ descend to $R/I$, giving $R/I$ the structure of a semiring. 
The projection map $R\to R/I$ is a semiring morphism. The intersection of congruences is again a congruence. 
Therefore, if $S$ is a subset of $R^2$ the \emph{congruence generated by $S$} makes sense, and we denote it by $(S)$. 
If $a,b\in R$ then we denote the congruence generated by $(a,b)\in R^2$ simply by $(a,b)$.
This is the transitive closure of the subsemiring of $R^2$ generated by $(a,b)$ and $(b,a)$ (by the running assumption that all semirings are unital, $(1,1)$ is automatically included).
More concretely, for elements $x,y\in R$ this means that $x \equiv y \bmod I$ if and only if there exists a sequence $x = z_1, z_2, \ldots, z_n = y$ of elements in $R$ such that we can write for $i=1, \ldots, n-1$
\begin{align*}
z_i &= \sum_{j,k\in J} u_{ijk} a^j b^k, \\
z_{i+1} &= \sum_{j,k\in J} u_{ijk} b^j a^k,
\end{align*}
where $J\subset \NN^2$ and $u_{ijk}\in R$.


If $S\to R, R'$ are two semiring morphisms, then one can define the tensor product $R\otimes_S R'$, which is simply the pushout in the category of semirings. In particular, it satisfies the usual universal property. Equivalently, one may think of $R$ and $R'$ as $S$-semimodules, and then as an $S$-semimodule $R\otimes_S R'$ is also the tensor product in the category of $S$-semimodules. We refer to~\cite[Sec.\,8]{ComteHalupczok} for more details.

\begin{example}
Consider $\NN$ with the usual operations $+, \cdot$. This is a cancellative semiring whose associated ring is $\ZZ$. In this case, there is no real loss of information in working with $\ZZ$ rather than $\NN$.
\end{example}

\begin{example}\label{eg:trop}
Consider the tropical semiring $\NN_{\trop}$ which is $\NN\cup\{-\infty\}$ with operations $\max, +$. This is a semiring in which every element is idempotent. In particular, this semiring is not cancellative, and the associated ring is the zero ring. In many geometric theories this semiring appears naturally when looking at the dimension of definable sets. 
\end{example}

\subsection{Grothendieck semigroups and semirings} 

Given an essentially small category $\cC$, define the \emph{Grothendieck semigroup} $\K_+\cC$ of $\cC$ as the semigroup given by generators $[X]$ for $X \in \Ob \cC$, modulo the relations
\begin{enumerate}
	\item $[X] = [Y]$ if $X$ is isomorphic to $Y$, and
	\item $[X \sqcup Y] = [X] + [Y]$, where $X \sqcup Y$ denotes the coproduct of $X$ and $Y$ when it exists.
\end{enumerate}
By definition, any map from $\Ob \cC$ to a semigroup, respecting isomorphism classes and coproducts factors through a unique semigroup morphism from $\K_+\cC$.

The objects of our categories will consist of $\emptyset$-definable sets for a given language and theory, possibly decorated with some extra information such as volume forms and equipped appropriate notions of morphisms. 
In all of our cases, the coproduct will be induced by the disjoint union.
Actually, our categories will often be closed under coproducts (and contain the empty set), whence the underlying set of $\K_+\cC$ simply consists of all isomorphism classes in $\cC$.

Given such a category $\cC$, we will often consider the full subcategories $\cC[n]$, consisting of all objects living in ambient dimension $n$.
The corresponding family of semigroups $\K_+\cC[n]$ for each $n \in \NN$ can be investigated through the \emph{graded Grothendieck semigroup} $\K_+\cC[*] = \bigoplus_{n \geq 0} \K_+\cC[n]$.
Objects in $\K_+\cC[*]$ are represented as finite sums of classes $[X]_n$ where the index $n$ emphasizes that $[X]_n$ belongs to $\K_+\cC[n]$.

In our settings, the cartesian product (or a natural operation induced by the cartesian product in the case of Section~\ref{sec:measures}) will induce a well-defined product on $\K_+\cC$ and $\K_+\cC[*]$, given respectively by
\begin{itemize}
	\item $[X] \cdot [Y] = [X \times Y]$, and
	\item $[X]_n \cdot [Y]_m = [X \times Y]_{m + n}$.
\end{itemize}
In all our use cases, this product will endow $\K_+\cC$ (resp. $\K_+\cC[*]$) with the structure of a (graded) commutative semiring.

Finally, one can consider the groupification of $\K_+\cC$ an  $\K_+\cC[*]$ to obtain the (graded) \emph{Grothendieck groups} $\K\cC$ and $\K\cC[*]$.
If $\K_+\cC$ or $\K_+\cC[*]$ was endowed with a (graded) semiring structure, $\K\cC$, resp. $\K\cC[*]$ is a (graded) \emph{Grothendieck rings}.

\begin{example}
Let $\cR$ be an o-minimal field. By~\cite[p.\,132]{vdD98}, two $\emptyset$-definable sets are $\emptyset$-definably isomorphic if and only if they have the same dimension and Euler characteristic. 
Let $R\subset \NN_{\trop}\times \ZZ$ be the subsemiring of positive elements
\[
R = \{(a,b)\in \NN\times \ZZ\mid b\geq 0 \text{ if } a = 0\} \cup\{(-\infty, 0)\}.
\]
Then we obtain an isomorphism
\[
\K_+ \cR \to R\colon [X]\mapsto (\dim X, \operatorname{Eu} X),
\]
where $\operatorname{Eu} X$ is the Euler characteristic of $X$. 
When taking the associated Grothendieck ring the tropical part of $R$ disappears, and we obtain $\K \cR = \ZZ$ where the isomorphism is given by the Euler characteristic. 
\end{example}

\begin{example}\label{ex:padic.semiring}
Let $L/\QQ_p$ be a $p$-adic field, considered in the language $\cL_\ring(L)$.
By a theorem of Cluckers--Haskell~\cite{CH01}, the Grothendieck ring of $L$ is trivial.
Even more, by~\cite[Cor.\,3]{cluckers.padic} two infinite definable sets $X,Y\subset L^n$ are in definable bijection if and only if $\dim X = \dim Y$.
Let us describe the Grothendieck semiring $\K_+ L$.
Define $\Delta = \NN \sqcup \{\delta_i\mid i\in \NN_{\geq 1}\}$ with the usual addition and multiplication for elements of $\NN$, the tropical semiring on the $\delta_i$, and where the $\delta_i$ absorb $\NN$.
In more detail, this means that $\delta_i + \delta_j = \delta_{\max\{i,j\}}, \delta_i\delta_j = \delta_{i+j}$, for $n\in \NN\subset \Delta$ we have $n+\delta_i = \delta_i$ and if $n\neq 0$ then $n\delta_i = \delta_i$.
Then the map $\K_+ L\to \Delta$ which maps $[X]$ to $\# X$ if $X$ is finite and to $\delta_{\dim X}$ if $X$ is infinite is an isomorphism.
\end{example}

\subsection{Basic notation in valued fields} \label{sec:basic_notation}

Let $K$ be a valued field with valuation ring $\cO_K$. We denote the valuation multiplicatively by $|\cdot |\colon K^\times\to \Gamma_K^\times$ and denote $\Gamma_K = \Gamma_K^\times \cup \{0\}$. The maximal ideal is denoted by $\cM_K$, and the residue field by $k$. If $\lambda\in \Gamma_K^\times$ and $a\in K$, then $B_{<\lambda}(a)$ and $B_{\leq \lambda}(a)$ denote the open resp.\ closed ball of radius $\lambda$ around $a$. 

We define the \emph{leading term structure} as $\RV = K^\times / (1+\cM_K)\cup \{0\}$, and the natural quotient map is denoted by $\rv\colon K\to \RV$ (extended by $\rv(0) = 0$). We let $\RV^\times = \RV\setminus \{0\}$. The structure $\RV$ combines information from the residue field and the value group via the short exact sequence
\[
1\to k^\times \to \RV^\times \to \Gamma_K^\times \to 1.
\]
Note that $\rv(x) = \rv(y)$ if and only if $|x-y|<|x|$ or $x=y=0$.

The fibres of the map $K^\times \to \RV^\times\colon x\mapsto \rv(x-c)$ are precisely the maximal open balls not containing $c$. We will call such a ball \emph{1-next to $c$}. If $C\subset K$ is a finite set, then a \emph{ball $1$-next to $C$} is an open ball of the form $\cap_{c\in C}B_c$, where each $B_c$ is a ball $1$-next to $c\in C$. Equivalently, the balls $1$-next to $C$ are the maximal open balls disjoint from $C$. 

A set $X\subset K$ is said to be \emph{$1$-prepared} by a finite set $C\subset K$ if every ball $1$-next to $C$ is either contained in $X$, or disjoint from $X$. The reader should think about this finite set $C$ as controlling $X$. 

We will also need some higher depth variants of these definitions. If $\lambda\in \Gamma_K^\times, \lambda\leq 1$, then we define 
\[
\RV_\lambda = \frac{K^\times}{1+B_{<\lambda}(0)} \cup \{0\},
\]
with the natural map $\rv_\lambda\colon K\to \RV_\lambda$. Note that $\RV_1 = \RV$ and $\rv_1 = \rv$. There are natural surjections $\RV_\lambda\to \RV_\mu$ whenever $\lambda \leq \mu$. Intuitively, when $\lambda$ becomes smaller, the structure $\RV_\lambda$ sees more of the valued field $K$. For example, if $K = k((t^\QQ))$, then (using additive notation for $\Gamma_K$)
\[
\rv_\lambda\left(\sum_{i= i_0}^\infty a_i t^i\right) = \sum_{i=i_0}^{i_0+\lambda} a_i t^i, \quad \text{ if } a_{i_0}\neq 0.
\]

We denote by $\oplus\subset \RV^3$ the partial addition on $\RV$, which is a $\emptyset$-definable ternary relation given by
\[
\oplus(\alpha, \beta, \gamma) \text{ iff } \exists a,b,c\in K\colon \rv(a)=\alpha\wedge \rv(b)=\beta \wedge \rv(c)=\gamma \wedge a+b=c.
\]
For $\alpha, \beta\in \RV$ we write $\alpha\oplus \beta$ for the set of $\gamma\in \RV$ for which $\oplus(\alpha, \beta, \gamma)$.
If $0\notin \alpha\oplus \beta$ then $\alpha\oplus \beta$ consists of a singleton, and we also use $\alpha\oplus \beta$ to refer to this element.

\subsection{Some model-theoretic conventions}

Throughout, we will consider a non-trivially valued field $K$ of characteristic zero equipped with an $\cL$-structure for some language $\cL$ expanding the language of valued fields $\cL_{\val} = \{0,1,+,-,\cdot,\cO_K\}$. 
As is usual in applications of Hensel minimality, the precise language $\cL$ is not important, only which sets are definable and over what parameters. By ``$\emptyset$-definable'', we mean $\cL$-definable (without parameters). 
We say that a set is $A$-definable for some collection of parameters $A$ if it is $\cL(A)$-definable. A set is called definable, it is $A$-definable for some parameter set $A$.

The language $\cL$ itself may be one-sorted or multi-sorted, including e.g.\ sorts for $k,\Gamma$ and $\RV$, equipped with the induced structure, along with the usual maps from $K$ to these sorts. 
Even if $\cL$ is one-sorted, we can still work with elements and subsets of arbitrary imaginary sorts, by either moving to the $\cL^\mathrm{eq}$-structure $K^\mathrm{eq}$ or by simulating them in $K$, as explained in \cite[Sec.\,2.2]{CHR}. 
For example: a $\emptyset$-definable subset of $\RV^n$ can be identified with a $\emptyset$-definable subset of $K^n$ which is a union of fibres of $\rv \colon K^n \to \RV^n$.

Given any (potentially imaginary) set of parameters $A$, a (potentially imaginary) element $b$ is said to be in the algebraic closure of $A$ if there exists a finite $A$-definable set containing $b$. 
We write e.g. $\acl(A) \cap K$ or $\acl(A) \cap \RV$ to denote the elements of respectively $K$ and $\RV$ that are in the algebraic closure of $A$. 
We use similar notation for other imaginary sorts such as $k$ or $\Gamma$, and also for the definable closure $\dcl(A)$.

\subsection{Hensel minimality and cell decomposition}

We recall some basic results about $1$-h-minimality. Throughout, let $\cT$ be a $1$-h-minimal theory of equicharacteristic zero valued fields, and let $K$ be a model of $\cT$. We refer to~\cite{CHR} for the precise definition and recall some of the consequences that are useful to us. 
The most fundamental property is that we can uniformly prepare definable sets.

\begin{prop}[{{{Uniform preparation~\cite[Cor.\,2.6.6]{CHR}}}}]
Let $A \subset K$ be a parameter and let $W\subset K\times \RV^k$ be $A \cup \RV$-definable. Then there exists a finite $A$-definable set $C\subset K$ such that for every $\xi\in \RV^k$, $C$ $1$-prepares the set $W_\xi\subset K$.
\end{prop}

This implies for example that $\RV$-unions stay finite.
That is, if $(W_\xi)_{\xi\in \RV^m}$ is a $A$-definable family of finite subsets $W_\xi\subset K$, then $\cup_{\xi\in \RV^m} W_\xi$ is still finite.
This will frequently be used implicitly.

%
%

An important tool in Hensel minimal fields is cell decomposition. 
For many applications of 1-h-minimality, one can assume the existence of algebraic Skolem functions.
That is, one can assume that every finite-to-one $\emptyset$-definable function has a $\emptyset$-definable section. 
This leads to a simpler notion of cell decomposition as in \cite[Thm.\,5.2.4]{CHR}, compared to the more classical version in \cite[Thm.\,5.7.3]{CHR}. 
We will rather use the latter version, as the former requires expanding our structure by new definable sets.
See also \cite{Verm:h-min} for a similar treatment of cell decomposition.

For $m\leq n$ denote by $\pi_{\leq m}$ the projection $K^n\to K^m$ onto the first $m$ coordinates, and by $\pi_m$ the projection onto the $m$-th coordinate. Similarly define $\pi_{<m} \coloneqq \pi_{\leq m-1}$ for $1 < m \leq n$.

\begin{defn}[Twisted boxes]
Let $A\subset K\cup \RV$ be a parameter set, and let $X\subset K^n$ be $A$-definable. Then $X$ is called an $A$-definable \emph{twisted box} if there exists an element $r\in \RV^n$, and $A$-definable maps $c_i\colon \pi_{<i}(X)\to K$ for $i=1, \ldots, n$ such that
\[
X = \{x\in K^n\mid (\rv(x_i-c_i(\pi_{<i}(x))) )_{i=1, \ldots, n} = r\}.
\]
The tuple $c = (c_i)_i$ is called the \emph{centre tuple of $X$}. We use the notation $\rtimes(c,r)$ to denote the twisted box $X$ as defined here.
\end{defn}

\begin{defn}[Cell decomposition]
Let $A\subset K\cup \RV$ be a parameter set. Then an \emph{$A$-definable cell decomposition of $K^n$} is an $A$-definable map $\chi\colon K^n\to \RV^k$ (for some $k$) such that every non-empty fibre $\chi^{-1}(\xi)$ of $\chi$ is an $A\xi$-definable twisted box whose centre tuple moreover depends definably on $\xi$. By a \emph{twisted box of $\chi$} we simply mean a fibre of $\chi$. 
\end{defn}

It is important to note that an $A$-definable cell decomposition $\chi\colon K^n\to \RV^k$ consists of both the data of $\chi$, and the definable data of all the cell centres. That is, if we write $\chi^{-1}(\xi) = \rtimes (c(\xi), r(\xi))$ then the tuple of functions $c(\xi) = (c_1(\xi), \ldots, c_n(\xi))$ are part of the data of the cell decomposition. Nevertheless, we typically suppress the cell centres from the notation and simply say that $\chi\colon K^n\to \RV^k$ is a cell decomposition. Note that the value $r(\xi)$ depends definably on $\xi$, and that may be recovered from $\xi, \chi$ and the centre tuple $c(\xi)$.

We define $0\cdot \RV^\times = \{0\}\subset \RV$ and $1\cdot \RV^\times = \RV^\times$.

\begin{defn}[{Cells}] \label{def:cells}
Let $A \subset K \cup \RV$ be a parameter set and $\chi \colon K^n \to \RV^m$ an $A$-definable cell decomposition. A twisted box $\rtimes(c(\xi),r(\xi))$ of $\chi$ is said to have type $(i_1,\dots,i_n) \in \{0,1\}^n$ if $r_{j}(\xi) \in i_j \cdot \RV^{\times}$ for $j = 1,\dots,n$. An $A$-definable \emph{cell} of type $(i_1,\dots,i_n)$ of $\chi$ is an $A$-definable set $X \subset K^n$ which is a union of twisted boxes of $\chi$, each of type $(i_1,\dots,i_n)$. 
\end{defn}

\begin{thm}[{{{Cell decomposition~\cite[Thm.\,5.7.3]{CHR} }}}]\label{thm:cell.decomp}
Let $X\subset K^n$ be $A$-definable, for some $A\subset K\cup \RV$. Then there exists an $A$-definable cell decomposition $\chi\colon K^n\to \RV^N$ such that $X$ is a union of twisted boxes of $\chi$. We say that \emph{$\chi$ is a cell decomposition of $X$}. Moreover, we may assume that any of the following properties hold for $\chi$:
\begin{enumerate}
	\item (Preparation of functions) If an $A$-definable map $\psi\colon X\to \RV^m$ is given, then we may assume that $\psi$ is constant on the twisted boxes of $\chi$. 
	\item (Uniform cell decomposition) \label{it:adapted} If an $A$-definable set $Y \subset K^n \times \RV^m$ is given, then we may assume that, for each $\xi \in \RV^m$, $\chi$ is a cell decomposition for the fibre $Y_{\xi} \subset K$. 
	\item (Continuous centers) We may assume that all functions in each centre tuple of $\chi$ are continuous. 
	\item (Compatible domain-image preparation) Assume $n = 1$ and that we are given an $A$-definable map $f \colon K \to K$. Then then we can arrange that there exists an $A$-definable cell decomposition $\psi \colon K \to \RV^M$ such that the image under $f$ of each twisted box of $\chi$ is a twisted box of $\psi$.
\end{enumerate}
\end{thm}
If $Y,\chi$ are as in Theorem~\ref{thm:cell.decomp}(\ref{it:adapted}) above, we say that $\chi$ is \emph{adapted} to $Y$, or simply that $\chi$ is a cell decomposition for $Y$.

Cell decomposition also leads to a good notion of dimension for definable sets.

\begin{defn}
Let $X\subset K^n$ be $\emptyset$-definable. Then we say that $X$ has \emph{dimension at most $k$} if there exists a finite-to-one $\emptyset$-definable map $f\colon X\to K^k$. The dimension of $X$ is denoted by $\dim X$. By convention $\dim \emptyset = -\infty$.
\end{defn}

We have the following results on dimension theory.

\begin{theorem}[{{{Dimension theory~\cite[Thm.\,5.3.4]{CHR} }}}]\label{thm:dim.theory.VF}
Let $X,Y\subset K^n, Z\subset K^m$ be $\emptyset$-definable and let $f \colon X\to Z$ be $\emptyset$-definable. Then the following hold.
\begin{enumerate}
\item (Linear maps) The dimension of $X$ is the largest integer $k$ such that there exists a $K$-linear map $\ell\colon K^n\to K^k$ such that $\ell(X)$ has non-empty interior.
\item \label{it:coordinate.projections.VF} (Coordinate projections) The dimension of $X$ is the largest integer $k$ such that there exists a coordinate projection $\pi\colon K^n\to K^k$ such that $\pi(X)$ has non-empty interior. In particular, dimension is independent of the language and of the model. Also, $\dim X = 0$ if and only if $X$ is finite.
\item (Unions) The dimension of $X\cup Y$ is $\max\{\dim X, \dim Y\}$.
\item (Definability) For $d\leq n$ the set of $z\in Z$ such that $\dim f^{-1}(z)=d$ is $\emptyset$-definable.
\item (Fibrations) If all fibres of $f$ have dimension $d$ then $\dim X = d+\dim Z$.
\item (Local dimension) There exists an $x\in X$ such that for every open ball $B$ around $x$ we have $\dim X = \dim (X\cap B)$.
\end{enumerate}
\end{theorem}
In particular, it follows from e.g. item (\ref{it:coordinate.projections.VF}) that if $X$ is a cell or twisted box of type $(i_1,\dots,i_n)$, then $X$ has dimension $k = i_1 + \dots + i_n$. In fact, for each $x \in X$ and open ball $B$ around $x$ it holds that $\dim(X \cap B) = k$.

\subsection{Differentiation and the Jacobian property}

As usual, a map $f\colon U\subset K^n\to K$ on an open $U$ is said to be $C^1$ at $x\in U$ if there exists $Df(x)\in K^n$ such that 
\[
\lim_{y\to x}\frac{|f(y)-f(x)-Df(x)\cdot (y-x)|}{|y-x|} = 0.
\]
For $f\colon U\subset K^n\to K^n$ $C^1$ on $U$, we denote by $Df\colon U\to K^{n\times n}$ its \emph{total derivative} and by $\Jac f\colon U\to K$ its \emph{Jacobian} $\Jac f = \det (Df)$.

A key theorem of~\cite{CHR} is that the Jacobian property holds for definable functions.

\begin{theorem}[{{{Jacobian property~\cite{CHR}}}}]
Let $f\colon K\to K$ be a $\emptyset$-definable function. Then there exists a finite $\emptyset$-definable set $C\subset K$ such that for every ball $B$ $1$-next to $C$ the following hold:
\begin{enumerate}
\item $f$ is $C^1$ on $B$ and $\rv(f')$ is constant,
\item \label{it:RVJac.rvfprime} for $x,y\in B$ we have
\[
\rv(f(x)-f(y)) = \rv(f'(x)) \rv(x-y), \text{ and}
\]
\item for every open ball $B'\subset B$, the image $f(B')$ is an open ball of radius $|f'(x)|\radop B'$ (for any $x\in B$).
\end{enumerate}
\end{theorem}
Note that item~(\ref{it:RVJac.rvfprime}) above implies in particular that
\[ \abs{f(x)-f(y)} = \abs{f'(x)} \abs{x-y}.  \]

\begin{theorem}[{{{$C^1$ cell decomposition~\cite{CHR}}}}]
Let $f\colon K^n\to K$ be $\emptyset$-definable. Then there exists a cell decomposition of $K^n$ such that every $n$-dimensional twisted box is open and $f$ is $C^1$ on every such twisted box.
\end{theorem}

For Fubini theorems, we need to be able to switch coordinates for a twisted box in $K^2$. This is enabled by the following lemma.

\begin{lem}[{{{\cite[Lem.\,5.4.12]{CHR}}}}]\label{lem:bi.twisted}
Let $A\subset K\cup \RV$ and let $X\subset K^2$ be an $A$-definable twisted box, say of the form
\[
X = \{(x,y)\in K^2\mid \rv(x-c_1)=\xi_1, \rv(y-c_2(x_1)) = \xi_2\}.
\]
Assume that $c_2$ has the Jacobian property. Then either $X$ is a box, or
\[
X = \{(x,y)\in K^2\mid y\in Y, \rv(x-c_2^{-1}(y)) = \xi_3\},
\]
where $Y$ is the projection of $X$ onto the second coordinate, and $\xi_3 = -\rv(c_2')^{-1}\xi_2$ is $\dcl(A) \cap \RV$-definable.
\end{lem}

We will call a twisted box $X\subset K^2$ as in this lemma, i.e.\ whose cell centre satisfies the Jacobian property, a \emph{bi-twisted box}.

\begin{lem}\label{lem:cell.decomp.bi.twisted}
Let $n\geq 2$ and let $X\subset K^n$ be $A$-definable, for $A\subset K\cup \RV$. Then there exists a cell decomposition $\chi\colon X\to \RV^N$ of $X$ such that for every $\eta\in \RV^N$ and every $a\in K^{n-2}$ the set
\[
\{x\in K^2\mid \chi(a,x) = \eta\}
\] 
is an $Aa\xi$-definable bi-twisted box.
\end{lem}

\begin{proof}
Let $\chi_0\colon K^n\to \RV^M$ be a cell decomposition of $X$. Now simply refine the first $n-1$ coordinates such that for every $a\in K^{n-2}$ and every twisted box $F$ of $\chi_0$ with cell centre $(c_1, \ldots, c_n)$ we have that the map 
\[
(a\times K)\cap \pi_{<n}(F) \to K\colon x_{n-1}\mapsto c_n(a,x_{n-1})
\]
satisfies the Jacobian property. The resulting cell decomposition $\chi$ is as desired.
\end{proof}

\subsection{On $\acl$-dimension}\label{ssec:acl.dim.prelims}

We recall how the exchange principle for $\acl$ leads to a dimension theory for definable sets, see~\cite{HrushovskiPillay, Gagelman, JY23, vdD89} for more information.
Let $\cT$ be any complete $\cL$-theory for some language $\cL$.
We say that $\cT$ is \emph{pregeometric} if $\acl$ satisfies the exchange principle, i.e.\ for every model $M$ of $\cT$ and for all $a,b\in M, A\subset M$, if $a\in \acl(Ab)\setminus \acl(A)$ then $b\in \acl(Aa)$.
If moreover $\cT$ has $\exists^\infty$-elimination then we call $\cT$ \emph{geometric}.
Recall that this means that for every $\emptyset$-definable map $f\colon X\to Y$ between $\emptyset$-definable sets, the set 
\[
\{y\in Y\mid f^{-1}(y)\text{ is infinite}\}
\]
is $\emptyset$-definable in any model of $\cT$.
By compactness, this is equivalent to saying that there exists some integer $N$ such that if $f^{-1}(y)$ is finite then $|f^{-1}(y)|\leq N$.

If $\cT$ is multi-sorted and $S$ is a sort of $\cT$, or more generally just an imaginary $\emptyset$-definable set, we will call $\cT$ \emph{pregeometric on $S$} if $\acl_S$, the restriction of the algebraic closure operator to $S$, satisfies exchange. 
If we moreover have $\exists^\infty$-elimination in $S$, then we say that $\cT$ is \emph{geometric on $S$}.
For example, this allows us to speak about $\RV$ being geometric in theories of valued fields.

\begin{defn}
Let $\cT$ be pregeometric, let $M$ be a sufficiently saturated model of $\cT$ and let $X\subset M^n$ be a $\emptyset$-definable set.
Then the \emph{$\acl$-dimension} of $X$ is the largest integer $m$ such that there exist $i_1, \ldots, i_m\in \{1, \ldots, n\}, i_1< \ldots < i_m$ and an element $a=(a_1, \ldots, a_n)\in X$ for which
\[
a_{i_j}\notin \acl(a_{i_1}, \ldots, a_{i_{j-1}}), j = 1, \ldots, m.
\]
We denote the dimension of $X$ by $\dim X$ and call $a$ a \emph{generic point of $X$}.
\end{defn}

We recall some basic properties of this dimension function.

\begin{prop}[$\acl$-dimension theory]\label{prop:acl.dim}
Let $\cT$ be pregeometric and let $M$ be a sufficiently saturated model of $\cT$.
Let $X\subset M^n, Y\subset M^m$ and $f\colon X\to Y$ be $\emptyset$-definable.
Then the following hold. 
\begin{enumerate}
\item \label{it:acl.fin} (Finiteness) $X$ is infinite if and only if $\dim X > 0$.
\item \label{it:acl.union} (Unions) If $m=n$ then $\dim (X\cup Y)  = \max\{\dim X, \dim Y\}$.
\item \label{it:acl.subset} (Inclusion) If $m=n$ and $X\subset Y$ then $\dim X\leq \dim Y$.
\item \label{it:acl.ind} (Independence of parameters) If $A\subset M$ is a parameter set such that $M$ is at least $\abs{A}^+$-saturated, then the dimension of $X$ in the language $\cL(A)$ is still equal to $\dim X$.
\item \label{it:acl.map} (Definable maps) If $f$ is injective then $\dim X\leq \dim Y$. If $f$ is surjective then $\dim X \geq \dim Y$.
\end{enumerate}
\end{prop}

Note that by property (\ref{it:acl.ind}), the notion of $\acl$-dimension extends to all definable sets (with parameters).

\begin{proof}
Property (\ref{it:acl.fin}) follows from compactness, while (\ref{it:acl.subset}) and (\ref{it:acl.map}) are clear.
Properties (\ref{it:acl.union}) and (\ref{it:acl.ind}) are contained in~\cite[Lem.\,2.3]{HrushovskiPillay}.
\end{proof}

We also recall work by van den Dries~\cite{vdD89} on dimensions of definable sets.
Let $M$ be a sufficiently saturated model of $\cT$.
Denote by $\cS_n$ the set of definable subsets of $M^n$, with parameters.
A \emph{dimension function} is a function $d\colon \cup_n \cS_n\to \NN\cup\{-\infty\}$ satisfying the following:
\begin{itemize}
\item[(Dim 1)] $d(S) = -\infty $ if and only if $S = \emptyset$, $d(\{x\}) = 0$ for every $x\in M$, and $d(M) = 1$.
\item[(Dim 2)] $d(S_1\cup S_2) = \max\{d(S_1), d(S_2)\}$.
\item[(Dim 3)] For every coordinate permutation $\sigma\colon M^n\to M^n$ we have $d(S) = d(\sigma(S))$.
\item[(Dim 4)] Let $T\subset M^{n+1}$ be definable and for $x\in M^n$ let $T_x = \{y\in M\mid (x,y)\in T\}$. 
For $i\in \{0,1\}$ let $T(i) = \{x\in M^m\mid d(T_x) = i\}$. 
Then $T(i)$ is definable and 
\[
d(\{(x,y)\in T\mid x\in T(i)\}) = d(T(i)) + i.
\]
\end{itemize}

\begin{lem}\label{lem:geom.dim.function}
Assume that $\cT$ is geometric.
Then the $\acl$-dimension defines a dimension function in the sense of~\cite{vdD89}.
\end{lem}

\begin{proof}
Properties (Dim 1), (Dim 2) and (Dim 3) from~\cite[p.\,189]{vdD89} are clear, while (Dim 4) follows from $\exists^\infty$-elimination.
\end{proof}

If $S$ is an imaginary $\emptyset$-definable set, then Proposition~\ref{prop:acl.dim} and Lemma~\ref{lem:geom.dim.function} generalize directly to the setting where $\cT$ is (pre)geometric on $S$.

Let us also mention that if $K$ is a $1$-h-minimal field, then the structure on $K$ is in fact geometric and the resulting notion of dimension agrees with the dimension theory from Theorem~\ref{thm:dim.theory.VF}.
Indeed, this follows from~\cite[Lem.\,5.3.6]{CHR}.

\section{Effective theories}\label{sec:eff}

Throughout this section we will work in a complete theory $\cT$ of equicharacteristic zero valued fields in a language $\cL$ expanding the language of valued fields $\cL_{\val} = \{0,1,+,-,\cdot,\cO_K\}$. 
We will typically work in a fixed model $K$ of $\cT$ and assume that $K$ is sufficiently saturated.

\subsection{The definition and basic properties}

In this section we introduce a notion called effectivity, which will be needed to lift bijections between objects in $\RV^n$ to their pre-images in $K^n$. 
This notion is similar to the one introduced by Hrushovski--Kazhdan~\cite[Sec.\,3.4]{HK}, but allows parameters to make use of compactness. 

\begin{defn} \label{def:effectivity}
%
%
Let $\cT$ be a theory of equicharacteristic zero valued fields in a language $\cL$ expanding the language of valued fields. Then $\cT$ is \emph{$\emptyset$-effective} if for every model $K$ of $\cT$, and every $\emptyset$-definable element $\xi \in \RV$, there exists a $\emptyset$-definable element $x\in K$ with $\rv(x) = \xi$.

Call $\cT$ \emph{effective} if for every model $K$ of $\cT$ and every $A\subset K$, the theory $\Th_{\cL(A)}(K)$ is $\emptyset$-effective.
\end{defn}

\begin{remark}
	We warn readers that our notion of $\emptyset$-effectivity is what Hrushovski--Kazhdan call effectivity. 
	However, for $V$-minimal theories, $\emptyset$-effectivity implies effectivity in our sense (combine Lemma~\ref{lem:lift.fin} with \cite[Lem.\,3.19, 6.2]{HK}), so there is no real risk of confusion.
	The fact that $\emptyset$-effective implies effective is specific to $V$-minimal theories, and does not hold in general.
	For example, $\CC((t))$ in the language $\cL_\val\cup \{\ac, t\}$ is $\emptyset$-effective but not effective, as will be explained in Example~\ref{eg:ac.not.eff}.
\end{remark}

The parameters from $A$ allow us to use compactness and apply effectivity in families.
We first show that we can lift finite sets from $\RV$ to $\VF$.

\begin{lem}[Lifting finite sets] \label{lem:lift.fin}
Let $\cT$ be a theory of equicharacteristic zero valued fields.
Then $\cT$ is effective if and only if for every model $K$ of $\cT$, every $A\subset K$, and every finite $A$-definable set $R\subset \RV$, there exists a finite $A$-definable set $X\subset K$ such that $X\subset \rv^{-1}(R)$ and $\rv\colon X\to R$ is a bijection. 
\end{lem}


\begin{proof}
If this condition holds then clearly $\cT$ is effective.

So assume that $\cT$ is effective, let $K$ be a model of $\cT$ and let $A\subset K$ be a parameter set.
Let $R\subset \RV$ be a finite $A$-definable set, say with $n$ elements.
Since we have finite definable choice in $\Gamma$, we may assume that the valuation $|R|\subset \Gamma$ consists of a single element $\lambda$.

Recall the partial addition $\oplus$ on $\RV$ from Section~\ref{sec:basic_notation}, and consider the following associated $\emptyset$-definable map
\[
\oplus'\colon \{(\zeta, \eta)\in \RV^2\mid |\zeta| = |\eta|\}\to \RV\colon (\zeta, \eta)\mapsto \begin{cases}
0 & \text{ if } 0\in \zeta\oplus \eta,\\
\zeta\oplus \eta & \text{ if } 0\notin \zeta\oplus \eta.
\end{cases}
\]
Note that if $\zeta\oplus'\eta\neq 0$, then $|\zeta\oplus'\eta| = |\zeta| = |\eta|$.
We extend the definition of $\oplus'$ by putting $\zeta\oplus' 0 = 0\oplus' \zeta = \zeta$ whenever $\zeta\in \RV$.
For $\gamma\in \Gamma$ denote $\RV_{=\gamma} = \{\zeta\in \RV\mid |\zeta| = \gamma\}\cup\{0\}$.
Then $\oplus'$ defines an associative addition on $\RV_{=\gamma}$.

Let $f_1, \ldots, f_n\in \ZZ[x_1, \ldots, x_n]$ denote the $n$ symmetric polynomials in $x_1, \ldots, x_n$, where $f_i$ has degree $i$.
By replacing the addition by the operator $\oplus'$, each $f_i$ determines a $\emptyset$-definable map $f_i\colon \RV_{=\lambda}^n\to \RV_{=\lambda^i}$.
Moreover, if $\sigma\in S_n$ is any permutation and $\zeta_1, \ldots, \zeta_n\in \RV_{=\lambda}$ then $f_i(\zeta_1, \ldots, \zeta_n) = f_i(\zeta_{\sigma(1)}, \ldots, \zeta_{\sigma(n)})$.
Hence if $R = \{\xi_1, \ldots, \xi_n\}$, the elements $\sigma_i = f_i(\xi_1, \ldots, \xi_n)$ are $A$-definable.
By effectivity there exist $A$-definable elements $b_i\in K$ with $\rv(b_i) = \sigma_i$.
Consider the polynomial $g = x^n - b_1x^{n-1} + \ldots \pm b_n\in K[x]$.
Since the $\xi_i$ are all distinct, henselianity ensures that $g$ has $n$ distinct roots in $K$.
Thus the set $X$ of roots of $g$ is an $A$-definable set lifting $R$.
\end{proof}

There are several simple reformulations of effectivity. 

\begin{lem} \label{le:effective.characterizations}
Let $\cT$ be a $1$-h-minimal theory of equicharacteristic zero valued fields. Then the following are equivalent:
\begin{enumerate}
\item\label{it:effective} (Effectivity) The theory $\cT$ is effective, 
\item\label{it:picking.points.in.balls} (Picking points in balls) For every model $K$ of $\cT$, for every set $A\subset K$, and every $A$-definable set which is a finite union of pairwise disjoint open and closed balls, there exists a finite $A$-definable set containing exactly one point from each of these open and closed balls,
\item\label{it:acl.rv.commute} ($\acl$ and $\rv$ commute) For every model $K$ of $\cT$ and for every subset $A\subset K$ we have that $\rv(\acl(A)\cap K) = \acl(A) \cap \RV$.
\item\label{it:lifting.families} (Lifting families) For every $\aleph_0$-saturated model $K$ of $\cT$, for every subset $A\subset K$, and for every $A$-definable family $(R_\xi)_{\xi\in \RV^n}$ of finite subsets in $\RV^m$, there exists an $A$-definable family $(C_a)_{a\in K^n}$ of finite sets in $K^m$ such that for every $a\in K^n$, $\rv$ maps $C_a$ bijectively to $R_{\rv(a)}$. 
\end{enumerate}
\end{lem}

\begin{proof}
(\ref{it:effective} $\Rightarrow$ \ref{it:picking.points.in.balls}) Assume that $\cT$ is effective, we show that one can pick points in balls. Let $X$ be an $A$-definable subset of $K$ which is a finite union of open and closed balls, say $X = \cup_i B_i$. If the value group is discrete, then we may moreover assume that all $B_i$ are open balls. 

Let $C$ be a finite $A$-definable set $1$-preparing $X$. For each $c\in C\setminus X$, we use effectivity to find a finite $Ac$-definable set $D_c$ which contains exactly one point above each element of $\rv(X - c)$. Note that this is indeed possible since $\rv(X-c)$ is finite since $c\notin X$. Let $C_1 = X\cap (C\cup \bigcup_c D_c)$, which is a finite $A$-definable set. Then $C_1$ is contained in $X$, and contains a point in every ball $B_i$. Indeed, if the value group is non-discrete and $B_i$ is a closed ball then $C$ already contains a point from $B_i$. If $B_i$ is an open ball which does not already contain a point from $C$, then it must contain a point from $\bigcup_c D_c$.

To ensure that $C_1$ contains exactly one point in each $B_i$, we average. In more detail, there is an $A$-definable equivalence relation on $C_1$ for which $x\sim y$ if and only if $x$ and $y$ are contained in the same ball $B_i$. Let $C_2$ consist of the averages of each equivalence class. Then $C_2$ is a finite $A$-definable set which contains exactly one point in each ball $B_i$. Note that this uses that $K$ is of equicharacteristic zero.

(\ref{it:picking.points.in.balls} $\Rightarrow$ \ref{it:effective}) For the converse, it is clear that if one can pick points in balls, then $\cT$ is effective. Indeed, this follows from the fact that a pullback of a finite subset of $\RV$ is a finite union of open balls.

(\ref{it:effective} $\Rightarrow$ \ref{it:acl.rv.commute}) It is clear that $\rv(\acl(A) \cap K)\subset \acl(A)\cap \RV$. So let $\xi$ be an element of $\acl(A)\cap \RV$. Then there exists a finite $A$-definable set $R\subset \RV$ containing $\xi$. By effectivity there exists a finite $A$-definable set $C\subset K$ such that $\rv(C)\subset R$ and $\rv\colon C\to R$ is a bijection. Hence $\xi\in \rv(C)\subset \rv(\acl(A) \cap K)$.

(\ref{it:acl.rv.commute} $\Rightarrow$ \ref{it:effective}) Let $R$ be a finite $A$-definable subset of $\RV$. Then $R\subset \acl(A)\cap \RV = \rv(\acl(A)\cap K)$ and so there exists a finite $A$-definable set $C\subset K$ such that $\rv(C) = R$. By taking averages as above, we can again ensure that the map $\rv\colon C\to R$ is in fact a bijection.

(\ref{it:effective} $\Rightarrow$ \ref{it:lifting.families}) 
For each $a \in K^n$, effectivity provides an $a$-definable set $C_a$ lifting $R_{\rv(a)}$. Since $K$ is $\aleph_0$-saturated, the family $(C_a)_a$ may be taken $\emptyset$-definable. 

(\ref{it:lifting.families} $\Rightarrow$ \ref{it:effective}) This is trivial.
\end{proof}

A key property of effective theories is that maps can be lifted.

\begin{cor}[Lifting maps $\RV\to \RV$]\label{cor:lifting.maps}
Let $\cT$ be an effective 1-h-minimal theory, let $K$ be a model of $\cT$, and let $A\subset K$. If $f\colon \RV^n\to \RV^m$ is $A$-definable, then there exists an $A$-definable map $F\colon K^n\to K^m$ such that for all $x\in K^n$
\[
\rv(F(x)) = f(\rv(x)).
\]
\end{cor}

We will call such a map $f$ a \emph{lift of $h$}.

\begin{proof}
It suffices to prove this when $K$ is $\aleph_0$-saturated. 
Then the claim follows directly from the previous lemma, since it is a specific instance of lifting families.
\end{proof}
Also, maps from $K$ to $\RV$ can be lifted.
\begin{lem}[Lifting maps $K\to \RV$] \label{lem:K.to.RV.lift}
	Let $\cT$ be an effective $1$-h-minimal theory and $K$ a model of $\cT$. 
	For $A \subset K$, let $f \colon K^n \to \RV^m$ be $A$-definable. 
	Then, there exists an $A$-definable map $F \colon K^n \to K^m$ such that $f = \rv \circ F$.
\end{lem}
\begin{proof} 
	We may again assume that $K$ is $\aleph_0$-saturated.
	For each $b \in K^n$, we find an $A$-definable element $F(b) \in K $ such that $\rv(F(b)) = f(b)$.
	By compactness, $F$ can be taken $A$-definable. 
\end{proof}

By definition, effectivity is preserved when adding constants from $K$.
Let us note that effectivity is also preserved after adding constants from $\Gamma$, see Lemma~\ref{lem:eff.Gamma}.
This will only be needed when comparing derivatives in Section~\ref{sec:intrinsic.jac} for developing integration with measures.
Moreover, the proof uses dimension theory on $\RV$, which will be developed in Section~\ref{sec:dim.RV}.
Hence we postpone the proof to Section~\ref{sec:intrinsic.jac}

\subsection{Examples of effective theories}\label{sec:effective.examples}

In this section we give several natural examples of effective theories. 
In particular this covers all $\emptyset$-effective V-minimal theories, and power-bounded $\cT_{\mathrm{omin}}$-convex fields. 
In full generality one cannot expect h-minimal theories to be effective, since h-minimality says nothing about the structure on $\RV$. 
For example, there are h-minimal structures in which a bijection $\RV\to \RV^2$ is definable, and such structures can never be effective. 
This follows from the results on dimension theory in $\RV$ from Section~\ref{sec:dim.RV}.
To show that certain theories are effective, we use a special case of a relative quantifier elimination statement for short exact sequences of abelian groups from \cite{ACGZ22}.
We first fix some notation.

Let $\cT$ be a theory of valued fields in some language $\cL$ expanding the theory of valued fields.
We consider the residue field $k$ and value group $\Gamma$ with their induced structures. More precisely, define the languages $\cL_{\RF}$ and $\cL_{\VG}$ as follows.
\begin{itemize}
	\item $\cL_{\RF}$ is the expansion of the ring language $\{0,1,+,-,\cdot\}$ by a predicate for each $\emptyset$-definable subset of each Cartesian power $k^n$, and
	
	\item $\cL_{\VG}$ is the expansion of $\{0,1,\cdot,(\cdot)^{-1},<\}$ by a predicate for each $\emptyset$-definable subset of each Cartesian power $\Gamma^n$. Recall that we denote the group operation on $\Gamma^{\times}$ multiplicatively.
\end{itemize}
Consider now the three-sorted language $\cL_{\RRV}$ with sorts $(\RF,\RV,\Gamma)$, consisting of
\begin{itemize}
	\item the language $\cL_{\RF}$ on $\RF$,
	\item the language $\cL_{\VG}$ on $\Gamma$,
	\item the language $\{1,\cdot, (\cdot)^{-1} \}$ on $\RV$,
	\item a function symbol $\iota \colon \RF \to \RV$, and
	\item a function symbol $\abs{\cdot} \colon \RV \to \Gamma$.
\end{itemize}
Given a valued field $K$, let $\cL_{\RF}$, $\cL_{\VG}$ be as above, and define $\iota,\abs{\cdot}$ respectively as the injection and surjection (extended by $0$ at $0$) in the short exact sequence of abelian groups
\begin{equation} \label{eq:ses}
	1 \to k^{\times} \xrightarrow{\iota} \RV^{\times} \xrightarrow{\abs{\cdot}} \Gamma^{\times}_K \to 1. 
\end{equation}
Note that the structure on $\RV$ induced from $K$ may be richer than the one induced by the $\cL_{\RRV}$-structure on $(\RF,\RV,\Gamma)$.

Now consider the expansion $\cL_{\RF}'$ of $\cL_{\RF}$ by new sorts $P^{(n)}$ and maps $\pi_n \colon \RF \to P^{(n)}$. The predicates $P^{(n)}$ are interpreted as the $n$-th power residue classes $k^{\times}/(k^{\times})^{(n)} \cup \{0\}$ and the maps $\pi_n$ are the canonical projections, extended by mapping zero to zero. Note that the corresponding $\cL_{\RF}'$-structure is interpretable in the $\cL_{\RF}$-structure on $k$.

We also need certain maps $\rho_n \colon \RV \to P^{(n)}$ for each $n \geq 0$. 
If $\xi \in \RV$ is such that $\abs{\xi}$ belongs to the subgroup of $n$-th powers $\Gamma^{\times(n)}$, then $\rho_n(\xi)$ is the unique class $\zeta \cdot k^{\times(n)}$ such that $\xi$ and $\zeta$ have the same image in $\RV^{\times}/(\RV^{\times})^{(n)}$. 
For all other $\xi$, we set $\rho_n(\xi) = 0$.
Note that the map $\rho \coloneqq \rho_0$ is simply the inverse to $\iota$ on $\iota(k)$, and zero else.
Finally, write $\res^n \colon K \to k^{\times}/k^{\times(n)} \cup \{0\}$ for the composition $ \rho_n \circ \rv$.
\begin{remark}
	We note that for $A\subset K$ it can happen that $\pi_n (\res A) \subsetneq \res^n A$. 
	Take for example $K = \CC((s))((t))$ with its $t$-adic valuation, so that the residue field is $\CC((s))$. 
	Let $A = \CC((st^2)) \subset K$.
	Then $\res A = \CC$, while $\res^2 A$ contains $s \CC((s))^{\times(2)}$.
	Considering $K$ with its natural $\cL_\val$-structure, we have that $\acl(A) \cap K = A^\alg \cap K = A$.
	In particular, condition~(\ref{it:weak_RF-effective}) in Proposition~\ref{prop:effective-sufficient} below is stronger than simply asking that every $\cL_\RF(\res(A))$-definable subset of $k$ is contained in $\res(A)$. 
\end{remark}
\begin{prop}[{\cite[Corollary 4.8]{ACGZ22}}] \label{prop:RVQE}
	Let $K$ be a valued field, considered in a language $\cL$ expanding $\cL_{\val}$. 
	Let $\cL_{\VG}$, $\cL_{\RF}'$ and $\cL_{\RRV}$ be as described above.
	Then each $\cL_{\RRV}$-definable subset of $\RV^m$ can be defined by a disjunction of formulas of the form 
	\[ \phi_{\RF'}( \rho_{n_0}(t_0(\xi)),\dots,\rho_{n_k}(t_k(\xi)) ) \land \phi_{\VG}( \abs{t_0'(\xi)},\dots,\abs{t_{\ell}'(\xi)}   ), \]
	where $\phi_{\RF'}$ is an $\cL_{\RF}'$-formula, $\phi_{\VG}$ is an $\cL_{\VG}$-formula, $\xi$ is an $m$-tuple of $\RV$-variables, and the $t_i,t_i'$ are terms in the language of groups with inversion $\{1, \cdot, (\cdot)^{-1} \}$.
\end{prop}
Note that \cite[Cor.\,4.8]{ACGZ22} actually applies to a slightly more general setting, as it allows for nontrivial relation symbols using variables over $\RF$ and $\VG$ simultaneously.
\begin{prop}{\label{prop:effective-sufficient}}
	Let $\cT$ be a theory of equicharacteristic zero henselian valued fields. Assume that for every model $K$ of $\cT$ and every parameter set $A \subset K$ with $A = \acl(A) \cap K$ the following hold
	\begin{enumerate}
		
		\item\label{it:determined} every finite $A$-definable subset of $\RV$ is already $\cL_{\RRV}(\rv(A))$-definable, and
		\item\label{it:weak_RF-effective} every finite $\cL_{\RF}'( \bigcup_n \res^n(A) )$-definable subset of $k$ is contained in $\res(A)$.
	\end{enumerate}
	Then $\cT$ is effective.
\end{prop}
\begin{proof}
	Let $K$ be a model of $\cT$ and let $R \subset \RV^\times$ be a nonempty finite set definable over parameters $A \subset K$ with $A  = \acl(A) \cap K$.
	We need to show that $R \subset \rv(A)$.
	Condition \ref{it:determined} guarantees that $R$ is already $\cL_{\RRV}(\rv(A))$-definable.
	By Proposition \ref{prop:RVQE}, $R$ can be defined by a disjunction of formulas of the form
	\begin{equation*}
		\phi_{\RF'}( \rho_{n_0}(t_0(\xi,\rv(a))),\dots,\rho_{n_k}(t_k(\xi,\rv(a))) ) \land \phi_{\VG}( \abs{t_0'(\xi,\rv(a))},\dots,\abs{t_{\ell}'(\xi,\rv(a))}),
	\end{equation*} 
	where $\phi_{\RF'}$ is an $\cL_{\RF}'$-formula, $\phi_{\VG}$ is an $\cL_{\VG}$-formula, $\xi$ is a single $\RV$-variable, $a \in A^n$, and $t_i,t_i'$ are $\{1,\cdot, (\cdot)^{-1} \}$-terms. 
	It suffices to consider the case where $R$ is defined by a single such formula.
	Also, up to passing to a finite definable subset of $R$, we may assume that all elements of $R$ have the same image in $\Gamma$.
	
	We claim that there must be at least one $i\in \{1, \ldots, k\}$ for which $n_i = 0$.
	Indeed, otherwise let $N$ be the least common multiple of all $n_i$'s and let $\eta$ be an $N$-th power in $k^{\times}$. Then by construction $\rho_{n_i}(\eta) = 1$ for all $i = 1,\dots,k$ . Since all $t_i$ are terms in the language $\{1,\cdot,(\cdot)^{-1}\}$, we can factor out each occurrence of $\eta$ to observe that $\rho_{n_i}(t_i( \xi \cdot \eta , \rv(a)) )= \rho_{n_i}( t_i( \xi,\rv(a)))$ for $i = 1,\dots,k$. Combined with the fact that $\abs{\eta} =1$, it follows that $\xi \cdot \eta \in R$. This yields a a contradiction, as $k^{\times (N)}$ is infinite.
	
	Next, we argue that the singleton $\abs{R}$ is contained in the relative divisible hull of $\abs{A^{\times}}$ inside $\Gamma_K^{\times}$. Indeed, else $\rho_{0}(t_i(\xi,\rv(a)))$ is zero for each $i$ with $n_i = 0$.
	But then, as before, there would be some sufficiently large integer $N$ such that $R \cdot k^{\times (N)} \subset R$, contradicting the finiteness of $R$.

	Hence, there must be some minimal $m \in \ZZ_{>0}$ and $a_0 \in A$ such that $\abs{\xi}^m = \abs{a_0}$ for all $\xi \in R$.
	Fix such an $a_0$ and define a new set
	\[ S \coloneqq \{ \xi^m/\rv(a_0) \mid  \xi \in R \}  . \]
	We claim that it suffices to lift $S$.
	To see this, take $\xi\in R$ and let $y\in A$ be a lift of $\xi^m/\rv(a_0)$.
	Since $K$ is henselian of equicharacteristic zero, there exists an $x\in K$ such that $x^m = a_0 y$.
	But then $x\in \acl(A)\cap K = A$.
 	
 	We still need to show that $S \subset \res(A)$.
	A defining formula for $S$ (or rather $\iota(S)$) is given by plugging in terms of the form $\rho_n(\zeta^{m_1} \rv(a)^{m_2})$ into an $\cL_{\RF}'$-formula, for $n \geq 0$, $a \in A$ and $m_1,m_2 \in \ZZ$. As $\abs{\zeta} = 1$ for $\zeta \in S$ it follows that $\rho_n(\zeta^{m_1} \rv(a)^{m_2}) = \pi_n(\zeta^{m_1}) \res^n(a^{m_2})$. Hence, $S$ is $\cL_{\RF}'(\bigcup_{n \in \NN} \res^n(A))$-definable, and we may conclude by condition (\ref{it:weak_RF-effective}). 
\end{proof}
We show in some situations that the second condition becomes easier to check.

\begin{lem} \label{lem:Pn.fin}
Let $K$ be any valued field.
Suppose that the quotient groups $k^{\times}/(k^{\times})^{(N)}$ are finite for all $N > 0$. Then, for each parameter set $A \subset K$ with $A = \acl(A)\cap K$, any finite $\cL_{\RF}'(\bigcup_{n \in \NN} \res^n(A) )$-definable subset of $k$ is contained in a finite $\cL_{\RF}(\res(A))$-definable set.
\end{lem}
\begin{proof}
This follows from the fact that the $\cL_{\RF}'$-structure is interpretable in the $\cL_{\RF}$-structure on $k$ and that we can quantify away the finitely many parameters of the form $\res^n(a)$, for $a \in A$ and $n > 0$.
%
%
%
\end{proof}

\begin{defn}[{{\cite{vdD89}}}] \label{def:alg.bdd}
Let $\cL$ be a language expanding $\cL_\ring$ and let $k$ be a field with $\cL$-structure. 
Let $\ell\subset k$ be a subfield.
Then we say that $k$ is \emph{algebraically bounded over $\ell$} if for any $\cL$-formula $\phi(x,y)$ with $x$ a tuple and $y$ a single variable, there exist polynomials $P_1, \ldots, P_m\in \ell[x,y]$ such that for any $a\in k$, if $\phi(a,k)$ is finite then $\phi(a,k)$ is contained in the zero set of $P_i(a,y)$ for some $i$ for which $P_i(a,y)$ is non-zero.

We say that $k$ is \emph{algebraically bounded over $\emptyset$} if it is algebraically bounded over the prime subfield, and that $k$ is \emph{algebraically bounded} if it is algebraically bounded over $k$.
\end{defn}

By \cite{JY23}, a field $k$ (possibly with extra structure) is algebraically bounded over $\emptyset$ if and only if for every elementary extension $k^*$ and every $A\subset k^*$ we have that $\acl(A) = A^\alg \cap k^*$.
In other words, model-theoretic algebraic closure and field-theoretic algebraic closure coincide.
More generally, it is proven in loc.\ cit.\ that $k$ is algebraically bounded over $\ell$ if and only if for every elementary extension $k^*$ and every $A\subset k^*$ we have $\acl(A\ell) = (A\ell)^{\alg}\cap k^*$.
Moreover, \cite[Prop.\,2.17]{JY23} shows that if $k$ is algebraically bounded then it is already algebraically bounded over $\dcl_{\cL}(\emptyset)$.

\begin{lem}\label{lem:alg.bdd.Pn}
Let $k$ be a field of characteristic $0$ which is algebraically bounded for some language $\cL$ and let $A\subset k$.
Let $\cL'$ be the expansion of $\cL$ with new sorts $P^n$ and corresponding projection maps as above, where $P^n$ is interpreted as $k^\times / k^{\times (n)}$.
Then every finite $\cL'(A\cup_{n\geq 2} P^n)$-definable subset of $k$ is contained in a finite $\cL(A)$-definable set.
\end{lem}

\begin{proof}
Let $Z \subset k$ be a finite $\cL'(AR)$-definable set, where $R = \{\xi_1,\ldots,\xi_m\} \subset P^n$ for some $n \geq 2$. 
Then there exists an $\cL$-formula $\phi(x,y,z)$ and a tuple $a$ over $A$ such that for every $m$-tuple $b \in B = \pi_n^{-1}(\xi_1,\dots,\xi_n)$ it holds that $\phi(a,b,k) = Z$.
Now, we may without loss of generality assume that $A = \dcl_{\cL}(A)$, whence $k$ is algebraically bounded over $A$.
In other words, there exists finitely many non-zero polynomials $P_1,\ldots,P_\ell \in A[y,z]$ such that for each $b \in B$ there is at least one $i$ such that $P_i(b,z) \neq 0$ and $Z$ is contained in the vanishing locus of $P_i(b,z)$.
Now define the (nonzero) polynomial $P = \prod_{i} P_i$ and denote by $U$ the Zariski open set containing those $b \in k^m$ for which $P(b,z)$ is not the zero polynomial.
Since $n \geq 2$, $B$ is Zariski dense, whence so is $U \cap B$.
As for each $z_0 \in Z$, one has for all $b \in U \cap B$ that $P(b,z_0) = 0$, it follows that $P(b,z_0) = 0$ for all $b \in k^m$ 
Hence, the intersection of all vanishing loci of the $P(b,z)$ with $b \in k^m$ is finite, $\cL(A)$-definable, and contains $Z$.
\end{proof}

\subsection{Henselian valued fields} We can now conclude that various theories of henselian valued fields are effective, either in the language of valued fields or augmented with analytic structure.
The key property is that the residue field should be algebraically bounded.

\begin{cor} \label{cor:vf.effective}
Each of the following theories is effectively 1-h-minimal.
\begin{enumerate}
\item\label{it:V-min.effective} Any $\emptyset$-effective V-minimal theory. In particular $\mathrm{ACVF_{(0,0)}}$ and $\mathrm{ACVF}^{\mathrm{an}}_{(0,0)}$.
\item \label{it:val.effective} The $\cL_{\val}$-theory of any equicharacteristic zero henselian valued field whose residue field is algebraically bounded over $\emptyset$. In particular this applies when the residue field is algebraically closed, real closed, $p$-adically closed, pseudofinite or more generally bounded PAC.
\item \label{it:an.effective} Any of the previous theories from~(\ref{it:val.effective}) may be augmented with analytic structure, as in \cite[Definition~2.7]{CLR}.
\end{enumerate}
\end{cor}

\begin{proof}
Note that all the mentioned theories are 1-h-minimal by~\cite[Thm.\,6.2.1,Prop.\,6.4.2]{CHR}. We check that they are effective.

(\ref{it:V-min.effective}) This follows from~\cite[Lem.\, 3.29, 6.2]{HK}.

(\ref{it:val.effective}) We verify the conditions of Proposition \ref{prop:effective-sufficient}.
If $K$ is a henselian valued field, then it has quantifier elimination relative to $\RV$, where the latter is equipped with the structure from $\cL_\RV = \{0,1,\cdot, \oplus\}$, by~\cite[Proposition~4.3]{Flen}. 
Hence, for each $A \subset K$ and each $A$-definable $R \subset \RV$, there is an $\cL_\RV$-formula $\phi_{\RV}$ and a collection of polynomials $p_i \in \ZZ[x_1, \ldots, x_n]$ such that $R$ is defined by
\[ \phi_{\RV}( \xi, \rv( p_1(a) ),\dots,\rv( p_n(a) )  ), \]
for some $a \in A^n$. Since $\oplus$ is definable in $\cL_{\RRV}$ (see e.g.\ ~\cite[Lem.\,5.17]{ACGZ22}), condition (\ref{it:determined}) of Proposition~\ref{prop:effective-sufficient} is satisfied.

As for condition (\ref{it:weak_RF-effective}), by Lemma~\ref{lem:alg.bdd.Pn}, it suffices to show that $\acl(\res(A)) \cap k \subset \res(A)$.
Now recall that the induced structure on the residue field in this case is just the pure field structure and that since $k$ is algebraically bounded over $\emptyset$ the model-theoretic and relative field-theoretic algebraic closure agree. It follows that
\[  \acl(\res(A)) \cap k = \res(A)^{\alg} \cap k \subset \res(A^{\alg} \cap K) = \res(A),  \]
where we used that $K$ is henselian of equicharacteristic zero.

As for the named examples of theories, the facts that algebraically closed fields, real closed fields and $p$-adically closed fields are algebraically bounded over $\emptyset$ is classical and follows from quantifier elimination in the correct language.
See~\cite{vdD89} for details. 
Perfect bounded PAC fields are algebraically bounded over $\emptyset$ by work of Chatzidakis--Hrushovski~\cite{ChatzidakisHrushovski}.

(\ref{it:an.effective}) The arguments for the pure valued field case apply, mutatis mutandis, since we have quantifier elimination in a suitable two-sorted language with sorts $(\VF,\RV)$ (see e.g. \cite[Definition~3.3]{Rid}, but with $\cR = \cO_K$).
Let us recall this language in some more detail.
We start from a noetherian ring $E$ with some ideal $I$ such that $E$ is complete and separated for its $I$-adic topology. The ring of separated power series $E\langle \xi \rangle [[\rho]] $ is defined as in \cite[Definition~2.1]{CLR}.
The language on the valued field sort $\VF$ is an expansion of the valued field language by function symbols for the elements of $E\langle\xi\rangle [[\rho]]$ and one additional 2-ary function symbol $Q(x,y)$ for division in the field, extended by zero when $y =0$. 
The language on $\RV$ is an expansion of $\{0,1,\cdot,\oplus\}$ by constant symbols. Indeed, a priori one has to add a function symbol for the function  induced on $\RV$ by $\rv(u(x,y))$, for each unit $u(\xi,\rho) \in (E\langle \xi \rangle [[\rho]])^{\times}$. 
However, as in \cite[Remark~4.1.7]{CLip}, each such unit is of the form $u = u_0 + \sum_{i} b_i g_i + \sum_{j} \rho_j h_j$, for some $u_0 \in E^{\times}$, certain $g_i,h_j \in E\langle\xi\rangle[[\rho]]$ and $b_i \in I$.
Hence, on $\cO_K^{m} \times \cM_K^n$, $\rv(u(x,y))$ takes a constant value $\rv(u_0) \eqqcolon \xi_{u}$, and it is zero elsewhere. It thus suffices to add constant symbols for the $\xi_{u}$ rather than the function symbols $\bar{u}$.

By relative quantifier elimination as in \cite[Prop.\,3.10]{Rid}, condition (1) is again fulfilled.
Condition (2) also follows from an almost identical argument as in the pure valued field case. Indeed, let $R \subset \RV$ be finite and $\res(A)$-definable with $A = \acl(A) \cap K$. Then $R$ is $\cL_{\ring}$-definable over $\res(A)$, together with all the constants $\xi_{u}$. But since each $u_0$ belongs to  $E$, it follows that $\res(A)$ already includes these $\xi_u$.
\end{proof}
In fact, for pure henselian valued fields effectivity is equivalent to having algebraically bounded residue field over $\emptyset$.

\begin{prop}\label{prop:eff.pure.val.field}
Let $K$ be a henselian equicharacteristic zero field.
Then $\Th_{\cL_\val}(K)$ is effective if and only if the residue field $k$ of $K$ is algebraically bounded over $\emptyset$.
\end{prop}

\begin{proof}
If $k$ is algebraically bounded over $\emptyset$ then this follows from the previous corollary.

So assume that $k$ is not algebraically bounded over $\emptyset$.
This means that there exists some set $B\subset k$ for which $\acl(B)$ contains an element $\xi$ which is not in $B^\alg$.
Let $R\subset k$ be a finite $B$-definable set containing $\xi$.
Let $k'\subset K$ be a lift of the residue field, which exists since we are in equicharacteristic zero, and take $A\subset k'$ for which $\res(A) = B$.
Assume that the finite $A$-definable set $R$ has a finite $A$-definable lift $X\subset K$.
Then $X$ contains an element $x$ with $\res(x) = \xi$.
Now, the proof of \cite[Thm.\,5.5]{JKslim} shows that $K$ is very slim in the language $\cL_{\val}$, whence it is algebraically bounded by~\cite[Thm.\,2.13]{JY23}.
It follows that $\res(x)\in \res(A^\alg)$.
Since $k'$ is relatively algebraically closed in $K$ we conclude that $\res(x)$ is contained in $\res(A)^\alg = B^\alg$, contradiction.
\end{proof}

\begin{remark}
Let $K$ be a henselian equicharacteristic zero valued field, considered in the language $\cL_{\val}$.
Even if the residue field $k$ is algebraically bounded, it may happen that $\Th_{\cL_{\val}}(K)$ is not effective.
Consider for example the field $k$ from~\cite[Ex.\,4.30]{JY25} which is algebraically bounded, but not algebraically bounded over $\emptyset$. 
Then the above proposition shows that $\Th_{\cL_\val}(K)$ is not effective.
Of course, this situation is easy to rectify by adding enough constants to the language.
In particular $\Th_{\cL_\val(K)}(K)$ will be effective.

In general, if the residue field $k$ is algebraically bounded over some subfield $\ell\subset k$, and $A\subset K$ is some subset with $\ell \subset \dcl(A)$, then $\Th_{\cL(A)}(K)$ will be effective.
For $A$ one may take for example a lift of $\ell$ to $K$.
\end{remark}

\subsection{$\cT_{\omin}$-convex valued fields}
Let $\cT_{\omin}$ be a complete o-minimal theory expanding the theory of real closed fields in a language $\cL_{\omin}$ containing $\{0,1,+,\cdot,<\}$.
Call a nonzero subring $\cO_K \subsetneq K$ \emph{$\cT_{\omin}$-convex} if it is convex and $f(\cO_K) \subset \cO_K$ for every $\emptyset$-definable continuous function $f \colon K \to K$. Note that since $\cO_K$ is convex, it is automatically a valuation subring of $K$.

Let $\cL = \cL_{\omin} \cup \{\cO\}$ be an expansion of $\cL_{\omin}$ by a unary relation symbol for a valuation ring. The corresponding \emph{$\cT_{\omin}$-convex theory $\cT$} is the $\cL$-theory of models $K$ of $\cT_{\omin}$ equipped with a $\cT_{\omin}$-convex subring $\cO_K$.
If $\cT_{\omin}$ has quantifier elimination and a universal axiomatiation in $\cL_{\omin}$, then by \cite[Theorem~3.10, Corollary~3.13]{vdDL95} $\cT$ is complete and has quantifier elimination in $\cL$.
Since $\cT_{\omin}$ has definable Skolem functions, we can always extend $\cL_{\omin}$ by definitions to achieve this.

By \cite[Theorem~2.12]{vdDL95} each maximal $\cL_{\omin}$-elementary substructure $K_0 \preccurlyeq K$ with $K_0\subset \cO_K$ maps bijectively onto the residue field under $\res \colon K \to k$. 
Such a bijection induces an $\cL_{\omin}$-structure on $k$, independent from the choice of $K_0$ by \cite[Theorem~2.15]{vdDL95}. 
The full induced structure on $k$ has the same definable sets as this $\cL_{\omin}$-structure \cite[Theorem~A]{vdD97}. 

\begin{lem} \label{lem:res.isom}
Let $\cT_{\omin}$ be a theory of o-minimal fields, and let $\cT$ be the corresponding $\cT_{\omin}$-convex theory.
For any model $K \models \cT$ and any $A \subset K$ with $A = \dcl_{\cL}(A)$ and $\abs{A^{\times}} = 1$, it holds that $\res(A)$ is an elementary $\cL_{\omin}$-substructure of $k$ with $A \cong_{\cL_{\omin}} \res(A)$.
\end{lem}
\begin{proof}
Note that $A$ is an elementary $\cL_{\omin}$-substructure of $K$. 
Zorn's lemma now provides a maximal $K_0 \preccurlyeq K$ with $A \subset K_0 \subset \cO_K$.
The restricted residue map $\res_{| K_0}$ is an $\cL_{\omin}$-isomorphism onto $k$.
Hence the restriction of the residue map to $A$ determines an $\cL_{\omin}$-isomorphism onto its image.
\end{proof}
We recall some facts from \cite{Mil.powBd} on power-bounded o-minimal structures.
A \emph{power function} on $K$ is an $\cL_{\omin}$-definable endomorphism of the multiplicative group $K^{\times}$. 
One calls $K$ \emph{power-bounded} if for each $\cL_{\omin}(K)$-definable function $f \colon K \to K$ there is a power function $g$ such that $\abs{f(x) } \leq |g(x)|$ for all sufficiently large $x \in K$. A theory $\cT$ is power-bounded if all, equivalently any, of its models are.
In any case, one has the \emph{field of exponents}
\[ E \coloneqq \{ f'(1) \mid  f \text{ is a power function on } K \}. \]
See \cite{Mil.powBd} or \cite{vdD97} for more details. 
For use below, we recall that the value group $\Gamma^\times$ has the structure of an $E$-vector space~\cite[Sec.\,3.3]{vdD97}.
\begin{prop} \label{prop:t.convex.effective}
Let $\cT_{\omin}$ be a power-bounded theory of o-minimal fields, and let $\cT$ be the corresponding $\cT_{\omin}$-convex theory. 
Then $\cT$ is effectively 1-h-minimal.
\end{prop}
\begin{proof}
The $\cL$-theory $\cT$ is 1-h-minimal by \cite[Theorem~6.3.4]{CHR}.

To prove effectivity, we shall assume that $\cT_{\omin}$ is given in a language $\cL_{\omin}$ in which it admits quantifier elimination and a universal axiomatization (cf. \cite[\S~2]{vdDL95}).
Take $A \subset K$ such that $A = \acl_{\cL}(A)$ and first assume that $A$ contains an element $a$ with $\abs{a} \neq 1$.
Then $A$ is a model of $\cT_{\omin}$ with a nontrivial $\cT_{\omin}$-convex subring $\cO_K \cap A$. 
By quantifier elimination for $\cT$, $A$ is an elementary $\cL$-substructure of $K$.
In particular, if $\xi \in \RV_K$ is $\cL(A)$-definable, then $\xi \in \RV_A = \rv(A)$.

Now suppose that $\abs{A^{\times}} = 1$.
Without loss of generality, $A$ is generated by a finite tuple $A = \acl_{\cL}(a_1,\dots,a_n)$.
Note that then $A = {\langle a_1,\dots, a_n \rangle}_{\cL_{\omin}}$, the substructure generated by $a_1,\dots,a_n$ by \cite[Lemma~6.3.9]{CHR} (which builds on \cite[Lemma~3.3]{Yin.tcon}, going back to \cite[Lemma~2.6]{vdD97}).
Let $\rank(A)$ be the minimal number $m$ such that there exists $a_1',\dots,a_m' \in A$ for which $A = \langle a_1',\dots,a_m'\rangle_{\cL_{\omin}}$. Hence, its o-minimal rank, denoted by $\rank(A)$, is equal to $m$.
Write $E$ for the field of exponents of $\Th_{\cL}(K)$. Replacing $K$ by a sufficiently saturated extension, we find some $b \in K$ such that $\abs{b}$ is $E$-linearly independent from any $A$-definable element of $\Gamma_K^{\times}$.
Then, as before, $B = \langle A,b\rangle_{\cL_{\omin}}$ is an elementary $\cL$-substructure of $K$.
In particular, for any $A$-definable $\xi \in \RV_K$, it holds that $\xi \in \rv(B)$.
We will show that even $\xi \in \res(B)$ and $\res(A) = \res(B)$ to complete the proof.

Suppose, towards a contradiction that $\abs{\xi} \neq 1$.
Then the $E$-vector space spanned by $\abs{\xi}$ and $\abs{b}$ is two-dimensional.
Since $\rank(B) = \rank(A) + 1$ is finite, we may use \cite[Corollary~5.5]{vdD97} to see that 
\begin{equation*} \label{eq:Wilkie-ineq-absurd}
	\rank(B) \geq \rank(\res(B)) + 2.
\end{equation*}
On the other hand, $\res(A) \subset \res(B)$ and $\res(A) \cong_{\cL_{\omin}} A$ by Lemma \ref{lem:res.isom}, implying that $\rank(\res(B)) \geq \rank(A)$.
This contradicts $\rank(B) = \rank(A) + 1$.

Now assume that $\abs{\xi} = 1$, implying that $\xi \in k \subset \RV$.
Applying \cite[Corollary~5.5]{vdD97} again, we find that
\[ \rank(A) + 1 =  \rank(B) \geq \rank(\res(B)) + 1 \geq \rank(\res(A)) + 1  \]
Again using that $A \cong_{\cL_{\omin}} \res(A)$ it follows that $\res(A)$ and $\res(B)$ have the same rank and thus coincide.
\end{proof}

\subsection{Ordered fields with real analytic structure}\label{ss:or.field.real.an}
Under some mild conditions, ordered fields with real analytic structure as in \cite[\S3]{CLip} are effectively 1-h-minimal.
This gives explicit examples of effectively 1-h-minimal theories where the structure on the residue field is a nontrivial expansion of the ring structure.
We use the notation of \cite{NSV24} and refer to that paper for precise definitions. We recall that if $\cB$ is a real Weierstrass system, then $\cL_{\cB}$ is an expansion of the language of ordered rings $\cL_{\ring} \cup \{<\}$ by function symbols for elements of $B_{m,\alpha}$ with $m \in \NN$ and $\alpha \in \RR_{>1}$, and a function symbol for inversion $(\cdot)^{-1}$. The language $\cL_{\val,\cB}$ is a further expansion of $\cL_{\cB}$ by a relation symbol $\cO_K$ for a convex valuation ring.
\begin{prop} \label{prop:real.an.effective}
Let $K$ be an ordered valued field with convex valuation ring.
Assume $K$ is equipped with real analytic structure from a strong and rich real Weierstrass system $\cB$ as in \cite[Definition~2.1.1]{NSV24}. Then $\Th_{\cL_{\val,\cB}}(K)$ is effectively 1-h-minimal.
\end{prop}
For the proof of this theorem, we will first assume that $K$ has real closed residue field.
In that case, we follow the strategy of Corollary \ref{cor:vf.effective}, which essentially boils down to relative quantifier elimination in an appropriate language.
The case when $K$ does not have real closed residue field will be proved in Section~\ref{sec:coarsenings} and is based on coarsenings.

We need to slightly extend some results from \cite{NSV24}, which we do below. First, we fix some notation.
Whenever $H \leq G$ are additive ordered abelian groups, we write $t^{H}$ for the corresponding ordered multiplicative subgroup $\{t^h \mid h \in H \}$ of $\RR(G)$.
As usual, we consider $\RR(G)$ as an ordered subfield of $\RR((G))$, ordered such that $t$ is positive and smaller than any positive real number.

For $n\geq 0, \alpha \in \RR_{>0}$ we denote by $A_{n, \alpha}$ the ring of real power series in $\RR\llbracket \xi_1, \ldots, \xi_n\rrbracket$ with radius of convergence strictly larger than $\alpha$ (in each of the coordinates simultaneously). Given an additive ordered abelian group $\Omega$, write
\[
A_{n, \alpha}((\Omega)) = \left\{ \sum_{i\in I} f_i t^i \mid f_i \in A_{n, \alpha}, I\subset \Omega \text{ well-ordered}\right\},
\]
and let $\cA((\Omega)) = \{A_{n, \alpha}((\Omega))\}_{n, \alpha}$. There is a Gauss norm on each $A_{n, \alpha}((\Omega))$, defined by 
\[
\big|\big|\sum_{i\in I} f_i t^i\big|\big| = \min \{i\in I\mid f_i\neq 0\}.
\]

For any subgroup $H \leq \Omega$ and $f \in A_{n,\alpha}((\Omega))$, write $\norm{f} \leq t^H$ if there exists some $h \in H$ such that $\norm{f} \leq t^h$.
Similarly, say that $\norm{f} < t^H$ if $\norm{f} < t^h$ for all $h \in H$.
If $H$ is convex, there are natural inclusions
\[ \{  \norm{f} \leq t^H \}/\{ \norm{f} < t^H \}  \subset A_{n,\alpha}((H)) \subset A_{n,\alpha}((\Omega)) . \]
\begin{defn} \label{def:truncation}
Let $H$ be a convex subgroup of an ordered abelian group $\Omega$.
The \emph{truncation map} $\tr_H \colon A_{n,\alpha}((\Omega)) \to A_{n,\alpha}((H))$ is the extension by zero of the quotient map
\[ \{  \norm{f} \leq t^H \} \to  \{ \norm{f} \leq t^H \}/\{ \norm{f} < t^H \}  \subset A_{n,\alpha}((H))   \]
\end{defn}
If $\cB = \{B_{n,\alpha} \}_{n,\alpha}$ is a real Weierstrass system over $\Omega$, then write $\tr_H(\cB) = \{ \tr_H(B_{n,\alpha}) \}_{n,\alpha}$.
\begin{lem}
Let $\cB$ be a real Weierstrass system over $\Omega$ and let $H \leq \Omega$ be a convex subgroup.
Then $\tr_H(\cB)$ is a real Weierstrass system over $H$.
Moreover, if $\cB$ is rich, then so is $\tr_H(\cB)$.
\end{lem}
\begin{proof}
By construction $\tr_H(\cB)$ is a system of subalgebras of those in $\cA((H))$.
We verify the conditions of \cite[Definition~2.1.1]{NSV24} (see also \cite[Definition~3.1.1]{CLip}).
Conditions (1)(a)-(e) follow from the corresponding conditions for $\cB$, using that the restriction of $\tr_H$ to $\{ \norm{f} \leq t^H \}$ is an $\RR$-algebra morphism.
Weierstrass division follows similarly, using the fact that if $\norm{\tr_H(f)} = 1$ and $\tr_H(f)$ is regular of degree $s$ in $\xi_1$, then the same holds for $f$.
Finally, if $\cB$ is rich, then $\tr_H(B_0)  \supset \RR(H)$.
\end{proof}
\begin{notn} \label{not:real.an}
For the rest of this subsection, $\Omega$ is an ordered abelian group, $\cB$ is a rich real Weierstrass over $\Omega$ and $K$ is an ordered field with $\cB$-analytic structure and nontrivial convex valuation ring $\cO_K$. Let $\sigma$ denote the corresponding family of $\RR$-algebra morphisms from each $\cB$ to functions $[-1,1]_K^n \to K$. 
Write $H \leq \Omega$ for the maximal convex subgroup of $\Omega$ such that $\RR(H) \subset \cO$. Note that $H$ is well-defined since $\RR(\Omega) \subset B_0$ by the richness assumption.

To avoid potential confusion with the absolute value on the ordered field $K$, we instead write $\norm{\cdot}_K$ in this section for the multiplicative valuation with respect to $\cO_K$.
Additionally, the convex hull of $\RR$ in $K$ also forms a valuation ring and we denote the corresponding valuation by $\norm{\cdot}$ and call it the \emph{natural valuation}. 
Note that $\norm{\cdot}_K$ is a coarsening of $\norm{\cdot}$.
For a tuple $a = (a_1,\dots,a_n)$, write $\norm{a} = \max\{\norm{a_1},\dots,\norm{a_n}\}$, and similarly for $\norm{\cdot}_K$.  
\end{notn}
\begin{def-lem}
The residue field $k$ has $\tr_H(\cB)$-analytic structure $\sigma_k$ by
\[ \sigma_k(\tr_H(f))(\res(a)) \coloneqq \res(\sigma(f)(a)), \]
for $a \in [-1,1]_K^m$, and $f \in B_{m,\alpha}$ with $\norm{f} \in t^{H} \cup \{0\}$ and $\alpha > 1$.
\end{def-lem}
\begin{proof}
Let $f \in B_{m,\alpha}$ with $\norm{f} \leq t^H$.
We may assume that $\norm{f} \in t^H$ for else $\tr_H(f) = 0$.
First note that $\sigma(f)(a) \in \cO_K$.
Indeed, let $b_0 \in B_0$ be such that $\norm{f/b_0} = 1$.
By \cite[Lemma~2.1.8]{NSV24} it follows that $\norm{ \sigma(f)(a) } \leq \norm{b_0}$.
As $b_0 \in \cO_K^{\times}$ by construction of $H$, it follows that $\norm{ \sigma(f)(a) }_K \leq 1$. 
Also observe that $\sigma_k$ induces an order-preserving ring morphism from $\tr_H(B_0)$ to $k$.

We need to show that $\res(\sigma(f)(a))$ depends only on $\tr_H(f)$ and $\res(a)$.
So take any $g \in B_{m,\alpha}$ for which $\tr_H(g) = \tr_H(f)$. 
Then $\norm{f - g} < t^H$. 
Let $t^{\delta} \in B_0$ be such that $\norm{f - g} = t^{\delta}$ and note that $t^{\delta} \in \cM_K$ by construction of $H$.
By \cite[Lemma~2.1.8]{NSV24}, it follows that $ \norm{ \sigma(f -g)(a)}_K \leq t^{\delta}$ for all $a \in [-1,1]_K^m$.
In particular, $\res( \sigma(f)(a) -\sigma(g)(a) ) = 0$.

Now let $a,b \in [-1,1]_K^m$ be such that $a-b \in \cM_K^n$. 
We show that $\sigma(f)(a)  - \sigma(f)(b) \in \cM_K$.
Up to rescaling with a factor from $\RR(H)$, we may assume that $\norm{f} = 1$.
Furthermore, using condition 1(c) from \cite[Definition~2.1.1]{NSV24}, we may also assume that $\norm{a}, \norm{b} < 1$.
By iterated Weierstrass division
\[ f(x) = \sum_{i = 1}^n (x_i - y_i) g_i(x,y) + f(y), \]
with $\norm{g_i(x,y)} \leq 1$ for $i=1,\dots,n$.
Hence, the ultrametric inequality and \cite[Lemma~2.1.8]{NSV24} imply that
\[ \norm{ \sigma(f(x) - f(y))(a,b)}_K \leq \norm{a - b}_K <1 . \qedhere \] 
\end{proof}
We now fill in the details to a question which was partially left open in \cite[Remarks~4.2.5,5.3.3]{NSV24}. We continue using Notation \ref{not:real.an}

Consider the four-sorted language $\cL_{\rv,\cB}$ with sorts $(\VF,\RV,\RF,\VG)$. The language on the valued field sort $\VF$ is simply $\cL_{\cB}$ (see \cite{NSV24}). The language $\cL_{\RV}$ on the leading term sort $\RV$  is $\{0,1,\cdot, (\cdot)^{-1}\}$. 
The language on the residue field sort $\RF$ is $\cL_{\tr_H(\cB)}$.
The language on the value group sort $\VG$ is the language of (multiplicative) ordered abelian groups extended by an absorbing element 0.
Finally, there are function symbols $\rv \colon \VF \to \RV$, $\iota \colon k \to \RV$ and $\norm{\cdot}_K \colon \RV \to \VG$ for the leading term map, natural inclusion and natural projection respectively. 
Then $(K,\RV_K,k,\Gamma_K)$ is naturally an $\cL_{\rv,\cB}$-structure. Recall that the partial addition $\oplus$ on $\RV$ is definable in this structure \cite[Lem\,.5.17]{ACGZ22}.
%
%
%
\begin{prop} \label{prop:real.an.QE}
The theory $\Th_{\cL_{\rv,\cB}}(K)$ eliminates valued field quantifiers.
\end{prop}
\begin{proof}
This follows exactly as in the proof of \cite[Proposition~4.2.3]{NSV24}.
For each unit $v \in B_{m,\alpha}^{\times}$ with $\norm{v} = 1$, we need a corresponding definable function $\bar{v}$ on $\RF$. 
These are already included as function symbols in $\cL_{\tr_H(\cB)}$.
\end{proof}
Recall the language $\cL_{\val,\cB} = \cL_{\cB} \cup \{\cO\}$ from \cite{NSV24}.
\begin{cor} \label{cor:induced.structure}
For any $\cL_{\val,\cB}$-definable $X \subset K^n$, we have that $\res(X \cap \cO_K^n)$ is a $\cL_{\tr_H(B)}$-definable subset of $k^n$.
\end{cor}
\begin{proof}
This follows from Proposition \ref{prop:real.an.QE} and the relative quantifier elimination for $\RV$ as in Proposition~\ref{prop:RVQE}.
\end{proof}
\begin{proof}[Proof of Proposition \ref{prop:real.an.effective} when $k$ is real closed]
The $\cL_{\val,\cB}$-theory of $K$ is $\omega$-h-minimal by \cite[Theorem~1.0.1]{NSV24}, and so in particular it is $1$-h-minimal.
Let $H$ be the maximal convex subgroup of $\Omega$ for which $\RR(H) \subset \cO_K$.
Let $L$ be a model of $\Th_{\cL_{\val,\cB}}(K)$ and let $A \subset L$ be such that $A = \acl(A) \cap L$.
We verify the two conditions of Proposition \ref{prop:effective-sufficient}.

Note that $L$ extends uniquely to an $\cL_{\rv,\cB}$-structure that is a model of $\Th_{\cL_{\rv,\cB}}(K)$.
Then, by Proposition \ref{prop:real.an.QE} and the construction of $\cL_{\rv,\cB}$ it follows that condition (\ref{it:determined}) is fulfilled

We now verify condition (\ref{it:weak_RF-effective}). 
Write $k_L$ for the residue field of $L$.
Since the quotients $k_L^{\times}/(k_L^{\times})^{(N)}$ consist of at most 2 points, for each $N \in \NN$, it suffices to check that every $\cL_{\tr_H(\cB)}(\res(A))$-definable element of $k_L$ has an $\cL_{\val,\cB}(A)$-definable lift in $L$.

Write $C = \langle \res(A)\rangle_{\cL_{\tr_H(\cB)}}$ for the substructure in the residue field generated by $\res(A)$.
As $k_L$ is real closed, the real closure $C^{\rcl}$ of $C$ is contained in $k_L$.
Since both are models of a common $\cL_{\tr_H(\cB)}$-theory with quantifier elimination (\cite[Lem.\,2.1.10]{NSV24} and \cite[Thm\,3.4.3]{CLip}), it follows that $C^{\rcl} \preccurlyeq k_L$. 
In particular, Corollary~\ref{cor:induced.structure} shows that any $\res(A)$-definable element of $k_L$ is already included in $C^{\rcl}$. 
By henselianity, it then suffices to show that $C \subset \res(A)$.

We now show that for any $\cL_{\cB}$-term $t$ and tuple $\res(a) \in \res(A)^n$ there exists some $b \in A$ such that $\res(b) = t(\res(a))$.
By induction, it suffices to consider the case where $t$ is a function symbol in $\cL_{\tr_H(\cB)}$.
If $t$ is addition, multiplication or field inversion, this is clear.
So let $t = \tr_H(f)$ and take $\res(a) \in [-1,1]_{k_L}^n$ for some $a \in A^n$.
By performing reflections around 1, we may assume that $a \in [-1,1]_{L}^n \cap A$.
Now, by construction of the $\cL_{\tr_H(\cB)}$-structure on $\RF$, it follows that $\res((\sigma f)(a)) = t(\res(a))$.
\end{proof}

\subsection{Coarsenings}\label{sec:coarsenings}

If $K$ is a valued field with valuation ring $\cO_K$ and $\cO_c\supset \cO_K$ is a valuation ring of $K$, then we can consider $K$ also as a valued field with valuation ring $\cO_c$.
The valued field $(K,\cO_c)$ is called a \emph{coarsening} of $(K, \cO_K)$.
We say that the coarsening is \emph{nontrivial} if $\cO_c\neq K$, and in that case we define the maximal ideal $\cM_c$, value group $\Gamma_c$, valuation $|\cdot|_c\colon K\to \Gamma_c$ and $\rv_c\colon K\to \RV_c$ in the usual way.
The coarsening is called \emph{proper} if $\cO_c\neq \cO_K$.

If $K$ is a $1$-h-minimal field in the language $\cL$ and $\cO_c$ is a nontrivial coarsening of the valuation ring, then the theory of $K$ in the language $\cL\cup\{\cO_c\}$ is still $1$-h-minimal for the coarsened valuation ring~\cite[Thm.\,2.2.8]{CHRV}. 
This leads to a way to transfer results from equicharacteristic zero to mixed characteristic.
In fact, the entire theory of h-minimality in mixed characteristic is built upon coarsening, and we refer to~\cite{CHRV} for more details.

As for effectivity, we show that any proper nontrivial coarsening of a $1$-h-minimal structure is automatically effective.
Such a statement makes sense if one interprets effectivity as saying that the residue field is in some sense geometric.
Indeed, the residue field of the coarsened valuation is itself a henselian valued field, and hence has dimension theory as a valued field.
On the other hand, such a result may seem surprising since the structure on the original $\RV$-sort can be essentially arbitrary. 
For example, every subset of $\RV_K$ may be $\cL$-definable, but we still obtain effectivity after coarsening.

\begin{prop}\label{prop:coarsening.eff}
Let $K$ be a $1$-h-minimal field in some language $\cL$, possibly of mixed characteristic. Let $\cO_c$ be a proper nontrivial coarsening of the valuation ring with $\QQ\subset \cO_c$ and let $\cL_c =\cL \cup \{\cO_c\}$. Then the $\cL_c$-theory of $K$ is effectively $1$-h-minimal for the coarsened valuation.
\end{prop}

\begin{proof}
We first record some facts about open balls and coarsening.
Let $\Delta\leq \Gamma_K^\times$ be the subgroup by which we coarsen, so that $|\cdot |_c: K^\times \to \Gamma_K^\times / \Delta$.
Since $\cO_c$ is proper, we have that $\Delta\neq 1$.
Let $B = B_{<\lambda}(a)$ be an open ball for $|\cdot|$, and let $B' = B_{<\mu \Delta}(a)$ be an open ball for $|\cdot|_c$.
Then we have the following.
\begin{enumerate}
\item \label{it:B.in.B'} If $B\subset B'$, then for every $\delta\in \Delta$ we necessarily also have that $B_{< \lambda \delta}(a)\subset B'$.
\item \label{it:B'.in.B} If $B'\subset B$ then $B'\subsetneq B$.
\end{enumerate}

Now let $\xi \in  \RV_c^{\times}$ be $\cL_c(A)$-definable for some $A\subset K$ and put $X = \rv_c^{-1}(\xi)$. 
Note that $\cO_c$ is a pullback of a subset of $\RV_K$ (for the original valuation), so that $\cL_{c}$ is also 1-h-minimal for the original valuation $\abs{\cdot}$ by~\cite[Prop.\,2.6.5]{CHRV}.
Hence, there exists a finite $\cL_c(A)$-definable set $C\subset K$ and $m \in \NN$ such that $X$ is $\abs{m}$-prepared by $C$.
If $X \cap C$ is nonempty we are done, so assume that $X\cap C = \emptyset$.
Take $x\in X$ and let $B = B_{< \lambda}(x)$ be the ball $\abs{m}$-next to $C$ containing $x$.

We claim that $X\subset B$.
Indeed, assume to the contrary that $B\subset X$ and take $\delta\in \Delta$ with $\delta > 1$.
Such a $\delta$ exists since the coarsening is proper.
Since $\QQ\subset \cO_c$, (\ref{it:B.in.B'}) shows that $B_{< \lambda \delta / \abs{m}}(x)\subset X$.
But $B_{< \lambda \delta / \abs{m}}(x)$ is an open ball strictly containing an open ball $1$-next to $C$ (namely $B_{< \lambda / \abs{m}}(x)$), and so must contain a point from $C$.
Thus $C\cap X \neq \emptyset$.
This proves that $X\subset B$.
But then (\ref{it:B'.in.B}) shows that $X\subsetneq B$, contradicting the fact that $C$ $\abs{m}$-prepares $X$.
\end{proof}
Note that Proposition~\ref{prop:coarsening.eff} encompasses our previous example of a pure henselian valued field whose residue field is itself a pure henselian valued field, as in Corollary~\ref{cor:vf.effective}.
In fact this holds more generally when the residue field is t-henselian, even without having both valuations around.
Recall that a field $k$ is \emph{t-henselian} if it is elementarily equivalent to a field which admits a henselian valuation. 
We refer to~\cite{PZ-t-hen} for more details.

\begin{cor}
Let $K$ be a henselian valued field of equicharacteristic zero whose residue field is t-henselian.
Then $\Th_{\cL_{\val}}(K)$ is effective $1$-h-minimal.
\end{cor}

Since t-henselian fields are algebraically bounded over $\emptyset$ by~\cite[Thm.\,5.5]{JKslim}, this in fact already follows from Proposition~\ref{prop:eff.pure.val.field}.
However, the proof is much simpler and uses a trick based on coarsening.

\begin{proof}
The fact that $\Th_{\cL_{\val}}(K)$ is $1$-h-minimal follows from~\cite[Cor.\,6.2.6]{CHR}, so let us focus on effectivity.
We may assume that $K$ is sufficiently saturated, and then \cite[Thm.\,7.2(b)]{PZ-t-hen} shows that $k$ admits a nontrivial henselian valuation $|\cdot|_k\colon k\to \Delta$ with valuation ring $\cO_k$.
By composition we obtain a nontrivial henselian valuation $|\cdot|'\colon K\to \Delta'$.
Let $\cL = \cL_\val$ and $\cL' = \cL \cup \{\cO_k\}$.
By~\cite[Cor.\,6.2.6]{CHR} and~\cite[Thm.\,4.1.19]{CHR} $(K, |\cdot|')$ is $1$-h-minimal in the language $\cL'$.
Now $(K, |\cdot|)$ is a nontrivial proper coarsening of $(K, |\cdot|')$, and so by Proposition~\ref{prop:coarsening.eff} it is effective in $\cL'$.
But $\cL'$ is an $\RV$-expansion of $\cL$, and so~\cite[Cor.\,4.1.17]{CHR} shows that $\acl_\cL(A)\cap K = \acl_{\cL'}(A)\cap K$ for any $A\subset K$.
Therefore $(K, |\cdot|)$ is also effective in $\cL$, as desired.
\end{proof}

We use similar ideas to finish up the proof of Proposition~\ref{prop:real.an.effective} when the residue field is not real closed.

\begin{proof}[Proof of Proposition~\ref{prop:real.an.effective} without restriction on the residue field]
The residue field $k$ of $K$ is itself almost real closed so let $\cO_k\subset k$ be the convex closure of $\RR$.
Then $\cO_k$ is a nontrivial henselian valuation ring with residue field $\RR$.
Let $\cO_K'\subset K$ be the pull-back of $\cO_k$ under the residue map, and note that this is again a nontrivial henselian valuation ring with residue field $\RR$.
Let $\cL' = \cL\cup \{\cO'\}$ so that $K$ has $\cL'$-structure by interpreting $\cO'$ as $\cO_K'$.
Then $\Th_{\cL'}(K)$ with the valuation $\cO_K$ is 1-h-minimal and effective by Proposition~\ref{prop:coarsening.eff}.
Now $\cL'$ is an $\RV$-expansion of $\cL$ and so by~\cite[Cor.\,4.4.17]{CHR} we have $\acl_\cL(A)\cap K = \acl_{\cL'}(A)\cap K$ for any $A\subset K$.
Therefore $\Th_{\cL}(K)$ is also effective.
\end{proof}

To end, we note for later use that a coarsening of a 1-h-minimal structure induces a 1-h-minimal structure on the residue field. If $|\cdot|_c\colon K\to \Gamma_c$ is a nontrivial coarsening of the valuation, then the residue field $k_c$ for $|\cdot|_c$ comes equipped with a valuation $|\cdot|_{k_c}\colon k_c^\times\to k$, where $k$ is the residue field of $K$ for $|\cdot |$.

\begin{prop}\label{prop:coarsening.res.1.h.min}
Consider $k_c$ as a valued field with the full induced $\cL_c$-structure from $K$. Then $k_c$ is $1$-h-minimal.
\end{prop}

\begin{proof}
We may assume that $K$ is sufficiently saturated in both languages $\cL$ and $\cL_c$.
We write $\RV_K, \RV_{k_c}, \RV_{K_c}$ for the leading term sorts for $K$ with $|\cdot |$, $k_c$ with $|\cdot|_{k_c}$ and $K$ with $|\cdot |_c$ respectively. Note that $k_c\subset \RV_{K_c}$ and that $\RV_{k_c}\subset \RV_K$. Let $\lambda\in \Gamma^\times_{k_c}\subset \Gamma^\times_K$, take $\xi\in \RV_{k_c, \lambda}\subset \RV_{K, \lambda}$, let $A\subset k_c^\times$ be finite and let $X\subset k_c$ be $A\xi \cup \RV_{k_c}$-definable. We have to $\lambda$-prepare $X$ by a finite $A$-definable set $D\subset k_c$. 

Recall that $\acl$ is a pregeometry on $K$ in the language $\cL_c$, by~\cite[Lem.\,5.3.5]{CHR}.
Take $B\subset \cO_{K_c}^\times$ a finite set which is independent for the $\acl$ pregeometry and such that $\res_c(B)\subset A\subset \acl(B)\cap k_c$.
Consider $Y = \res_c^{-1}(X)\subset \cO_{K_c}^\times$, which is a $\acl(B)\xi\RV_K$-definable set.
By $1$-h-minimality of $K$ in the language $\cL_c$ with the valuation $|\cdot|$, there exists a finite $\acl(B)\cap K$-definable $C\subset K$ which $\lambda$-prepares $Y$.
After enlarging $C$, we may assume that $C$ is in fact $B$-definable and that $0\in C$.
Then $D = \res_c(C)\subset k_c$ will $\lambda$-prepare $X$.

We claim that $D$ is $A$-definable, which will finish the proof.
To see this, let $\sigma$ be any automorphism of $K$ fixing $A$ point-wise.
Write $B = \{b_1, \ldots, b_n\}$.
Since $\res_c(B)\subset A$, we have that $\rv_c(\sigma(b_i)) = \res_c(\sigma(b_i)) = \res_c(b_i) = \rv_c(b_i)$.
Hence an inductive procedure with~\cite[Lem.\,2.4.4]{CHR} using that $B$ is $\acl$-independent, shows that $(b_1, \ldots, b_n)$ and $(\sigma(b_1), \ldots, \sigma(b_n))$ have the same type over $\RV_c$.
If $\phi(x,B)$ is the $\cL$-formula defining $C$, we then have that 
\[
D = \res_c (\phi(K, B)) = \res_c (\phi(K, \sigma(B)),
\]
and therefore $\sigma(D) = D$. 
\end{proof}

The above result also holds for $\ell$-h-minimality for any $\ell\in \NN\cup\{\omega\}$, with an identical proof.
Together with~\cite[Thm.\,2.2.8]{CHRV} this shows that if $K$ is $1$-h-minimal, then so are $K_c$ and $k_c$. It is natural to wonder whether the converse holds as well.

\begin{question}
Let $K$ be a valued field of characteristic zero in some language $\cL$ expanding the language of valued fields.
Let $\cO_c$ be a nontrivial coarsening of $\cO_K$ and let $\cL_c = \cL \cup\{\cO_c\}$ as above.
Assume that $K_c$ is $1$-h-minimal in the language $\cL_c$, and that $k_c$ with its induced structure is $1$-h-minimal.
Is $K$ then $1$-h-minimal, either in the language $\cL$, or in the language $\cL_c$?
\end{question}

\section{Dimension theory on $\RV$}\label{sec:dim.RV}

Effectivity naturally yields some geometric consequences on $\RV$. It allows us to define derivatives, dimension theory, and so on. In fact, one can think of effectivity as saying that the structure on $\RV$ is in a sense geometric. The key idea is that effectivity tells us how to lift certain definable subsets in $\RV^n$ to $K^n$, and there is already a good dimension theory for definable subsets of $K^n$.

The theory of $\RV$-derivatives will be developed in Section~\ref{sec:measures}, when we deal with integration with measures.

\subsection{The exchange property}

Recall that we work in an effective 1-h-minimal theory $\cT$ of equicharacteristic zero, and $K$ is a model of $\cT$, which we may take to be sufficiently saturated, where necessary.
We will show that under effectivity the $\acl$-dimension defines a geometry on $\RV$.

\begin{defn}
A point $\xi\in (\RV^\times)^n$ is \emph{full-dimensional} if for every $\RV$-definable $X\subset K^n$ we have either $X\cap \rv^{-1}(\xi) = \emptyset$ or $\rv^{-1}(\xi)\subset X$.
\end{defn}

We will later see that in effective theories a full-dimensional point is the same as being a generic point of an $n$-dimensional set.

\begin{lem}\label{lem:emptyset.to.RV}
Let $Z\subset K^n$ be an $\RV$-definable set of dimension $m$. Then there exists a $\emptyset$-definable $Z'\subset K^n$ of dimension $m$ with $Z\subset Z'$.
\end{lem}

\begin{proof}
By item (\ref{it:adapted}) of Theorem~\ref{thm:cell.decomp} there exists a $\emptyset$-definable cell decomposition $\psi \colon K^n \to \RV^N$ for $Z$. 
Now let $Z'$ be the union of all twisted boxes of $\psi$ of dimension at most $m$, which contains $Z$.
By dimension theory on $K$, this set is $\emptyset$-definable and of dimension at most $m$.
\end{proof}

\begin{lem}\label{lem:cyl.cell.decomp}
Let $\chi$ be a $\emptyset$-definable cell decomposition of $K^n$ which refines $\rv \colon K^n\to \RV^n$ and let $B = \rv^{-1}(\xi)$ be an $\rv$-fibre.
Then either $B$ is a twisted box of $\chi$, or there exists a $\emptyset$-definable set $Z\subset K^n$ of dimension at most $n-1$ such that $Z\cap B\neq \emptyset$.
\end{lem}
\begin{proof}
	We prove this by induction on $n$, so assume first that $n=1$.
	Let $C\subset K$ be the finite $\emptyset$-definable set consisting of all cell centres of the twisted boxes of $\chi$, together with $0$.
	If $B\cap C\neq \emptyset$ we simply put $Z = C$.
	Otherwise, since $\chi$ refines $\rv$, $B$ is a twisted box of $\chi$.
	
	Now suppose $n \geq 2$ and that $B$ is not a twisted box of $\chi$.
	Let $\pi \colon K^n \to K^{n-1}$ be the projection onto the first $n-1$ coordinates.
	Let $\chi': K^{n-1}\to \RV^{n'}$ be a $\emptyset$-definable cell decomposition refining $\rv$ and $\pi (\chi^{-1}(\xi))$ for every $\xi\in \RV^n$.
	For $x\in K^{n-1}$, let $Z_x\subset K$ be the set of cell centres of all twisted boxes of $\chi$ over $x$.
	Then $(Z_x)_x$ is a $\emptyset$-definable family of finite subsets of $K$.
	Let $Z = \cup_x \{x\}\times Z_x\subset K^n$, which is a $\emptyset$-definable subset of $K^n$ of dimension at most $n-1$.
	
	If $\pi(B)$ is a twisted box of $\chi'$, then $Z$ must intersect $B$.
	Indeed, otherwise $B$ would itself be a twisted box of $\chi$, since $\chi$ refines $\rv$.
	
	If $\pi(B)$ is not a twisted box of $\chi'$, then the inductive hypothesis provides a $\emptyset$-definable set $Z'\subset K^{n-1}$ intersecting $\pi(B)$ of dimension at most $n-2$.
	But then $Z'\times K$ intersects $B$ and has dimension at most $n-1$.
\end{proof}

\begin{lem}\label{lem:full.dim.lower.dim}
A point $\xi\in (\RV^\times)^n$ is full-dimensional if and only if for every $\emptyset$-definable $Z\subset K^n$ with $\dim Z < n$ we have $Z\cap \rv^{-1}(\xi) = \emptyset$.
\end{lem}

\begin{proof}
Assume first that $\xi$ is full-dimensional and let $Z\subset K^n$ be $\emptyset$-definable of dimension $m<n$.
Since $\rv^{-1}(\xi)$ has dimension $n$, it cannot be contained in $Z$.
Therefore $Z\cap \rv^{-1}(\xi) = \emptyset$.

Assume conversely that $\rv^{-1}(\xi)$ is disjoint from every $\emptyset$-definable subset of $K^n$ of dimension at most $n-1$.
By Lemma~\ref{lem:emptyset.to.RV}, this then in fact holds for every $\RV$-definable subset of $K^n$ of dimension at most $n-1$.
Now let $X\subset K^n$ be any $\RV$-definable set. 
Take a $\emptyset$-definable cell decomposition $\chi$ of $K^n$ for $X$ and for the map $\rv\colon K^n\to \RV^n$. 
In other words, both $X$ and every fibre of $\rv$ are a union of twisted boxes of $\chi$. 
Then Lemma~\ref{lem:cyl.cell.decomp} implies that $\rv^{-1}(\xi)$ is a twisted box of $\chi$.
Hence we have either $\rv^{-1}(\xi)\subset X$ or $\rv^{-1}(\xi)\cap X = \emptyset$.
\end{proof}

\begin{lem}
If $\xi = (\xi_1, \ldots, \xi_n)\in (\RV^\times)^n$ satisfies $\xi_i\notin \acl(\xi_1, \ldots, \xi_{i-1})$ for $i = 1, \ldots, n$, then $\xi$ is full-dimensional.
\end{lem}

\begin{proof}
We induct on $n$.
For $n=1$, let $X\subset K$ be any $\RV$-definable set.
By $1$-h-minimality, there exists a finite $\emptyset$-definable set $C\subset K$ $1$-preparing $X$.
Since $\rv(C)$ is finite, it cannot contain $\xi$.
Hence $\rv^{-1}(\xi)$ must be contained in a single ball $1$-next to $C$, and so $\rv^{-1}(\xi)$ is either contained in $X$, or disjoint from $X$.

Now assume that $n>1$, and assume that $\xi$ is not full-dimensional.
Write $\xi' = (\xi_1, \ldots, \xi_{n-1})\in (\RV^\times)^{n-1}$, which is full-dimensional by induction on $n$.
By the previous lemma, there exists some $\emptyset$-definable $Z\subset K^n$ of dimension at most $n-1$ such that $\rv^{-1}(\xi)\cap Z\neq \emptyset$.
Let $\chi$ be a cell decomposition of $K^n$ which refines $Z$ and $\rv$.
Take any $z\in Z\cap \rv^{-1}(\xi)$ and denote by $F$ the twisted box of $\chi$ which contains $z$.
Then $F$ is contained in $\rv^{-1}(\xi)$ and has dimension at most $n-1$.
Since $\xi'$ is full-dimensional, $\pi_{<n}(F) = \rv^{-1}(\xi')$.
Hence $F$ is of type $(1, \ldots, 1, 0)$.

Denote by $Y$ the union of all twisted boxes $F'$ of $\chi$ of type $(1, \ldots, 1, 0)$ for which $\pi_{<n}(F') = \rv^{-1}(\xi')$.
This is a $\xi'$-definable set, and the restriction of $\pi_{<n}$ to $Y$ is finite-to-one onto $\rv^{-1}(\xi')$.
Take $x\in \rv^{-1}(\xi')$ and let $\{y_1, \ldots, y_N\} = Y\cap (\{x\}\times K)$.
We claim that $\rv(\{y_1, \ldots, y_N\})\subset \RV$ is independent of $x$.
To see this, take $x'\in \rv^{-1}(\xi')$ and let $\{y_1', \ldots, y_M'\} = Y\cap (\{x'\}\times K)$.
Fix $i$ and let $F'$ be a twisted box of $\chi$ of type $(1, \ldots, 1, 0)$ which contains $(x, y_i)$.
Since $\pi_{<n}(F') = \rv^{-1}(\xi')$ there exists some $y_j'\in K$ such that $(x',y_j')\in F'$.
But $F'$ is contained in a single $\rv$-fibre and so $\rv(y_i) = \rv(y_j')$.
Hence $\rv(\{y_1, \ldots, y_N\}) = \rv(\{y_1', \ldots, y_M'\})$ is $\xi'$-definable and $\xi_n\in \{\rv(y_1), \ldots, \rv(y_N)\}\subset \acl(\xi')$.
\end{proof}

The converse holds under effectivity.

\begin{lem}
Assume that $\cT$ is effective.
Then $\xi\in (\RV^\times)^n$ is full-dimensional if and only if $\xi_i\notin \acl(\xi_1, \ldots, \xi_{i-1})$ for $i = 1, \ldots, n$.
\end{lem}

\begin{proof}
The ``if'' direction is the previous lemma. 
For the other direction, we may assume after swapping some coordinates that $\xi_n\in \acl(\xi_1, \ldots, \xi_{n-1})$.
Denote by $\pi = \pi_{<n}\colon \RV^n\to \RV^{n-1}$ the projection onto the first $n-1$ coordinates.
Then there exists a $\emptyset$-definable $R\subset \RV^n$ containing $\xi$ such that the restriction of $\pi$ to $R$ is finite-to-one onto its image.

Consider the set
\[
X = \{(x,\zeta)\in K^{n-1}\times \RV\mid (\rv(x), \zeta)\in R\}.
\]
For every $x\in K^{n-1}$ the fibre $X_x\subset \RV$ is finite, and so by effectivity there exists a $\emptyset$-definable family of finite sets $Y_x\subset K$ such that $\rv(Y_x) = X_x$ for every $x\in K^{n-1}$.
Then $Y = \cup_{x\in K^{n-1}}\{x\}\times Y_x$ is a $\emptyset$-definable set of dimension at most $n-1$ which intersects $\rv^{-1}(\xi)$.
\end{proof}

\begin{lem}\label{lem:RV.acl.exchange}
If $\cT$ is effective then $\acl_\RV$ satisfies the exchange principle.
In particular $\cT$ is pregeometric on $\RV$.
\end{lem}

\begin{proof}
Let $R\subset \RV^\times, \xi, \zeta\in \RV^\times$ and assume that $\xi\in \acl(R\zeta)\setminus \acl(R)$.
We may assume that $R = \{\eta_1, \ldots, \eta_n\}$ is finite and that $\eta\zeta = (\eta_1, \ldots, \eta_n, \zeta)$ is full-dimensional.
Note that since $\xi\notin \acl(R)$, also $\eta\xi = (\eta_1, \ldots, \eta_n, \xi)$ is then full-dimensional.
Now $\eta\zeta\xi$ is not full-dimensional, and so neither is $\eta\xi\zeta$.
But this can only happen if $\zeta\in \acl(\eta\xi)\subset \acl(R\xi)$, as desired.
\end{proof}

Since $\acl$ is pregeometric it defines a notion of dimension for definable subsets of $\RV^n$ as in Section~\ref{ssec:acl.dim.prelims}.
In particular, recall that the $\acl$-dimension is also well-defined for sets definable over $\RV$-parameters even though $K$ may not be effective in the language $\cL(\RV)$.
\begin{defn}
Let $R\subset (\RV^\times)^n$ be definable.
Then we define the \emph{$\RV$-dimension of $R$} to be the $\acl$-dimension of $R$.
We denote the $\RV$-dimension by $\dim_\RV R$.
\end{defn}

\begin{lem}\label{lem:RV.acl.geometric}
If $\cT$ is effective then $\RV$ eliminates $\exists^\infty$.
In particular $\cT$ is geometric on $\RV$.
\end{lem}

The idea of the proof is to reduce from $\RV$ to $k$ and use that pregeometric theories of fields are automatically geometric, see e.g.\ ~\cite{JY23}.

\begin{proof}
Let $R\subset \RV^\times$ be definable with parameters.
Then we claim that $\dim_\RV R = 1$ if and only if the following holds:
\begin{equation}\label{eq:Einfty.elim}
\exists \xi\in \RV^\times \forall \alpha\in k \exists x,y,x',y'\in \xi^{-1}R\cap k^\times\colon x\neq x'\wedge \alpha = \frac{y-y'}{x-x'}.
\end{equation}
Clearly this condition is uniformly definable in families.

Assume first that $R$ has dimension $0$. 
Then $R$ is finite and so the set
\[
\left\{\frac{y-y'}{x-x'}\mid x,y,x',y'\in R, x\neq x', |x|=|y|=|x'|=|y'|\right\} \subset k
\]
is also finite.

Now assume that $R$ has dimension $1$.
We claim that there exists some $\xi\in R$ for which $\xi^{-1}R\cap k^\times$ has dimension $1$.
Suppose this is not the case, so that the map 
\[
|\cdot |\colon R\to \Gamma
\]
is finite-to-one.
Consider the definable set
\[
S = \{(\xi, \eta/\xi)\in \RV^\times \times k^\times \mid \eta\in R, |\eta| = |\xi|\}.
\]
We compute the $\RV$-dimension of $S$ in two different ways to obtain a contradiction.
Let $\pi_1\colon S\to \RV^\times$ be the projection onto the first coordinate and $\pi_2\colon S\to k^\times$ the projection onto the second coordinate.

By assumption $\pi_1$ is finite-to-one, and so $S$ has dimension at most $1$. 
On the other hand, by varying $\xi\in \RV^\times$ we see that $\pi_2(S) = k^\times$.
Also, if $\alpha\in k^\times$ then $\pi_2^{-1}(\alpha)$ is infinite, since $\dim_\RV(R) = 1$.
Hence $\pi_2^{-1}(\alpha)$ has dimension $1$ for every $\alpha\in k^\times$.
Take $\alpha\in k^\times \setminus \acl(\emptyset)$.
Since $\pi^{-1}_2(\alpha)$ has dimension $1$, there exists a $\beta\in \pi^{-1}_2(\alpha)$ for which $\beta\notin \acl(\alpha)$.
Then $(\alpha, \beta)\in S$ is full-dimensional, and so $S$ has dimension $2$.
This is the desired contradiction and we conclude that for some $\xi\in R$ the set $\xi^{-1}R\cap k^\times$ has dimension $1$.

Now consider \eqref{eq:Einfty.elim} again and assume that it is false.
Then there exists some $\alpha \in k$ such that for all $x,y,x',y'\in \xi^{-1}R\cap k$ with $x\neq x'$ we have $\alpha \neq \frac{y-y'}{x-x'}$.
But then the map 
\[
(\xi^{-1}R\cap k)^2\to k\colon (x,y)\mapsto x+\alpha y
\]
is a definable injection of a two-dimensional set into a one-dimensional set.
\end{proof}

\subsection{$\RV$-dimension}

We summarize the main properties of the dimension function on $\RV$.

\begin{theorem}[Dimension theory on $\RV$]\label{thm:dim.theory.RV}
Let $\cT$ be effectively $1$-h-minimal, let $R\subset (\RV^\times)^n, S\subset (\RV^\times)^m$ and $f\colon R\to (\RV^\times)^k$ be $\emptyset$-definable.
Then the following hold.
\begin{enumerate}
\item \label{it:finiteness} (Finiteness) We have $\dim_\RV R > 0$ if and only if $R$ is infinite.
\item \label{it:unions} (Unions) If $n=m$ then $\dim_\RV(R\cup S) = \max\{\dim_\RV R, \dim_\RV S\}$.
\item \label{it:coordinate.projections} (Coordinate projections) There exists a partition of $R$ into finitely many $\emptyset$-definable $R_i$ and for each $i$ a coordinate projection $\pi_i\colon \RV^n\to \RV^{\dim_{\RV}R}$ which is finite-to-one on $R_i$.
\item \label{it:generic.points} (Generic points) We have $\dim_\RV R=n$ if and only if $R$ contains a full-dimensional point.
\item \label{it:independence.of.parameters} (Independence of parameters) Let $A\subset K$. Then the $\RV$-dimension of $R$ in the language $\cL(A)$ is still equal to $\dim_\RV R$. 
\item \label{it:definable.maps} (Definable maps) We have $\dim_\RV R\geq \dim_\RV f(R)$.
\item \label{it:families} (Families) Let $(T_\xi)_{\xi\in R}$ be a $\emptyset$-definable family of subsets of $(\RV^{\times})^{n'}$ and let $q$ be a positive integer. 
Then the set
\[
R_q = \{\xi\in R\mid \dim_\RV T_\xi = q\}
\]
is $\emptyset$-definable and
\[
\{(\xi, \eta)\in \RV^n\times \RV^{n'}\mid \xi\in R_q, \eta\in T_\xi\}
\]
has $\RV$-dimension $q + \dim_\RV R_q$.
\item \label{it:fibrations} (Fibrations) Assume that $\dim_\RV f^{-1}(\xi) = q$ for all $\xi\in f(R)$. Then $\dim_\RV R = \dim_\RV f(R) + q$.
\item \label{it:finite-to-one.maps}(Finite-to-one maps) If $f$ is finite-to-one, then $\dim_\RV R = \dim_\RV f(R)$.
\item \label{it:products} (Products) The $\RV$-dimension of $R\times S$ is $\dim_\RV R + \dim_\RV S$.
\item \label{it:local.dimension} (Local dimension) There exists an element $\xi \in R$ such that $\dim_\RV (R\cap \xi(k^\times)^n) = \dim_\RV R$.
\end{enumerate}
\end{theorem}

\begin{proof}
Since $\acl$ defines a pregeometry on $\RV$ by Lemma~\ref{lem:RV.acl.exchange}, properties (\ref{it:finiteness}, \ref{it:unions}, \ref{it:coordinate.projections}, \ref{it:generic.points}, \ref{it:independence.of.parameters}, \ref{it:definable.maps}) all follow from Proposition~\ref{prop:acl.dim}.

By Lemma~\ref{lem:RV.acl.geometric} the theory of $\RV$ is geometric, and hence Lemma~\ref{lem:geom.dim.function} shows that the $\RV$-dimension is a dimension function in the sense of~\cite{vdD89}.
Therefore (\ref{it:families}) follows immediately from~\cite[Prop.\,1.4]{vdD89}. 
Properties (\ref{it:fibrations}, \ref{it:finite-to-one.maps}, \ref{it:products}) follow directly from the statement about definable families.

So let us prove (\ref{it:local.dimension}), for which we reason similarly as in the proof of Lemma~\ref{lem:RV.acl.geometric}.
Suppose the statement is false, so that every fibre of the map
\[
f\colon R\to (\Gamma^\times)^n\colon \xi\mapsto |\xi|
\]
has $\RV$-dimension strictly less than $m = \dim_\RV R$.
Consider the following $\emptyset$-definable set
\[
S = \{(\xi_1, \ldots, \xi_n, \eta_1/\xi_1, \ldots, \eta_n/\xi_n)\in (\RV^\times)^n\times (k^\times)^n\mid \eta\in R, |\eta_i| = |\xi_i| \text{ for } i = 1, \ldots, n \}.
\]
We will compute the $\RV$-dimension of $S$ in two different ways to obtain a contradiction. 

Let $\pi\colon S\to (\RV^\times)^n$ be the projection onto the first $n$ coordinates. By assumption, every fibre of the map $\pi$ has $\RV$-dimension at most $m - 1$. Hence by (\ref{it:fibrations}), $S$ has $\RV$-dimension at most $n+m-1$.

On the other hand, let $\rho\colon S\to (k^\times)^n$ be the projection onto the last $n$ coordinates. 
By varying $\xi$, we see that $\rho(S) = (k^\times)^n$, which has $\RV$-dimension $n$. 
Since $R$ has dimension $m$, every fibre of $\rho$ also has $\RV$-dimension $m$. But then (\ref{it:fibrations}) yields that $S$ has $\RV$-dimension $n+m$, contradiction.
\end{proof}

We note that the $\RV$-dimension interacts well with the dimension theory on $\VF$ itself.

\begin{prop}
Let $\cT$ be effective.
Let $X\subset K^n$ be a $\emptyset$-definable set of dimension $m$.
Then $\rv(X)\subset \RV^n$ has $\RV$-dimension at most $m$.
\end{prop}

\begin{proof}
Assume that $\dim_\RV \rv(X) = k > m$.
By the previous theorem there exists a coordinate projection $\pi\colon \RV^n\to \RV^k$ such that $\pi(\rv(X))$ has $\RV$-dimension $k$.
Write also $\pi\colon K^n\to K^k$ for the coordinate projection onto the same coordinates, and let $Y = \pi(X)\subset K^k$.
Then $\dim Y \leq m < k$, while $\rv(Y)$ has $\RV$-dimension $k$.
Now note that $\rv(Y)$ has a full-dimensional point $\xi$, while Lemma~\ref{lem:full.dim.lower.dim} shows that $\rv^{-1}(\xi)\cap Y = \emptyset$, contradiction.
\end{proof}

\subsection{Non-examples of effective theories}\label{sec:non.examples}

Dimension theory immediately gives some examples of non-effective theories.

\begin{example}
Let $K$ be any $1$-h-minimal field in some language $\cL$ with residue field $\QQ$. Then the theory of $K$ in $\cL$ is not effective. Indeed, if this theory was effective, then by Theorem~\ref{thm:dim.theory.RV}(\ref{it:products}) the $\RV$-dimension of $\QQ^2$ would be $2$. However, there is a $\emptyset$-definable bijection $\QQ^2\to \QQ$ showing that $\dim_\RV (\QQ^2) = 1$. 
\end{example}

In an effective theory there can be no angular component maps. This gives a large class of natural 1-h-minimal theories which are not effective. Recall that an \emph{angular component map} is a multiplicative map $\ac\colon K\to k$ which coincides with the residue map on $\cO_K^\times$. 

\begin{example}\label{eg:ac.not.eff}
let $K$ be a $1$-h-minimal field in some language $\cL$ and assume that there is a definable angular component $\ac\colon K\to k$. Then the theory of $K$ in $\cL$ cannot be effective. Since effectivity is preserved under adding constants from $K$, we may as well assume that $\ac$ is $\emptyset$-definable. 

The data of an angular component is identical to a splitting of the sequence
\[
1\to k^\times  \to \RV^\times \to \Gamma^\times\to 1,
\]
or in other words to a function $h\colon \RV^\times \to k^\times$ which is a section of the inclusion $k\to \RV$. If $\ac$ is $\emptyset$-definable then so is the map $h$.

The image of $h$ is $k^\times$, which has $\RV$-dimension $1$. If $a\in k^\times$, then $h^{-1}(a)$ maps bijectively to $\Gamma^\times$. Hence $h^{-1}(a)$ is an infinite subset of $\RV$, and so has $\RV$-dimension $1$. But then Theorem~\ref{thm:dim.theory.RV}(\ref{it:families}) would show that $\RV^\times$ has $\RV$-dimension $2$, contradiction.

Note that this gives examples of theories which are $\emptyset$-effective, but not effective.
As a concrete example, consider $\CC((t))$ in the language $\cL = \cL_\val\cup \{\ac, t\}$, where $\ac$ is interpreted as the usual angular component for which $\ac(t) = 1\in k$.

One can also prove more directly that there cannot be a definable angular component by lifting $h$ to a $\emptyset$-definable map $K^\times \to \cO_K^\times$, and showing that such a lift must have infinite discrete fibres.
Hence it cannot exist in any $1$-h-minimal theory.
\end{example}

More generally, we show that there cannot be any definable quasi-sections for $\abs{\cdot} \colon \RV^{\times} \to \Gamma$ in an effectively $1$-h-minimal theory. A \emph{quasi-section} for $\abs{\cdot} \colon \RV^{\times} \to \Gamma$ is an infinite subset $P\subset (\RV^\times)^n$ such that the projection $(\RV^\times)^n\to (\Gamma^\times)^n$ restricts to a finite-to-one map on $P$. Note that an angular component yields a quasi-section $P \subset \RV^\times \to \Gamma$ in a definable way via $P = \{\xi\in \RV^\times \mid \ac(\xi) = 1\}$. 

\begin{prop}
Let $\cT$ be an effective $1$-h-minimal theory of equicharacteristic zero valued fields. Then there does not exist any definable quasi-section.
\end{prop}

\begin{proof}
Let $P\subset (\RV^\times)^n$ be a definable quasi-section. Since effectivity is preserved by adding constants from $K$, we may assume that $P$ is $\emptyset$-definable. By Theorem~\ref{thm:dim.theory.RV}(\ref{it:local.dimension}), there exists an element $\xi\in P$ for which $P\cap \xi\cdot (k^\times)^n$ has the same $\RV$-dimension as $P$. But $P\cap \xi\cdot (k^\times)^n$ is exactly a single fibre of $|\cdot|\colon P\subset \RV^n\to \Gamma^n$, which is finite. Hence
\[
\dim_\RV(P) = \dim_\RV (P\cap \xi\cdot (k^\times)^n) = 0.\qedhere
\]

\end{proof}

Again, one can also more directly prove that there are no definable quasi-sections in effective theories.
Take the family $R$ appearing in the above proof and lift it to a $\emptyset$-definable set  $X\subset K^n\times (\cO_K^\times)^n$. Then show that such a lift cannot exist in any $1$-h-minimal theory.

\section{The integration map}\label{sec:int}

In this section we set up the relevant categories over $\VF$ and $\RV$, develop integration, and in particular prove Theorems~\ref{thm:main.not.eff} and~\ref{thm:main.eff}. 
Let $\cT$ be a 1-h-minimal theory of equicharacteristic zero valued fields. 
Throughout this section we work in a sufficiently saturated model $K$ of $\cT$.

\subsection{Definitions and construction of the lifting map}

We first define our categories of definable sets over $\VF$ and $\RV$.

\begin{defn} \label{def:RV[n].VF[n]}
Let $\cT$ be a 1-h-minimal theory of equicharacteristic zero, and let $n$ be a non-negative integer. 
\begin{enumerate}
\item We denote by $\VF[n]$ the category of $\emptyset$-definable sets $X\subset \VF^n\times \RV^m$ for which the projection $X\to \VF^n$ is finite-to-one. The morphisms in this category are arbitrary $\emptyset$-definable maps.
\item We denote by $\VF$ the colimit over all $\VF[n]$ along the fully-faithful functors $\VF[n] \to \VF[n+1]$ induced by the inclusions $\VF^n \times \RV^m \hookrightarrow \VF^{n+1} \times \RV^m$. Its objects are thus $\emptyset$-definable sets $X\subset \VF^n\times \RV^m$ for which the projection $X\to \VF^n$ is finite-to-one and where $n$ can be any non-negative integer. The morphisms are again arbitrary $\emptyset$-definable maps. 
\item We denote by $\RV[n]$ the category of $\emptyset$-definable sets $R\subset (\RV^\times)^n\times \RV^m$ for which the projection $R\to (\RV^\times)^n$ is finite-to-one. The morphisms in this category are arbitrary $\emptyset$-definable maps.
\item We denote by $\RV[\leq n]$ the formal direct sum of categories $\bigoplus_{i=0}^n \RV[i]$. Its objects are formal sums $\sum_{i = 0}^n R_i$ with $R_i \in \Ob \RV[i]$. The morphisms $ \sum_{i=1}^nf_i\colon  \sum_{i=0}^n R_i \to \sum_{i = 0}^n S_i$ are given by collections of morphisms $f_i \colon X_i \to Y_i$ in $\RV[i]$, for $i = 0,\dots,n$.
\item We denote by $\RV[*]$ the formal direct sum of all $\RV[n]$. This is the colimit of all $\RV[\leq n]$. In particular, every object and morphism of $\RV[*]$ already belongs to some $\RV[\leq n]$.
\end{enumerate}
\end{defn}

\begin{notn} \label{not:RV[<n]}
	For $R \in  \Ob \RV[k]$ and $n \geq k$, we will also write $[R]_k$ for its class in $\K_+ \RV[\leq n]$. This is simply to clarify the ambient dimension of $R$. For example, with notation as in Definition~\ref{defn:lifting.map} below, $[1]_1$ is the point in ambient dimension 1, with lift $\LL[1]_1 = [\cM_K]$, while $[1]_0$ denotes a single point in $\RV[0]$, with lift $\LL[1]_0 = [\{1\}]$. 
\end{notn}

It is important to note that a morphism in both $\VF[n]$ and $\RV[n]$ is an arbitrary $\emptyset$-definable map, in particular it need not commute with the natural map $X\to \VF^n$ or $R\to (\RV^\times)^n$. 
The category $\VF[n]$ is equivalent to the category of $\emptyset$-definable subsets of $\VF^m$ (for some $m$) of $\VF$-dimension at most $n$. 
Similarly for effective 1-h-minimal theories, the category $\RV[n]$ is equivalent to the category of $\emptyset$-definable sets of $\RV$-dimension at most $n$, by Proposition~\ref{thm:dim.theory.RV}(\ref{it:coordinate.projections}). 
In particular, the Grothendieck semirings of these categories are the same. We stick with the above presentations of $\VF[n]$ and $\RV[n]$ since they are easier to work with.

\begin{remark} \label{rem:RV[<=n]}
In practice, we shall often use the following equivalent presentation for $\RV[\leq n]$.
Let $\cC_n$ be the category whose objects are $\emptyset$-definable sets $R \subset \RV^n \times \RV^m$ with finite-to-one coordinate projection onto $\RV^n$.
For each $(i_1,\dots,i_n) \in \{0,1\}^n$ consider the $\emptyset$-definable subsets 
\[R_{(i_1,\dots,i_n)}  = \{ (\xi,\zeta) \in R \mid \xi_j \in i_j  \cdot \RV^{\times}   \},\]
where $0\cdot \RV^\times = \{0\}$ and $1\cdot \RV^\times = \RV^\times$.
The morphisms in $\cC_n$ are generated by
\begin{enumerate}
	\item maps $R \to R$ such that each restriction $R_{(i_1,\dots,i_n)} \to R$ is a coordinate permutation,
	\item maps $R \to S$ which are a disjoint union of $\emptyset$-definable maps $R_{(i_1,\dots,i_n)} \to S_{(i_1,\dots,i_n)}$.
\end{enumerate}
The categories $\RV[\leq n]$ and $\cC_n$ are equivalent via the functor $\RV[\leq n] \to \cC_n$ that maps an object $\sum_{i=0}^n R_i$, with $R_i \subset (\RV^{\times})^i \times \RV^m$ to 
\[
\bigcup_{i=0}^n \{0\}^{n-i} \times R_i \subset \RV^n \times \RV^m.
\]
\end{remark}

\begin{defn}\label{defn:lifting.map}
Let $n$ be a non-negative integer. We define the \emph{lifting map} $\LL\colon \Ob \RV[n]\to \Ob \VF[n]$ which maps $R\subset (\RV^{\times})^n\times \RV^m$ to the lift
\[
\LL R = \{(x,\xi)\in K^n\times \RV^m\mid (\rv(x), \xi)\in R\}.
\]
\end{defn}

Note that $\LL$ is only defined on the objects of the category $\RV[n]$, and in particular it is not a functor.
Our goal is to show that in an effective theory, this lifting map descends to a surjective morphism on the level of Grothendieck semirings, and to identify the kernel precisely. 
First, we show that maps $R \to S$ lift to maps $\LL R \to \LL S$, similar to Corollary~\ref{cor:lifting.maps}. 
To make notation a bit lighter, write $\rv \colon K^n \times \RV^m \to \RV^{n + m}$ for the map $(x,\xi) \mapsto (\rv(x),\xi)$.

\begin{lem} \label{lem:lifting.maps.RV[n]}
Assume that $\cT$ is effective.
Let $R \in  \Ob \RV[n]$, $S \in \Ob \RV[m]$ and let $h \colon R \to S$ be $\emptyset$-definable. Then there exists a $\emptyset$-definable $f \colon \LL R \to \LL S$ such that
\[ \rv(f(x)) = h(\rv(x)) \] 
\end{lem}
\begin{proof}
Write $R \subset \RV^n \times \RV^k$ and $S \subset \RV^m \times \RV^k$.
By Lemma~\ref{le:effective.characterizations}, there exists a $\emptyset$-definable family $(C_x)_{x\in K^n}$ of finite subsets of $K^k$ such that $C_x$ is a lift of the fibre $R_{\rv(x)}\subset \RV^k$.
Then $h$ induces a map 
\[X = \bigsqcup_{x\in K^n} \{x\} \times C_x \to S.\] 
The graph of this map is a $\emptyset$-definable family over $X$. Lifting it using Lemma~\ref{le:effective.characterizations} yields a $\emptyset$-definable map $f_1 \colon X \to K^m \times K^{k}$.
Now $f$ can be constructed as the $\emptyset$-definable map with graph
\[\gr(f) = \{ (x_1,\rv(x_2),y_1,\rv(y_2)) \mid  f_1(x_1,x_2) = (y_1,y_2) \}. \qedhere \]
\end{proof}

\subsection{Lifting bijections} \label{sec:lifting.dim.theory}
We first show that bijections lift from $\RV[n]$ to $\VF[n]$. Our proof is similar to~\cite[Lemma~5.12]{Yin.tcon}, and uses the dimension theory on $\RV$ from Theorem~\ref{thm:dim.theory.RV} as a crucial ingredient.

\begin{lem}\label{lem:lift.bijection}
Assume that $\cT$ is effective.
Let $R,S\subset (\RV^\times)^n\times \RV^m$ be objects in $\RV[n]$ and let $h\colon R\to S$ be a $\emptyset$-definable bijection. Then there exists a $\emptyset$-definable bijection $f\colon \LL R\to \LL S$ lifting $h$.
\end{lem}

For the proof we need some more terminology. 
For $f\colon X\subset K^n\to K^m$ a $\emptyset$-definable map, and $x\in X$, we denote by 
\[
\fibdim(f, x) = \min_{\lambda\in \Gamma^\times} \dim\Big(f^{-1}(f(x)) \cap B_{<\lambda}(x))\Big) \in \{0, \ldots, n\}
\]
the \emph{local fibre dimension of $f$ at $x$}. Note that by dimension theory the set of $x\in X$ of a fixed local fibre dimension is $\emptyset$-definable. 
\begin{proof}
Write the map $h$ as
\[h \colon R \to S \colon (\xi,\zeta) \to ( h_0(\xi,\zeta), h_1(\xi,\zeta) ).\]
We induct on the ambient dimension $n$, the case $n = 0$ being trivial.  
We distinguish two cases: $\dim_{\RV}(R) < n$ and $\dim_{\RV}(R) = n$. 

Assume first that $\dim_{\RV}(R) < n$.
By Theorem~\ref{thm:dim.theory.RV}(\ref{it:coordinate.projections}) we may partition $R$ into $\emptyset$-definable pieces $R_i$ such that both $R_i$ and $h(R_i)$ have a finite-to-one coordinate projection onto $(\RV^\times)^{n-1} \times \RV^m$.
Working piecewise and permuting coordinates, we may assume that $R$ and $S$ have a finite-to-one projection onto their first $n-1$ coordinates.
Let $R',S' \subset \RV^{n-1} \times \RV^{1 + m}$ be isomorphic copies of $R$ and $S$, viewed as objects in $\RV[n-1]$.

The inductive hypothesis provides a $\emptyset$-definable bijection $g \colon \LL R' \to \LL S'$ lifting $h$.
Since each fibre $(\LL R)_x \subset \RV^{1+m}$ is finite, we are essentially reduced to the zero-dimensional case. 
We give some details. 
First write
\begin{align*}
	g \colon \LL R'\subset K^{n-1}\times \RV^\times \times \RV^m 
	&\to \LL S'\subset K^{n-1}\times \RV^\times \times \RV^m  \colon   \\
	 (x,\xi_n,\zeta) &\mapsto (g_0(x,\xi_n,\zeta),g_1(x,\xi_n,\zeta),h_1(\rv(x),\xi_n,\zeta)).
\end{align*}
Now use Lemma~\ref{lem:lifting.maps.RV[n]} to find $\emptyset$-definable maps $c \colon \LL R' \to K $ lifting $(\xi,\zeta) \to \xi_n$ and $d \colon \LL R' \to K$ lifting $g_1$. These lifts can be used to find bijections between the $\rv$-fibres $\rv^{-1}(\xi_n) \subset K$ and $\rv^{-1}(g_1(x,\xi,\zeta)) \subset K$.
Explicitly, we can define $f$ as follows:
\[ f \colon \LL R \to \LL S \colon (x,y,\zeta) \mapsto \left(g_0(x,\rv(y),\zeta), \frac{d(x,\rv(y),\zeta)  }{c(x,\rv(y),\zeta) } y, h_1(\rv(x),\rv(y),\zeta)   \right). \]

Now assume that $\dim_\RV R = n$. 
Lemma~\ref{lem:lifting.maps.RV[n]}, provides $\emptyset$-definable lifts $f\colon \LL R\to \LL S$ of $h$ and $g\colon \LL S\to \LL R$ of $h^{-1}$.
Let $R'$ be the $\emptyset$-definable set such that for each $(\xi,\alpha) \in R'$, with $h(\xi,\alpha) = (\zeta,\beta) \in S$, the restriction of $f$ to 
\[ \rv^{-1}(\xi) \times \{\alpha\} \to \rv^{-1}(\zeta) \times \{\beta\} \]
is already a bijection.
It suffices to show that $\dim_{\RV}(R \setminus R') < n$ to reduce the problem to the previous case.
Since $R'$ is $\emptyset$-definable, we may as well assume $R' = \emptyset$. 

Let $\pi \colon \RV^n \times \RV^m \to \RV^n$ be the projection onto the first $n$ coordinates.
By Theorem~\ref{thm:dim.theory.RV}(\ref{it:finite-to-one.maps}) it is enough to show that $\dim_{\RV}(\pi(R)) < n$.
But now by item (\ref{it:generic.points}) of that theorem, this will be true unless $\pi(R)$ contains some $\xi$ with the property that each $\RV$-definable set $Y \subset K^n$ either contains $\rv^{-1}(\xi)$ or is disjoint from it.

Working towards a contradiction, assume there is some $\xi \in \pi(R)$ with that property. 
Let $\alpha \in \RV^m$ be such that $(\xi,\alpha) \in R$, and write $h(\xi,\alpha) = (\zeta,\beta) \in S$.
Further write $B = \rv^{-1}(\xi)$ and $F = \rv^{-1}(\zeta)$ 
Then $f$ induces an $\RV$-definable map $B \to F$, and $g$ induces an $\RV$-definable map in the opposite direction. We continue denoting these induced maps by $f$ and $g$.

First suppose that $f$ is not surjective. 
Then $f(B)$ is strictly contained in $F$. 
But then $F$ must contain an $\RV$-definable set $Z$ of lower dimension. 
This would yield a nomempty $\RV$-definable $g(Z) \subset B$ of dimension at most $n-1$, a contradiction.

We now show that $f$ is injective.
By genericity of $\xi$ and definability of the fibre dimension, $f$ must have the same fibre dimension at all points $x \in B$.
Moreover, this fibre dimension must be 0, for else $g(f( B )))$ would be a nonempty $\RV$-definable subset of lower dimension.

Thus for each $z \in F$, $f^{-1}(z)$ must be finite and necessarily equal to some fixed number $N \in \NN$. 
We now argue that $N = 1$. Indeed by~\cite[Lem.\,2.5.3]{CHR} we may find some $k \in \NN$ and an $\RV$-definable family of injections for $z \in F$
\[ \iota_z \colon f^{-1}(z) \to \RV^k \]
The union of $\iota_z$ over all $z \in F$ determines an $\RV$-definable map $\iota \colon B \to \RV^k$. 
If $N \geq 2$, then this map would not be constant, contradicting the genericity of $\xi$. So $N = 1$ and $f$ is injective and thus bijective. This gives the desired contradiction, as we assumed that $R' = \emptyset$. 
\end{proof}

We immediately obtain the following.

\begin{cor}
Assume that $\cT$ is effective.
Then the map $\LL$ descends to a well-defined morphism of semigroups
\[
\LL\colon \K_+ \RV[\leq n]\to \K_+ \VF[n].
\]
\end{cor}

\subsection{Surjectivity of the lifting map}

We next show that $\LL$ is surjective on objects up to isomorphism. In short, this follows from cell decomposition. 
\begin{lem} \label{lem:L.surjective}
	Let $X \subset K^n \times \RV^m$ be an object in $\VF[n]$ and take any $\emptyset$-definable cell decomposition $\chi \colon K^n \to \RV^N$ adapted to $X$. 
	Then there exists an $R \in \Ob\RV[\leq n]$ and a $\emptyset$-definable bijection $g \colon X \to \LL R$ which descends to a bijection
	\[ 
	\{ (\chi(x),\xi) \mid (x,\xi) \in X   \} \xrightarrow{\sim} R.
	\]
\end{lem}
\begin{proof}
	Consider a single fibre $F$ of $\chi$. It is of the form
	\[ \chi^{-1}(\xi) =  \{x \in K^n  \mid (\rv( x_i - c_{\xi, i}( \pi_{<i}(x) ) ) )_{i = 1}^n =r_\xi \}, \]
	for some $\xi$-definable $r_\xi \in \RV^n$ and $\xi$-definable centres $c_{\xi,i}$. 
	Consider the map $f\colon K^n \to K^n$ which maps $x$ contained in $F$ to the element 
	\[
	(x_1-c_{\xi, 1}, x_2-c_{\xi, 2}(x_1), \ldots, x_n - c_{\xi,n}(\pi_{<n}(x))).
	\]
	Then $f$ maps each fibre of $\chi$ bijectively to an $\rv$-fibre in $K^n$. 
	Hence, since $\RV$-unions remain finite by~\cite[Cor.\,2.6.7]{CHR}, $f$ is finite-to-one.
	
	Now consider the $\emptyset$-definable map 
	\[
	g\colon X\to K^n\times \RV^n \times \RV^N\colon (x, \zeta) \mapsto (f(x), \zeta, \chi(x)).
	\]
	Then $g$ is bijective onto its image $Y$, and the projection $Y\to K^n$ is finite-to-one since $f$ maps only finitely many twisted boxes to the same $\rv$-fibre. 
	Hence $Y$ is an object in $\VF[n]$ and $X$ and $Y$ are isomorphic in $\VF[n]$. 
	
	Defining
	\[
	R = \{(\rv(y), \eta) \in \RV^n\times \RV^{m+N} \mid (y, \eta)\in Y\},
	\]
	we see that $R$ is an object in $\RV[\leq n]$, and $\LL R = Y$ is isomorphic to $X$. Additionally, $f$ maps fibres of $\chi$ bijectively to fibres of $\rv$, so $g$ descends to a bijection ${\chi(X) \to R}$.
\end{proof}
In effective theories we obtain surjectivity on the level of Grothendieck semirings.
\begin{cor} \label{cor:L.surjective}
	Assume that $\cT$ is effective.
	Then the map $\LL\colon \K_+ \RV[\leq n]\to \K_+ \VF[n]$ is surjective.
\end{cor}

\subsection{Integration in dimension one}\label{sec:kernel.of.LL.dim.one}

Essentially, $\LL$ is the inverse to the map $\Ob\VF[n] \to \K_{+}\RV[\leq n]$ induced by cell decomposition. 
This latter map is not actually well-defined as $\cM_K \cong 1 + \cM_K$ in $\VF[1]$, while $\rv(\cM_K) = \RV_{< 1}^{\times} \sqcup \{0\}$ is clearly not in bijection with $\rv(1 + \cM_K) = \{1\}$.
Switching perspectives, it follows that the kernel of $\LL$ must contain congruence generated by $(\rv(\cM_K),\rv(1 + \cM_K))$.
Amazingly, it turns out that this is the only quotient we have to take.
We prove this following the strategy of Hrushovski--Kazhdan~\cite{HK}, starting with the one-dimensional case.
Let us also mention that the results from this and the following section do not need the assumption of effectivity.

To properly state the following definition, recall from Notation \ref{not:RV[<n]} that elements $[1]_0,[1]_1 \in \RV[\leq 1]$ denote a single point in ambient dimension zero and one respectively.
\begin{defn} \label{def:Isp}
Denote by $I_{\mathrm{sp}}$ the congruence on the semiring $\K_+\RV[*]$ generated by $([\RV_{< 1}^{\times}]_1 + [1]_0,[1]_1)$. 

Denote the induced semigroup congruence on $\K_+ \RV[\leq n]$ by $I_{\mathrm{sp}}[n]$, or simply by $I_{\mathrm{sp}}$, when $n$ is clear from the context.

For any parameter set $A\subset K \cup \RV$. write $I_{\mathrm{sp},A}$ for the same relation, but relative to the language $\cL(A)$.
\end{defn}
Thus, by definition, $\K_+\RV[\leq n]/I_{\mathrm{sp}}$ is the image of $\K_+ \RV[\leq n]$ in $\K_+ \RV[*]/I_\mathrm{sp}$.

For the following lemma, recall that we assume $K$ to be sufficiently saturated (at least $\aleph_0$-saturated).
\begin{lem}[$I_\mathrm{sp}$ in families]\label{lem:compactness.Isp}
	Let $R,S$ be objects in $\RV[\leq n]$ and let $f\colon R \to \RV^N, g\colon S\to \RV^N$ be $\emptyset$-definable.
	Assume that for every $\eta\in \RV^N$ we have that $([f^{-1}(\eta)], [g^{-1}(\eta)])$ belongs to $I_{\mathrm{sp},\eta}$. Then $([R],[S])\in I_\mathrm{sp}$.
\end{lem}

\begin{proof}
	By replacing $R$ by the isomorphic set $\{(\xi,f(\xi)) \mid \xi \in R \}$, we may assume that $f$ is simply the projection onto the last $N$ coordinates. Similarly, we may assume that $g$ is the projection onto the last $N$ coordinates of $S$.
	For $\eta \in \RV^N$, write $R_{\eta}, S_{\eta} \subset \RV^n \times \RV^{m}$ for their respective fibres.
	
	Since $[R_{\eta}]$ and $[S_{\eta}]$ are congruent modulo $I_{\mathrm{sp},\eta}$, we find a sequence $T_{1, \eta}, \ldots, T_{M(\eta), \eta}$ of $\eta$-definable subsets of $\RV^n\times \RV^{m'(\eta)}$ with finite-to-one projection $T_{\ell, \eta}\to \RV^n$, and for each $\ell = 1, \ldots, M(\eta)-1$ we find $\eta$-definable $U_{ij\ell, \eta}\subset \RV^n \times \RV^{m'(\eta)}$ for $0\leq i,j\leq k(\eta)$ with finite-to-one projection onto $\RV^n$ such that:
	\begin{enumerate}
	\item $[R_\eta] = [T_{1, \eta}], [S_\eta] = [T_{M(\eta), \eta}]$, and
	\item for each $\ell = 1, \ldots, M(\eta)-1$ we have
	\begin{equation}\label{eq:kernel.eta}
	\left\{ \begin{array}{ll}
		[T_{\ell, \eta}] &= \sum_{i,j = 0}^{k(\eta)} [U_{ij\ell, \eta}]([\RV_{<1}^{\times}]_1 + [1]_0)^i \cdot ([1]_1)^j, \\
		
		[T_{\ell + 1, \eta}] &= \sum_{i,j = 0}^{k(\eta)} [U_{ij\ell, \eta}]([\RV_{<1}^{\times}]_1 + [1]_0)^j \cdot ([1]_1)^i. \\		
	\end{array} \right.  
	\end{equation}
	\end{enumerate}
	By compactness, we may assume that $m'(\eta) = m', M(\eta) = M$ and $k(\eta) = k$ are independent of $\eta$, and that the $T_{\ell, \eta}$ and $U_{ij\ell, \eta}$ are uniformly $\eta$-definable.
	
	Define $T_\ell = \cup_{\eta} T_{\ell, \eta}\times \{\eta\}$ and similarly $U_{ij\ell, \eta} = \cup_{\eta} U_{ij\ell, \eta}\times \{\eta\}$.
	These are naturally objects in $\RV[\leq n]$, and so we obtain the analogue of \eqref{eq:kernel.eta} but without the subscripts $\eta$ by compactness again.
%
\end{proof}
\begin{defn} \label{def:I.dim1}
	Let $X \subset K \times \RV^m$ be an object in $\VF[1]$. Given a cell decomposition $\chi \colon K \to \RV^N$ adapted to $X$ define
	\[ I_{\chi} X \coloneqq \{ (\chi(x),\xi) \mid (x, \xi) \in X \}. \]
	By Lemma~\ref{lem:L.surjective}, $I_{\chi} X$ is (isomorphic to) a well-defined object object in $\RV[\leq 1]$.
\end{defn}
\begin{remark} \label{rem:I.dim1}
	With notation as in Definition~\ref{def:I.dim1} above, if for each twisted box $F = \rtimes (c_{\xi},r_{\xi})$ and each $x \in F$ one has that the first coordinate of $\chi(x)$ is equal to $r_{\xi}$, then $I_{\chi} X$ is literally an element of $\RV[\leq 1]$. In any case, the class $[I_{\chi} X] \in \RV[\leq 1]$ is well-defined, but dependent on $\chi$.
\end{remark}
%
In general, $I_{\chi} X$ and $I_{\psi} X$ are not isomorphic in $\RV[\leq 1]$ for two different cell decompositions $\chi,\psi$. The following lemma shows that this failure is completely explained by $I_\mathrm{sp}[1]$.

\begin{lem} \label{lem:cell.refinement}
Let $\chi,\psi \colon K \to \RV^N$ be two $\emptyset$-definable cell decompositions adapted to some $X \in \VF[1]$.
Then $[I_{\chi} X] \equiv [I_{\psi} X]$ modulo $I_\mathrm{sp}[1]$.
\end{lem}

\begin{proof}
It suffices to prove this when $\psi$ is a refinement of $\chi$.
Note that for each $\xi \in \RV^m$ the fibre $(I_{\chi} X)_{\xi} \subset \RV^N$ coincides with $(I_{\chi} X_{\xi})$, and similarly for $\psi$. By Lemma~\ref{lem:compactness.Isp}, it then suffices to consider the case where $X \subset K$.
By a similar argument, we may work separately on each of the twisted boxes of $\chi$. So we may reduce to the case where $X$ is a single $\emptyset$-definable twisted box $\rtimes(c,r)$ and $\psi \colon K \to \RV^N$ is a $\emptyset$-definable cell decomposition for $X$. 

First suppose that $X$ does not contain any centre of $\psi$. 
Then $X$ is already a single twisted box of $\psi$ and thus we are done. 
So let $\{d_1,\ldots,d_k\}$ be the set of all centres of $\psi$ in $X$. We proceed by induction on $k$, and may assume $k \geq 1$.
Let $a$ be the $\emptyset$-definable average of $d_{1},\ldots,d_k$. Then also $a \in D$, since $K$ has characteristic zero.
Now consider the following refinement $\chi'$ of $\chi$:
\[ 
\chi' \colon K \to \RV^{N + 1} \colon  
x \mapsto	
\begin{cases}
	(\chi(x),\rv(x - a)/\rv(c - a)) &\text{if } x \in X \\
	(\chi(x),0) &\text{if } x \notin X. 
\end{cases}
\]
Then $\chi'$ replaces the single twisted box $X$ by $\RV_{<1}^\times \sqcup \{0\}$, whence modulo $I_\mathrm{sp}$
\[ [I_{\chi} X] =  [1]_1 \equiv  [\RV^{\times}_{<1}]_1 + [1]_0 =  [I_{\chi'} X ] . \]
If $k = 1$, then $d = a$, so $[I_{\chi'} X] = [I_{\psi} X]$ and we are done.
Else, consider the corresponding refinement $\psi'$ of $\psi$, where we replace the $\emptyset$-definable twisted box of $\psi$ which contains $a$ by $\RV_{<1}^\times \sqcup \{ 0\}$. Then also $[I_{\psi'} X] \equiv [I_\psi X]$ modulo $I_{\mathrm{sp}}$. Since not all ${d_1,\ldots,d_k}$ can live in the same ball 1-next to $a$, we may conclude by induction on $k$ that $[I_{\psi'} X] \equiv [I_{\chi'} X]$.
\end{proof}

The computation of the kernel of $\LL \colon K_+\RV[\leq 1] \to \VF[1]$ now follows from the following reformulation of compatible domain and image preparation from~\cite[Thm.\,5.7.3]{CHR}. 
%
\begin{lem}[{Compatible domain and image preparation}] \label{lem:compatible.domain.image}
	Let $X, Y \subset K \times \RV^m$ be objects in $\VF[1]$. Let $\chi_0,\psi_0 \colon  K \to \RV^N$ be $\emptyset$-definable cell decompositions for $X$ and $Y$ respectively. If $f \colon X \to Y$ is a $\emptyset$-definable bijection, then there exist refinements $\chi,\psi$ of $\chi_0,\psi_0$, such that $f$ descends to an isomorphism $I_{\chi}(X) \to I_{\psi}(Y)$ in $\RV[\leq 1]$.
\end{lem}
\begin{proof}
%
If $m = 0$, then we need to find refinements $\chi,\psi$ of $\chi_0,\psi_0$ such that $f$ maps a single twisted box of $\chi$ to a single twisted box of $\psi$. This is possible by addenda 1 and 3 of~\cite[Thm.\,5.7.3]{CHR}. The general case follows from compactness and keeping track of the $\RV^m$-coordinates.
\end{proof}
%
%
\begin{def-lem} \label{def-lem:integral.dim.1}
	For any $X \in \VF[1]$ and any $\emptyset$-definable cell decomposition $\chi \colon K \to \RV^N$ adapted to $X$, the class
	\[ I [X]  \coloneqq [I_{\chi} X] \in \RV[\leq 1]/I_{\mathrm{sp}}  ,\]
	is well-defined and independent of $\chi$.
\end{def-lem}
\begin{proof}
	This follows immediately from Lemma~\ref{lem:cell.refinement} and Lemma~\ref{lem:compatible.domain.image}.
\end{proof}
\begin{lem}\label{lem:kernel.dim.1}
Let $\cT$ be effective. Then $\LL$ induces an isomorphism of semigroups
\[
\LL\colon \K_+ \RV[\leq 1]/I_{\mathrm{sp}} \to \K_+\VF[1],
\]
inverse to the map $I$ from Lemma-Definition~\ref{def-lem:integral.dim.1}.
\end{lem}

\begin{proof}
As remarked at the beginning of this section, 
\[\LL[1]_1 = [1 + \cM_K] = [\cM_K] = \LL( [\RV_{<1}^{\times}]_1 + [1]_0 ])\]
and hence $\LL$ descends to a map  $\K_+ \RV[\leq 1]/I_{\mathrm{sp}}[1] \to \K_+\VF[1]$, which is surjective by Corollary~\ref{cor:L.surjective}. As for injectivity, note that for $R \in \RV[\leq 1]$ one has by definition that
\[ I \LL [R] = [\rv(\LL R) ] = [R] . \qedhere \]
\end{proof}

\subsection{Integration in higher dimensions}\label{sec:int.higher.dim}

We now determine the kernel of $\LL$ in higher dimensions. We continue following the strategy of~\cite{HK} and start by showing that definable bijections can be written piecewise as a composition of relatively unary maps.
This allows for a good induction on dimension.

\begin{defn}[Relatively unary maps]
Let $f\colon K^n \times \RV^m \to K^n \times \RV^{m'}$ be any map. We say that $f$ is \emph{relatively unary} if all but one of the $K$-coordinates of $f$ are constant. In other words, $f$ is of the form
\[
(x_1, \ldots, x_n,\xi)  \mapsto (x_1, \ldots, x_{k -1}, f_k(x,\xi), x_{k+1},  \ldots, x_{n},\bar{f}(x,\xi)) .
\]
\end{defn}

\begin{lem} \label{lem:transpositions}
	For any permutation $\sigma$ of $\{1,\dots,n\}$, the map
	\[ K^n \to K^n \colon (x_1,\ldots,x_n) \mapsto (x_{\sigma(1)}, \ldots,x_{\sigma(n)}) \]
	is a composition of $\emptyset$-definable relatively unary maps.
\end{lem}
\begin{proof}
	We consider the case $n = 2$. The general case follows similarly. 
	The map $(x_1,x_2) \mapsto (x_2,x_1)$ decomposes as
	\[ (x_1,x_2) \mapsto (x_1,x_2 + x_1) \mapsto (x_2,x_2 + x_1) \mapsto (x_2,x_1). \qedhere \]
\end{proof}
We now show that definable bijections can be made relatively unary.
Hrushovski--Kazhdan prove this result based on the exchange principle~\cite[Lem.\,7.11]{HK} for the algebraic closure operator. Such a statement is also implicit in \cite{CLoes}, to obtain a change of variables formula (see also \cite[Thm.\,5.4.10]{CLNV24}).
We give a detailed proof in the language of cell decomposition to keep this paper self-contained.
\begin{lem}\label{lem:rel.unary}
Let $X,Y \in \Ob\VF[n]$ and let $f\colon X\to Y$ be a $\emptyset$-definable bijection. Then there exists a finite partition of $X$ into $\emptyset$-definable sets $(X_i)_i$, such that the restriction of $f$ to $X_i$ is a finite composition of $\emptyset$-definable relatively unary bijective morphisms in $\VF[n]$.
\end{lem}
\begin{proof}
We will reorder the coordinates on $X$ and $Y$ multiple times throughout this proof, which is allowed by Lemma~\ref{lem:transpositions}.

We have $X,Y \subset K^n \times \RV^m$ for some suitable $m \in \NN$. Write $x = (x_1,\ldots,x_n)$ for the $K$-coordinates of $X$ and $\xi$ for the coordinates on $\RV^m$. Further write
\[
f(x,\xi) = (f_{1}(x,\xi),\ldots,f_{n}(x,\xi),\bar{f}(x,\xi)).
\]
We work by induction on the number of $f_{j}(x,\xi)$ for which $f_j(x,\xi) = x_j$ on $X$. 
If $k = n$, then we are done. So assume that $k < n$. 
By working relative to those $k$ coordinates, we reduce by compactness to the case $k = 0$.

We may moreover assume that for every $\xi \in \RV^m$, both $X_{\xi},Y_{\xi} \subset K^n$ are of the same dimension $d$, independent of $\xi$, and that the projection onto the first $d$ coordinates of $X,Y$ is a finite to-one map to $K^d$. We now proceed by a further induction on $d$, where the case $d = 0$ follows from the fact that finite sets are $\RV$-parametrized (\cite[Lem~2.5.3]{CHR}): we can arrange that already $\bar{f}(x,\xi)$ is injective.

For $i \in \{1,\ldots,n\}$ and $x \in K^n$, write $\hat{x}_i \in K^{n-1}$ for the tuple $x$ without $x_i$.
For each $(x,\xi) \in X$ and $j \in \{1,\ldots,n\}$ we may consider the map induced by $f_{j}$ on the fibre over $(\hat{x}_i,\xi)$:
\[ f_{j,\hat{x}_i,\xi}\colon X_{\hat{x}_i,\xi} \subset K \to K \colon x_i \mapsto f_j(x,\xi).  \]
Now take a cell decomposition $\chi \colon K^n \to \RV^N$ adapted to both $X$ and $Y$ such that for each twisted box $D$ of $\chi$ and all $\xi \in \RV^m$ with $D \times \{\xi\} \subset X$ every map $f_{j,\hat{x}_i,\xi}$ is either constant or injective on $X_{j,\hat{x}_i,\xi} \subset K$. 
Such a cell decomposition exists in view of~\cite[Lem.\,2.8.2]{CHR}.
Moreover, up to taking a finite partition of $X$, we may assume that whether $f_{j,\hat{x}_i,\xi}$ is constant or injective only depends on $i,j$ and $D$ (and not on $\hat{x}_i$ or $\xi$).

Now fix a twisted box $D$, take any $x \in D$ and write $R \coloneqq X_x \subset \RV^m$. By construction of $\chi$, $R$ does not depend on $x$ and is $\chi(D)$-definable. 
We first consider the situation when $X = D \times R$.
Suppose then that each $f_{j,\hat{x_i},\xi}$ is constant on the corresponding fibre in $ D \times R$, for each $i,j$.
As $f$ is a bijection, this would imply that $\dim(X) = 0$, a case which we already considered.

So without loss of generality, $f_{1,\hat{x}_1,\xi}$ is injective. Now replace $D \times R$ by its image under the injective map
\[ (x,\xi) \mapsto (f_{1}(x,\xi),x_2,\ldots,x_n,\bar{f}(x,\xi), \xi,\chi(x)). \]
This works uniformly in $D$, and thus for $X$. We may conclude by our induction on $n - k$.  
\end{proof}

To induct on dimension, we require objects interpolating between $\RV[\leq n]$ and $\VF[n]$. Recall Definition \ref{def:RV[n].VF[n]} for the definition of these categories, as well as the more convenient presentation of $\RV[\leq n]$ from Remark \ref{rem:RV[<=n]}.

\begin{defn}
Let $k,\ell$ be non-negative integers. Denote by $\VFR[k, \ell]$ the category with objects $\emptyset$-definable sets $X\subset K^k\times \RV^\ell\times \RV^m$ such that the projection onto $K^k\times \RV^\ell$ is finite-to-one. The morphisms are $\emptyset$-definable maps $f \colon X \to Y$ over $K^k$, such that for each $x \in K^k$ the induced map between the fibres $f(x,\cdot) \colon X_{x} \subset \RV^{\ell} \times \RV^m \to Y_x$ is a morphism in $\RV[\leq \ell]$.

If $A\subset K\cup \RV$ is some parameter set, we denote by $\VFR_A[k, \ell]$ the corresponding category for the language $\cL(A)$.
\end{defn}

\begin{remark}
	By definition, the category $\VFR[0,\ell]$ is equivalent to $\RV[\leq \ell]$, for any $\ell \in \NN$.
	On the other hand, since morphisms in $\VFR[k,\ell]$ are over $K^k$, the category $\VFR[k,0]$ is not equivalent to $\VF[k]$, for any $k > 0$. The corresponding Grothendieck ring $\K_+ \VFR[k,0]$ keeps track of cut-and-paste relations, but not isomorphisms. One may rather think of $\K_+ \VFR[k,\ell]$ as a ring of ``functions'' from $K^k$ to $\K_+ \RV[\leq \ell]$ (denoted by $\Fn(K^k,\K_+\RV[\leq n])$ in \cite{HK}). 
\end{remark}
\begin{example}
 Consider the classes $[1 + \cM_K] \neq [2 + \cM_K]$ in $\VFR[1,0]$. One can think of them as representing indicator functions for $\rv^{-1}(1)$ and $\rv^{-1}(2)$, respectively. Intuitively, these are distinct functions whose supports have the same volume. More formally, one observes that the integration map from Lemma-Definition~\ref{def-lem:integral.dim.1} maps them both these classes to $[1]_1 = [2]_1 \in \K_+\RV[\leq 1]/I_{\mathrm{sp}}$. 
\end{example}

To compute the kernel of the integration map on $\VF[k]$, we also need a relative version of $I_\mathrm{sp}$ on the interpolating semirings of ``functions'' $\K_+\VFR[k,\ell]$. Note that for $X\in  \VFR[k, \ell]$, if $a\in K^k$ then $X_a\subset \RV^\ell\times \RV^m$ is an object of $\RV_a[\leq \ell]$ in the expanded language $\cL(a)$.

\begin{defn}
Let $I_\mathrm{sp}$ be the equivalence relation on $\K_+ \VFR[k,\ell]$ given by $([X],[Y])\in I_\mathrm{sp}$ if for each $a\in K^k$, we have that $([X_a], [Y_a])\in I_{\mathrm{sp},a}$. 
\end{defn}
The integration map in ambient dimension 1 from Lemma-Definition~\ref{def-lem:integral.dim.1} relativizes to $\K_+\VFR[k,\ell]$. 
Write $x = (x_1,\ldots,x_k)$ for the coordinates on $K^k$ and $\hat{x}_j$ for the same tuple, but with $x_j$ omitted, for $j \in \{1,\ldots,k\}$. 
\begin{defn} \label{def:Ij_chi}
Let $X$ be an object in $\VFR[k, \ell]$ and take $j \in \{1,\ldots,k\}$. Then let $\chi \colon K^n \to \RV^N$ be a $\emptyset$-definable map such that for every $(x,\xi) \in K^k \times \RV^{\ell + m}$ the induced map $\chi(\hat{x}_j,\cdot) \colon K \to \RV^{N}$ is a cell decomposition for the fibre $X_{\hat{x}_j,\xi} \subset K$.
Given such $X$, $j$ and $\chi$, define
\[ I^{j}_{\chi} X \coloneqq \{ (\hat{x}_j,\chi(x),\xi) \in K^{k -1} \times \RV^{ N + \ell + m } \mid (x,\xi) \in X  \} .\]
\end{defn}
One should think of $I^j_\chi X$ as the integral of $X$ with respect to the $j$-th coordinate.
Note that this depends on the chosen cell decomposition. 
\begin{remark}
	As in Definition~\ref{def:I.dim1} and Remark~\ref{rem:I.dim1}, note that $I^j_\chi X$ is naturally in bijection with a well-defined object of $\VFR[k-1,\ell +1]$, and that can always slightly adapt $\chi$ and the used presentation of the $\RV[\leq n]$ such that $I^j_\chi X$ is literally in $\VFR[k-1,\ell+1]$. As in the mentioned remark: $[I_{\chi} X] \in \K_+ \VFR[k-1,\ell +1]$ is thus well-defined, but depends on $\chi$.
\end{remark}
\begin{lem}[Integration is well-defined] \label{lem:I^j.well-defined}
Let $X,Y$ be isomorphic objects in $\VFR[k, \ell] $ and let $j \in \{1,\ldots,k\}$. If $\chi,\psi \colon K^k \to \RV^N$ are $\emptyset$-definable maps such that for each $(x,\xi) \in K^k \times \RV^{\ell + m}$ the induced maps $\chi(\hat{x}_j,\cdot)$ and $\psi(\hat{x}_j,\cdot)$ are cell decompositions of respetively $X_{\hat{x}_j,\xi}$ and $Y_{\hat{x}_j,\xi}$, then $[I^j_{\chi}X] \equiv [I^j_{\psi} Y]$ modulo $I_{\mathrm{sp}}$.
\end{lem}
\begin{proof}
	Everything is relative to $\hat{x}_j \in K^{k-1}$ so by compactness it suffices to prove this lemma for $\VFR[1,\ell]$. 
	Take a $\emptyset$-definable refinement $\phi \colon K \to \RV^M$ of both $\chi$ and $\psi$, such that for all $\eta \in \RV^M$ and $a \in \phi^{-1}(\eta)$ the corresponding fibres $X_a, Y_a \subset \RV^{\ell + m}$ depend only on $\eta$ and are $\eta$-definably isomorphic. It follows that $I^1_{\phi} Y \cong I^1_{\phi} X$ in $\VFR[0,\ell + 1]$.
	
	It thus suffices to show that $[I^1_{\phi} X] \equiv [I^1_{\chi} X]$ modulo $I_{\mathrm{sp}}$. To this end, take any $\xi \in \RV^{\ell + m}$ apply Lemma-Definition~\ref{def-lem:integral.dim.1} to obtain the following chain of equalities in $\RV_{\xi}[\leq 1]/I_{\mathrm{sp},\xi}$
	\[ [ (I^1_{\chi} X)_{\xi} ] =  [I^1_{\chi} (X_{\xi})] = I [ X_{\xi} ] = [ I^1_{\phi} (X_{\xi}) ] = [ (I^1_{\phi} X)_{\xi} ]. \]
	Now Lemma~\ref{lem:compactness.Isp} completes the proof.
\end{proof}
\begin{defn} \label{def:I^j}
	Let $X \in \Ob\VFR[k,\ell]$ and $j\in \{1,\ldots,k\}$. Write $I^j[X] \in \VFR[k,\ell]/I_{\mathrm{sp}}$ for the class of $I^j_{\chi} X$, for any $\chi$ as in Definition~\ref{def:Ij_chi}.
\end{defn}
Let $X\subset K^k\times \RV^\ell\times \RV^m$ be an object in $\VFR[k, \ell]$, and suppose that we are given a $\emptyset$-definable family $(V_x)_{x\in X}$ of sets where $V_x$ is an object in $\VFR[k', \ell']$ in the expanded language $\cL(x)$. We define
\[
\sum_{x\in X} [V_x] = \left[\{(x, y)\in K^k\times \RV^\ell\times \RV^m\times K^{k'}\times \RV^{\ell'}\times \RV^{m'}\mid x\in X, y\in V_x\}\right],
\]
which is an element of $\K_+ \VFR[k+k', \ell+\ell']$. By compactness, if $(V'_x)_x$ is such that $[V_x] = [V'_x]$ for each $x \in X$, then $\sum_{x \in X} [V_x] = \sum_{x \in X} [V_x'] $.

The following is a form of Fubini between $\VF$ and $\RV$.

\begin{lem}[Integration commutes with $\rv$]\label{lem:integration.commutes.rv}
	Let $X \in \Ob\VFR[k, \ell]$ and let $f\colon X\to \RV^n$ be a $\emptyset$-definable map. Then, for each $j \in \{1,\ldots,k\},$
	\[
	I^j [X] = \sum_{\xi\in \RV^n} I^j [f^{-1}(\xi)]
	\]
	in $\K_+ \VFR[k,\ell]/I_{\mathrm{sp}}$. 
\end{lem}
\begin{proof}
	Since everything is again relative to the $K$-coordinates at positions distinct from $j$, it suffices by compactness to show this for $\VFR[1,\ell]$. 
	Take a cell decomposition $\chi \colon K \to \RV^N$ such that $f$ factors through
	$I_{\chi}^1 X$. The lemma now follows from Lemma~\ref{lem:I^j.well-defined} and the observation that
	\[ I^1[X] = [I_{\chi}^1 X] = \sum_{\xi \in \RV^n} [I^1_{\chi} f^{-1}(\xi) ] = \sum_{\xi \in \RV^n} I^1 [f^{-1}(\xi)]. \qedhere   \]
\end{proof}
\begin{lem}[Integration and $I_\sp$]\label{lem:int.Isp}
	Let $X,Y$ be objects in $\VFR[k,\ell]$ such that $[X] \equiv [Y]$ modulo $I_\mathrm{sp}$. Then, for each $j \in \{1,\ldots,k\}$
	it holds that $I^j[X] \equiv I^j[Y]$ modulo $I_\mathrm{sp}$. Hence, $I^j$ descends to a map
	\[ I^j \colon \K_+ \VFR[k,\ell]/I_{\mathrm{sp}} \to  \K_+ \VFR[k-1,\ell +1]/I_{\mathrm{sp}}. \]
\end{lem}
\begin{proof}
	By assumption, we have for each $x \in K^{k}$ that $([X_x],[Y_x]) \in I_{\mathrm{sp},x}$.
	This is witnessed by an equation as in Equation~\ref{eq:kernel.eta} of the proof of Lemma~\ref{lem:compactness.Isp}, for certain $x$-definable-objects in $\RV_x[\leq n]$. 
	We may now find a cell decomposition $\chi \colon K^k \to \RV^N$ such the fibre $X_x \subset \RV^{\ell + m}$ depends only on $\chi(x)$ and all occurring objects and morphisms in that equation are already $\chi(x)$-definable.
	By Lemma~\ref{lem:integration.commutes.rv} and Lemma~\ref{lem:compactness.Isp} it suffices to consider the case $X = X' \times R$ and $Y = X' \times S$ with $X' \subset K^k$ and $(R,S) \in I_{\mathrm{sp}}[\ell]$. This is clear, as $I^j[X] = (I^j[X']) \cdot [R]$, and similarly for $Y$. 
\end{proof}

\begin{lem}[Fubini]\label{lem:Fubini}
Let $j, j'\in \{1,\ldots,k\}$ be distinct, and let $X$ be an object in $\VFR[k,  \ell]$. Then $I^jI^{j'}[X] = I^{j'}I^j [X]$ in $\K_+\VFR[k-2, \ell+2]/I_\mathrm{sp}$.
\end{lem}

\begin{proof}
It suffices to consider the case $k =2$. 
By Lemma~\ref{lem:integration.commutes.rv} and Lemma~\ref{lem:compactness.Isp}, we may fix any $\emptyset$-definable cell decomposition $\chi \colon K^2 \to \RV^N$ and work on each twisted box seperately.
Hence, by Lemma~\ref{lem:cell.decomp.bi.twisted} we may assume that $X = X' \times R$ for some $\emptyset$-definable bi-twisted box $X' \subset K^2$ and some $\emptyset$-definable $R \subset \RV^{\ell + m}$. 
As $I^j[X] = (I^j[X'])\cdot R$, we may ignore $R$ and simply show that $I^1 I^{2}[X'] = I^{2}I^1 [X']$.

By construction, $X'$ is of the form
\[ 
X' = \{(x,y)\in K^2\mid \rv(x-c_1)=\xi_1, \rv(y-c_2(x_1)) = \xi_2\},
\] 
where $c_2$ has the Jacobian property. 
We may read of from this presentation of $X'$ that $I^1 I^2 [X'] = [(\xi_1,\xi_2)] \in \RV[\leq 2]$.

If $X'$ is an actual box, then take a cell decomposition $\psi \colon K \to \RV^M$ of the projection of $X'$ on the $Y$-coordinate to further reduce to the case where $c_2$ is constant. Then $I^2 I^1[X'] = [(\xi_2,\xi_1)]$, which is indeed equal to $[(\xi_1,\xi_2)]$ in $\RV[\leq 2]/I_{\mathrm{sp}}$.

If $X'$ is not a box, then Lemma~\ref{lem:bi.twisted} implies that
\[
X' = \{(x,y)\in K^2\mid y \in Y, \rv(x-c_2^{-1}(y)) = \xi_3\},
\]
where $\xi_3 = {-} \rv(c_2')^{-1} \xi_2$, and $Y\subset K$ is the projection of $X'$ onto the second coordinate. 
Hence, $I^1[X'] = [Y \times \{\xi_3\}] \in \VFR[1,1]/I_\mathrm{sp}$. 
Now $c_2$ provides a $\emptyset$-definable bijection between $\rv^{-1}(\xi_1) + c_1$ and $Y$, whence $I[Y'] = [\xi_1] \in \K_+ \RV[\leq 1]$. 
It thus follows that in $\K_+ \RV[\leq 2]$
\[  I^2 I^1[X'] = [(\xi_1,\xi_3)] = [(\xi_1,\xi_2)] = I^1 I^2[X'], \]
using that $\xi_2$ and $\xi_3$ are interdefinable.
\end{proof}
\begin{lem}\label{lem:kernel.LL}
Let $R,S$ be objects in $\RV[\leq n]$ and assume that $[\LL R] = [\LL S]$ in $\VF[n]$. Then $[R] \equiv [S]$ modulo $I_\mathrm{sp}$.
\end{lem}

\begin{proof}
Let $f\colon \LL R \to \LL S$ be an isomorphism. Denote by $I = I^1\circ \ldots \circ I^n$ the iterated integration map
\[
I\colon \K_+ \VFR[n,0]/I_\mathrm{sp} \to \K_+ \VFR[0,n]/I_\mathrm{sp} =  K_+ \RV[\leq n] / I_\mathrm{sp}.
\]
We will show that $I [\LL R] = I [\LL S]$. The result then follows by noting that, by construction, $I [\LL R]$ is nothing but the class of $R$ in $\K_+ \RV[n]/I_\mathrm{sp}$. 

By combining Lemma~\ref{lem:rel.unary}, Lemma~\ref{lem:integration.commutes.rv} and Lemma~\ref{lem:compactness.Isp}, we may assume that $\LL R$ and $\LL S$ are in $K^n$ and $f$ is relatively unary. After the same coordinate permutation of both the domain and codomain, $f$ is therefore of the form
\[
(x_1, \ldots, x_n)\mapsto (x_1, \ldots, x_{n-1}, f(x_1, \ldots, x_n)).
\]
By Fubini (Lemma~\ref{lem:Fubini}), this coordinate permutation does not affect the result of computing $I [\LL R]$ or $I [\LL S]$. For every $y\in K^{n-1}$, we note that $f$ defines a $y$-definable bijection between $(\LL R)_y$ and $(\LL S)_y$.
Write $I^n [\LL R] = [X]$ and $I^n [\LL S] = [Y]$. Then, for each $y \in K^{k-1}$ it follows from Lemma-Definition~\ref{def-lem:integral.dim.1} that $[X_y] = I [(\LL R)_y] = [Y_y]$ in $\RV_{y}[\leq n]/I_{\mathrm{sp},y}$.
It follows that $I^n [\LL R] = I^n [\LL S] $ as objects in $\VFR[n-1,1]/I_\mathrm{sp}$, proving the lemma.
\end{proof}

\subsection{Main results}

Our main theorems now follow immediately from the work in the previous sections.

\begin{thm} \label{thm:integration.VF[n]}
Let $\cT$ be an effectively $1$-h-minimal theory of equicharacteristic zero valued fields.
Then for every $n$ there is an isomorphism of semigroups
\[
\int\colon \K_+ \VF[n]\to \K_+ \RV[\leq n]/I_{\mathrm{sp}}
\]
such that for $X\in \VF[n], R\in  \RV[\leq n]$ we have $\int[X] = [R]$ if and only if $[\LL R] = [X]$.
\end{thm}

\begin{proof}
By Lemma~\ref{lem:lift.bijection} the map $\LL$ is well-defined on $\K_+ \RV[\leq n]\to \K_+ \VF[n]$. Lemma~\ref{lem:L.surjective} shows that it is surjective, and by Lemma~\ref{lem:kernel.LL} the kernel is precisely $I_{\mathrm{sp}}$. 
Then $\int$ is precisely the map constructed in that lemma.
\end{proof}

\begin{thm} \label{thm:integration.semiring}
Let $\cT$ be an effectively $1$-h-minimal theory of equicharacteristic zero valued fields.
Then there is an isomorphism of semirings
\[
\int\colon \K_+ \VF\to \K_+ \RV[*]/I_{\mathrm{sp}}
\]
such that for $X\in \VF[n], R\in  \RV[\leq n]$ we have $\int [X] = [R]$ if and only if $[\LL R] = [X]$.
\end{thm}

\begin{proof}
This follows from Theorem \ref{thm:integration.VF[n]}, noting that $\LL$ is compatible with multiplication.
\end{proof}

\begin{cor}
	The integral in Theorem~\ref{thm:integration.semiring} induces an isomorphism of rings
	\[
	\int\colon \K \VF \to \K \RV[*] / ([\RV^\times_{<1}]_1 + [1]_0 - [1]_1).
	\]
\end{cor}

For non-effective theories it might not be possible to lift bijections from $\RV[\leq n]$ to $\VF[n]$. 
Nevertheless, Lemma~\ref{lem:kernel.LL} and surjectivity of $\LL$ on the level of objects still hold. 
We thus obtain the following.

\begin{thm}\label{thm:non.eff.int}
Let $\cT$ be a $1$-h-minimal theory of equicharacteristic zero valued fields.
Then there exists a surjective morphism of semirings
\[
\int\colon \K_+ \VF\to \K_+ \RV[*]/I_{\mathrm{sp}}
\]
such that if $R,S$ are objects in $\RV[\leq n]$ for which $[\LL R] = [\LL S]$ then $[R]=[S]$ modulo $I_\mathrm{sp}$.
\end{thm}

In~\cite{HK} it is proven that in a V-minimal theory if $\int [X] = \int [Y]$ then $X$ and $Y$ are \emph{effectively isomorphic}. 
This means that for every model $K$ of $\cT$ and for every parameter set $A$ such that $\Th_{\cL(A)}(K)$ is effective there exists an $A$-definable bijection $X\to Y$. 
Such a result is false in our context, since there might not even exist any parameter set $A$ for which the theory of $K$ is effective in $\cL(A)$. 
Indeed, consider for example any theory with an angular component map in the language.
On the other hand, this notion does make sense when the residue field is algebraically bounded and one works in an expansion of the pure valued field language by constants.
In that case, one can indeed prove a result of this form.

Let $K$ be a henselian valued field, considered in the language $\cL_\val$, and let $\cL_\an\supset \cL_\val$ be an expansion by a separated analytic structure in the sense of Cluckers--Lipshitz--Robinson~\cite{CLR}.
Then the $\cL_\an(K)$-definable sets in $\RV$ are the same as the $\cL_\val(K)$-definable sets, as follows from the relative quantifier elimination of~\cite[Thm.\,3.10]{Rid}.
Combining this with Theorem~\ref{thm:integration.semiring} immediately shows the following.

\begin{cor}
Let $K$ be an equicharacteristic zero valued field for which $\Th_{\cL_\val}(K)$ is effective.
Let $\cL_\an\supset \cL_\val$ be an expansion by a separated analytic structure as in~\cite{CLR}.
Then the natural map
\[
\K_+ \VF(\cL_\val(K))\to \K_+ \VF(\cL_\an(K))
\]
is an isomorphism.
\end{cor}

In particular, injectivity of this map shows that if $X,Y\subset K^n$ are $\cL_\val(K)$-definable and there exists an $\cL_\an(K)$-definable bijection $X\to Y$, then there already exists an $\cL_\val(K)$-definable bijection $X\to Y$.
Surjectivity shows that for every $\cL_\an(K)$-definable $X\subset K^n$ there exists an $\cL_\val(K)$-definable $Y\subset K^n$ and an $\cL_\an(K)$-definable bijection $X\to Y$.

\section{Measures}\label{sec:measures}

We extend the work from the previous chapter to include measures. 
In more detail, we will define analogues of the categories $\VF[n]$ and $\RV[\leq n]$ where one additionally has a volume form on the objects. 
These categories will be denoted by $\mu \VF[n]$ and $\mu \RV[\leq n]$. We then classify objects in $\mu \VF[n]$ up to measure-preserving bijections. 
In a future paper we will prove that the Grothendieck rings we construct specialize to the rings of motivic volumes as constructed by Cluckers--Loeser~\cite{CLoes}.
However, our ring of motivic functions is more general, and keeps track of more information.

Throughout this section we continue with working in a $1$-h-minimal theory $\cT$ of equicharacteristic zero. We fix a sufficiently saturated model $K$ of $\cT$.

\subsection{Differentiation on $\RV$}

Effectivity leads to a notion of Jacobians on $\RV$. If $h\colon R\subset \RV^\times \to \RV^\times$ is a $\emptyset$-definable map, we would like to define a notion of derivative $h'\colon R\to \RV^\times$. 
Such a notion will be required to define measure-preserving maps, which will be our morphisms in the categories $\mu \RV[n]$. Moreover, we would like this notion of $\RV$-derivative to commute with lifting, in the sense that if $f\colon \LL R\to K$ is a $\emptyset$-definable lift of $h$, then it should hold that $\rv f'(x) = h'(\rv(x))$. There are of course several problems with this approach. Firstly $f$ might not be $C^1$ everywhere, so that $f'(x)$ is not well-defined. Secondly, even if $f$ is $C^1$ everywhere, $\rv f'(x)$ might depend on $x$ even when $\rv(x)$ is fixed. In that case $\rv f'(x)$ is not a function of $\rv(x)$. And thirdly, the resulting derivative $h'$ might depend on the chosen lift, so that derivatives are not unique. We will nevertheless follow such an approach based on lifting to define $\RV$-Jacobians. We solve the first two issues by taking a good enough lift, but $h'$ might still depend on the chosen lift.

For $\xi = (\xi_1, \ldots, \xi_n)\in \RV^n$ we define $|\xi| = |\xi_1|\cdots |\xi_n|$. This should be thought of as the volume of the box $\rv^{-1}(\xi)\subset K^n$. For $f\colon U\subset K^n\to K^n$ a $C^1$ map on an open $U$, we denote by $\Jac f\colon U\to K$ the Jacobian of $f$, which is the determinant of the total derivative of $f$.


\begin{lem}\label{lem:lift.bijections.mu}
Let $\cT$ be effective, let $R,S \subset (\RV^{\times})^n \times \RV^m$ be objects in $\RV[n]$ and let 
\[h \colon  R \to S \colon (\xi,\zeta) \mapsto (h_0(\xi,\zeta),h_1(\xi,\zeta))\]  
be a $\emptyset$-definable bijection. 
Then there exists a $\emptyset$-definable lift $f\colon \LL R\to \LL S$ of $h$ with the following properties:
\begin{enumerate}
\item $f$ is a bijection,
\item for each $(\xi,\zeta) \in R $, $f$ determines a $C^1$ map 
	\[ f_{\zeta} \colon \rv^{-1}(\xi) \to \rv^{-1}(h_0(\xi,\zeta)). \]
\item $\rv(\Jac f_\zeta)$ is constant on $\rv^{-1}(\xi)$ for every $(\xi,\zeta) \in R$, and 
\item for $(\xi,\zeta) \in R$ and $x\in \rv^{-1}(\xi)$ we have
\[
\abs{\Jac f_{\zeta}(x)} = \abs{h_0(\xi,\zeta))} / \abs{\xi}.
\]
\end{enumerate}
\end{lem}

The statement of this lemma is quite similar to Lemma~\ref{lem:lift.bijection}, but moreover ensuring good differentiability of the resulting lift $f$. 

\begin{proof}
As in the proof of Lemma~\ref{lem:lift.bijection} we induct on the ambient dimension $n$, the case $n = 0$ being trivial.

Assume first that $\dim_\RV R = k < n$. 
We perform an additional induction on $k$.
By dimension theory for $\RV$, Theorem~\ref{thm:dim.theory.RV}(\ref{it:coordinate.projections}), we may assume after a $\emptyset$-definable partition and a coordinate permutation that $R$ and $S$ have finite-to-one projection onto the first $k$ coordinates. Let $R', S'\subset (\RV^\times)^k \times \RV^{m+n-k}$ be the sets $R,S$ but considered as objects in $\RV[k]$. By induction on the ambient dimension, there exists a $\emptyset$-definable bijection $g\colon \LL R'\to \LL S'$ lifting $h$ with all of the desired properties. It is of the form
\begin{align*}
	g \colon &\LL R'\subset K^k \times (\RV^\times)^{n-k}\times \RV^m 
		\to   \LL S'\subset K^k \times (\RV^\times)^{n-k}\times \RV^m \\
	&(x,\xi_2,\zeta) 
			\mapsto (g_0(x, \xi_2, \zeta), g_1(x, \xi_2, \zeta), h_1(\rv(x), \xi_2, \zeta)).
\end{align*}
Using Lemma~\ref{lem:lifting.maps.RV[n]}, let $d \colon \LL R'\to K^{n-k}$ be a lift of $g_1$ and $c \colon \LL R'\to K^{n-k}$ be a lift of $(x,\xi, \zeta)\mapsto \xi$.

Take a cell decomposition $\chi \colon K^k \to \RV^N$ adapted to $\LL R'$ with the following property for every twisted box $B$ of $\chi$: for each tuple $(\xi_2,\zeta)$ the induced maps $c(\cdot,\xi_2,\zeta)$ and $d(\cdot,\xi_2,\zeta)$ are $C^1$ on $B$ and the leading terms of these maps and their partial derivatives are constant on $B$.
Let $Z \subset \LL R'$ be the union of all twisted boxes of $\chi$ of dimension strictly less than $k$, and put $R'' = \rv Z$. 
By Theorem~\ref{thm:dim.theory.RV}, $R''$ has $\RV$-dimension at most $k-1$, and so we are done on $R''$ by induction on $k$.

Now define
\begin{align*}
	f \colon  	&\LL (R\setminus R'')\subset K^k\times K^{n-k} \times \RV^m	
				 	\to \LL (S\setminus h(R''))  \\
				&(x,y,\xi)	
					\mapsto \left( g_0(x,\rv(y), \xi), \frac{d(x, \rv(y), \xi)}{c(x,\rv(y), \xi)}  y, h_1(\rv(x), \xi, \zeta)\right),
\end{align*}
where the product in the second component is the componentwise multiplication.

We claim that $f$ is as desired.
Firstly, just as in the proof of Lemma~\ref{lem:lift.bijection}, $f$ is a bijection.
So let us check the desired differentiability properties on an $\rv$-fibre $F=\rv^{-1}(\xi_1, \xi_2)\times\{\eta\}$ where $\xi_1\in \RV^k, \xi_2\in \RV^{n-k}$ and $\eta\in \RV^m$.
Since $g_0,d,c$ are $C^1$ on $\rv^{-1}(\xi_1)$, and for every $x\in \rv^{-1}(\xi_1)$ the map
\[
y\mapsto \frac{d(x, \rv(y), \xi)}{c(x,\rv(y), \xi)} y
\]
is $C^1$, we see that $f$ is $C^1$.
The Jacobian $\Jac f$ is of the form
\[
\Jac (f)(x) = \Jac(g_0)(x) \cdot \prod_{i=1}^{n-k} \frac{d_i(x,\xi_2, \eta)}{c_i(x,\xi_2,\eta)}.
\]
By induction $\rv \Jac(g_0)$ is constant, while our construction shows that $\rv d_i$ and $\rv c_i$ are constant. 
So we indeed obtain that $\rv \Jac f$ is constant on $F$.
Write $\xi_{2} = (\xi_{2,1}, \ldots, \xi_{2,n-k})$ and note that the map
\[
y\mapsto \frac{d_i(x, \rv(y), \xi)}{c_i(x,\rv(y), \xi)} y_i
\]
defines a $\emptyset$-definable bijection between $\rv^{-1}(\xi_{2,i})$ and $\rv^{-1}(h_{2,i}(\xi_1, \xi_2, \eta))$. 
Hence we see that
\[
\prod_{i=1}^{n-k} \frac{|d_i(x,\xi_2, \eta)|}{|c_i(x,\xi_2,\eta)|} = |g_1(\xi_1, \xi_2, \eta)|/|\xi_2|.
\]
Combining this with the computation of $\Jac(f)$ above, we see that
\[
|\Jac f| = |h(\xi_1, \xi_2, \eta)| / |(\xi_1, \xi_2)|.
\]

Now assume that $\dim_\RV R = n$. 
By lemma~\ref{lem:lift.bijection}, there exist a bijective lift $f$ of $h$.
Similar to the proof of that lemma, it now suffices to prove that for a generic point $(\xi,\zeta) \in R$, the desired differentiability properties hold for $f$.
Note that on $\rv^{-1}(\xi)$, $f_{\zeta}$ is automatically $C^1$ and $\rv(\Jac(f_{\zeta}))$ is constant. 
Indeed, $\xi$ is full-dimensional and there exists a cell decomposition such that these properties hold on each twisted box.

It remains to verify that
\[
\abs{\Jac f_{\zeta}(x)} = \abs{(h_0(\xi,\zeta))} / \abs{\xi}.
\]
for all $x \in \rv^{-1}(\xi)$.
For this, fix $i,j\in \{1, \ldots, n\}$. 
We consider $f_j$ when fixing all but the $i$-th coordinate, yielding a map
\[
 \rv^{-1}(\xi_i)\to \rv^{-1}(h_j(\xi)).
\]
Since $\xi$ is full-dimensional each $\xi_i$ is as well (and remains so after fixing all but the $i$-th $K$-coordinate in $\rv^{-1}(\xi)$).
Hence, this map is $C^1$ and has the Jacobian property, yielding that
\[
|\partial_i f_j(x)|\leq \abs{h_j(\xi))} / \abs{\xi_i}.
\]
We thus obtain that $\abs{\Jac f(x)} \leq \abs{h(\xi)} / \abs{\xi}$. 
A similar argument applied to $f^{-1}$ gives the other inequality.
\end{proof}

\begin{notn}
	Let $R,S, f$ be as in the lemma above.
	Given $y = (x,\zeta) \in \LL R \subset K^n \times \RV^m$, write
	\[ \Jac f(y) \coloneqq \Jac f_{\zeta}(x). \]
\end{notn}

We can now define $\RV$-Jacobians.

\begin{defn}[$\RV$-Jacobian]
Let $R, S$ be objects in $\RV[n]$ and let $h\colon R\to S$ be a $\emptyset$-definable bijection. 
Let $f\colon \LL R\to \LL S$ be a map as in the previous lemma. 
Then the map
\[
\Jac_\RV h\colon R\to \RV\colon \xi\mapsto \rv(\Jac(f(x))) \text{ for } \rv(x)=\xi
\]
is called an \emph{$\RV$-Jacobian} of $h$.
\end{defn}

Note that the definition of an $\RV$-Jacobian makes sense without effectivity, and that $\RV$-Jacobians are by definition $\emptyset$-definable. 
Lemma~\ref{lem:lift.bijections.mu} shows that in effective theories, $\RV$-Jacobians always exist.
We stress that a Jacobian of $h$ is not necessarily unique, although the following shows that it is well-defined on generic points.
Recall that we are working in (sufficiently satured) models of equicharacteristic zero effective 1-h-minimal theories so that $\RV$-dimension and generic points are well-defined.

\begin{lem}\label{lem:RV.Jacobian.almost.everywhere}
Let $R,S$ be objects in $\RV[n]$ and let $h \colon R\to S$ be a $\emptyset$-definable bijection with $\RV$-Jacobians $h_1', h_2'\colon R\to \RV$. 
Then there exists a $\emptyset$-definable $R'\subset R$ of $\RV$-dimension at most $n-1$ such that for $\xi \in R\setminus R'$ we have $h_1'(\xi) = h_2'(\xi)$.

Equivalently, if $\dim_\RV R = n$, then for $\xi \in R$ generic, $h'(\xi)$ is well-defined independent of choice of $\RV$-Jacobian.
\end{lem}

\begin{proof}
Write $R \subset (\RV^{\times})^n \times \RV^m$.
Since the lemma is vacuous if $\RV$-dimension less than $n$, we may assume that $\dim_\RV R = n$.
Let $(\xi,\zeta) \in R$ be a generic point of $R$, and let $f,g$ be lifts of $h$ as in Lemma~\ref{lem:lift.bijections.mu}.
Since suffices to show that $\rv(\Jac(f \circ g^{-1})) = 1$ on the box above $\xi$, we may as well assume that $g,h$ are the identity maps on their respective domains.

As in the proof of the previous lemma, the Jacobian property for the coordinate functions $f_j$ when fixing all but the $i$-th coordinate yields that
\[ \abs{\partial_i (f_j(x) - x_j) } < \abs{x_j}/\abs{x_i} \]
It follows that $\rv(\partial_i f_i (x)) = 1$ and for each nontrivial permutation $\sigma$ of $\{1,\dots,n\}$, it holds that 
\[ \prod_{i = 1}^n \abs{ \partial_i f_{\sigma(i)}} < 1 .\]
We conclude that $\rv(\Jac(f)) = 1$, as desired.
\end{proof}

\begin{example}
The conclusion of the above lemma cannot be strengthened to hold for a $\emptyset$-definable $R'\subset R$ whose $\RV$-dimension is strictly less than that of $R$. 
Indeed, let $R = \{1\}$ in $\RV[1]$ and consider the $\emptyset$-definable bijection $\id \colon R\to R$.
Then the following maps are both $\emptyset$-definable lifts of $h$
\[
1+\cM_K\to 1+\cM_K\colon x\mapsto x, \quad 1+\cM_K\to 1+\cM_K\colon x\mapsto 2-x.
\]
Clearly, these yield different $\RV$-Jacobians on $R$.
\end{example}

\subsection{Volume forms and integration}\label{ssec:volume.forms}

We now introduce the relevant categories of definable sets with volume forms in the valued field and the $\RV$-sort. 
Morphisms in this category will be defined only almost everywhere.
We will define several variants of these categories and their Grothendieck rings, e.g.\ with $\RV$-valued volume forms or with $\Gamma$-valued volume forms, or variants where we only consider bounded definable sets.

For the following definition, note that if $X \subset K^n \times \RV^m$ is definable with finite-to-one projection onto $K^n$, then it is in bijection with a subset of $K^{n + m}$, by effectivity.
Hence, the characterizations of Theorem~\ref{thm:dim.theory.VF} extend to a notion of dimension on $X$.

\begin{defn}\label{def:almost.everywhere}
\begin{enumerate}
\item Let $X\subset K^n\times \RV^m$ be $\emptyset$-definable, with finite-to-one projection onto $K^n$. 
We say that a property holds for \emph{almost all $x\in X$} if it holds for all $x$ outside some $\emptyset$-definable subset $Z\subset X$ with $\dim Z < n$.
\item Let $X,Y\subset K^n\times \RV^m$ be $\emptyset$-definable, with finite-to-one projections onto $K^n$. 
A $\emptyset$-definable \emph{essential bijection} $f\colon X\to Y$ is a $\emptyset$-definable bijection $f\colon X_0\to Y_0$ such that $X_0\subset X, Y_0\subset Y, \dim (X\setminus X_0) < n$ and $\dim (Y\setminus Y_0) < n$.
\end{enumerate}
\end{defn}

Note that in the above definition both items depend on the ambient dimension.
The reason is that for example a point in ambient dimension zero should have non-zero measure, but a point in ambient dimension one certainly has measure zero.

For $R\subset (\RV^\times)^n\times \RV^m$ an object in $\RV[n]$ and $(\xi, \eta)\in R$ we define $|(\xi, \eta)| = \prod_i |\xi_i|$.

\begin{defn}
\begin{enumerate}
\item Denote by $\mu \VF[n]$ the category of pairs $(X,\omega)$, where $X\subset \VF^n\times \RV^m$ is a $\emptyset$-definable set for which the projection $X\to \VF^n$ is finite-to-one and $\omega\colon X\to \RV^\times$ is a $\emptyset$-definable map. A morphism between $(X,\omega)$ and $(Y, \rho)$ is a $\emptyset$-definable essential bijection such that for almost all $x\in X$ we have that
\[
\omega(x) = \rho(f(x)) \rv(\Jac f(x)).
\]
\item Denote by $\mu_\Gamma \VF[n]$ the category of pairs $(X,\omega)$, where $X\subset \VF^n\times \RV^m$ is a $\emptyset$-definable set for which the projection $X\to \VF^n$ is finite-to-one and $\omega\colon X\to \Gamma^\times$ is a $\emptyset$-definable map. A morphism between $(X,\omega)$ and $(Y, \rho)$ is a $\emptyset$-definable essential bijection such that for almost all $x\in X$ we have that
\[
\omega(x) = \rho(f(x))|\Jac f(x)|.
\]
\item Denote by $\mu \RV[n]$ the category of pairs $(R, \omega)$, where $R\subset (\RV^\times)^n\times \RV^m$ is a $\emptyset$-definable set for which the projection $R\to (\RV^\times)^n$ is finite-to-one and $\omega\colon R\to \RV^\times$ is a $\emptyset$-definable map. A morphism between $(R, \omega)$ and $(S, \rho)$ is a $\emptyset$-definable bijection $h\colon R\to S$ such that there exists a Jacobian $\Jac_\RV h$ of $h$ with
\[
\omega(\xi) = \rho(h(\xi)) \Jac_\RV h(\xi) \text{ for all }\xi\in R.
\]
\item Denote by $\mu_\Gamma \RV[n]$ the category of pairs $(R, \omega)$, where $R\subset (\RV^\times)^n\times \RV^m$ is a $\emptyset$-definable set for which the projection $R\to (\RV^\times)^n$ is finite-to-one and $\omega\colon R\to \Gamma^\times$ is a $\emptyset$-definable map. A morphism between $(R, \omega)$ and $(S, \rho)$ is a $\emptyset$-definable bijection $h\colon R\to S$ such that 
\[
\omega(\xi)|\xi| = \rho(h(\xi))|h(\xi)| \text{ for all }\xi\in R.
\]
\end{enumerate}
\end{defn}

\begin{remark}
In many natural effective theories, and in fact for all examples from Section~\ref{sec:eff}, there is a natural notion of differentiation on the residue field.
This notion of derivatives extends to definable maps on $\RV$, and leads to a different more intrinsic notion of $\RV$-Jacobians.
This will be explained in detail in Section~\ref{sec:intrinsic.jac}, where we will also show that in effective theories this leads to an alternative, but equivalent, definition of the categories $\mu \RV[n]$.
In particular, in an effective V-minimal theory, our definition of $\mu \RV[n]$ is the same as that of Hrushovski--Kazhdan~\cite[Def.\,8.13]{HK}.
\end{remark}

Similarly as in the case without measures, we denote by $\mu \RV[\leq n]$ the coproduct of the categories $\mu\RV[i]$ for $i\in \NN$ with $i\leq n$, and $\mu \RV[*]$ as the coproduct of $\mu\RV[i]$ with $i\in \NN$.
Let $\mu \VF[*]$ denote the coproduct of the categories $\mu \VF[n]$. 
Note that we now distinguish between sets in distinct ambient dimensions, compared to $\K_+ \VF$.
Again, the reason is that the ambient dimension plays a role in determining the measure of a definable set.

The map $\omega$ in this definition will be called the \emph{volume form}. Note that all morphisms in these categories are isomorphisms. 
Since we are only interested in the Grothendieck rings of these categories, this is not a restriction.

We also introduce bounded variants of these categories. 
A set $X\subset K^n\times \RV^m$ resp.\ $R\subset \RV^n\times \RV^m$ is said to be \emph{bounded} if its projection onto $K^n$ resp.\, $\RV^n$ is bounded. 
A volume form $\omega\colon X\to \RV^\times$ or $\omega\colon X\to \Gamma^\times$ is called \emph{bounded} if there exists a $\gamma\in \Gamma^\times$ such that for all $x\in X$ we have $|\omega(x)|\leq \gamma$.

\begin{defn}
We denote by $\mu_{\bdd}\VF[n], \mu_{\Gamma, \bdd}\VF[n], \mu_{\bdd}\RV[n]$ and $\mu_{\Gamma, \bdd} \RV[n]$ the full subcategories of the corresponding category whose objects are bounded with a bounded volume form.
\end{defn}

If $(X, \omega)$ is an object in $\mu\VF[n]$, then we will denote the class of $(X,\omega)$ in $\K_+ \mu \VF[n]$ by $[X, \omega]$.
If we wish to stress the ambient dimension, we write $[X, \omega]_n$ instead.
If $\omega$ is the constant map $1$, then we also denote the class of $[X, 1]$ by $[X]$ or $[X]_n$.

One can also look at a category consisting of only volumes of definable sets.
In that case, one should impose some relation between measures of definable sets of different dimension.
We impose the condition the $1+\cM_K$ has measure $1$, which is reminiscent of working with the Haar measure on $\QQ_p^\times$.

\begin{defn}
Define $\K_+ \mu^\vol \VF$ to be the quotient $\K_+ \mu \VF[*]$ by the congruence $([\cM_K]_1, [1]_0)$, and
define $\K_+ \mu^\vol \RV$ to be the quotient $\K_+ \mu \RV[*]$ by the congruence $([1]_1, [1]_0)$.
Similarly define $\mu^\vol \RV, \mu^{\vol}_{\Gamma} \VF$ and so on.
\end{defn}

Our goal is again to precisely relate the semirings $\K_+ \mu \VF[*]$ and $\K_+ \mu \RV[*]$, with similar variants for $\mu_\Gamma$, $\mu_{\bdd}$, $\mu_{\Gamma, \bdd}$, $\mu^{\vol}$, $\mu_{\bdd}^\vol$, $\mu_{\Gamma}^\vol$ and $\mu_{\Gamma, \bdd}^\vol$. 

\begin{defn}
Let $(R,\omega)$ be an object in $\mu\RV[n]$. Then we define the object $\LL (R, \omega)$ in $\mu\VF[n]$ as $(\LL R, \LL \omega)$ where
\[
(\LL \omega)(x,\xi) = \omega(\rv(x), \xi).
\]
If $(R, \omega)$ is an object in any of the other categories $\mu_{\Gamma}\RV[n]$, $\mu_{\bdd}\RV[n]$, $\mu^\vol\RV[n]$ then $\LL(R, \omega)$ lands in the corresponding category over $\VF$.
We call $\LL$ the \emph{lifting map}.
\end{defn}

As in the non-measured case, we will first show that for effective theories $\LL$ descends to a well-defined surjective morphism
\[
\LL\colon \K_+ \mu \RV[n]\to \K_+ \mu \VF[n].
\]
The kernel of $\LL$ differs slightly in the measured case. Indeed, the translation map $\cM_K\setminus \{0\} \to 1+\cM_K$ is an essential bijection with Jacobian $1$ and so the elements $[\RV^\times_{\leq 1}]_1$ and $[1]_1$ have the same lift. Surprisingly, this relation generates the entire kernel of $\LL$.

\begin{defn}
Denote by $I_\mathrm{sp}^\mu$ the congruence in $\K_+ \mu \RV[*]$ or in $\K_+ \mu_\Gamma \RV[*]$ generated by $([\RV^\times_{< 1}]_1, [1]_1)$. 
Write $I_{\mathrm{sp}}^\mu[n]$ for the induced semigroup congruence on $\K_+\mu\RV[n]$, or simply $I_{\mathrm{sp}}^\mu$ when $n$ is clear from the context. 
\end{defn}

We show that $I_\mathrm{sp}^\mu$ interacts well with the bounded category.

\begin{lem}\label{lem:Ispmu.bdd}
The congruence in $\K_+ \mu_{\bdd} \RV[*]$ generated by $I_\mathrm{sp}^\mu\cap \K_+ \mu_{\bdd} \RV[*]$ is identical to the congruence generated by $([\RV^\times_{< 1}]_1, [1]_1)$. The same holds in $\K_+ \mu_{\Gamma, \bdd} \RV[*]$.
\end{lem}

On the level of Grothendieck rings, this lemma is stating that
\[
(([\RV^\times_{< 1}]_1 - [1]_1)\K \mu \RV[*])\cap \K \mu_{\bdd}\RV[*] = ([\RV^\times_{< 1}]_1 - [1]_1)\K \mu_{\bdd} \RV[*].
\]

\begin{proof}
Assume that $(R, \omega)$ and $(S,\rho)$ are objects in $\K_+ \mu_{\bdd} \RV[n]$ which are equivalent modulo $I_\mathrm{sp}^\mu$ in $\K_+ \mu \RV[*]$. 
Then there exist objects $A, B, C$ in $\mu \RV[*]$ such that
\begin{align*}
[R, \omega] &= [A] + [B][\RV_{<1}^\times]_1 + [C] [1]_1, \\
[S, \rho] &= [A] + [B][1]_1 + [C][\RV_{<1}^\times]_1
\end{align*}
in $\K_+ \mu \RV[*]$. 

Note that $A$ is isomorphic to a sub-object of $[R, \omega]$, and hence $[A]$ is an element of $\K_+ \mu_{\bdd} \RV[n]$. Since $[C][1]_1$ is isomorphic to a sub-object of $(R,\omega)$, $C$ itself is isomorphic to an object in $\mu_{\bdd} \RV[n]$. Similarly, $[B][1]_1$ is a sub-object of $(S, \rho)$, and $B$ is isomorphic to an object in $\mu_{\bdd}\RV[n]$. But then $[R,\omega]$ and $[S, \rho]$ are equivalent modulo the congruence generated by $([\RV^\times_{<1}]_1, [1]_1)$ in $\K_+ \mu_{\bdd}\RV[*]$ itself.

The proof for $\K_+ \mu_{\Gamma, \bdd}\RV[*]$ is identical.
\end{proof}

\subsection{Lifting bijections, surjectivity and the kernel}

In this subsection we prove that $\LL$ is well-defined on the level of Grothendieck semigroups, that it is surjective, and compute the kernel. This is in large part similar to the work done in the non-measured case.
Recall that we work with a sufficiently saturated model $K$ of an effective 1-h-minimal theory $\cT$.

\begin{lem}
Let $(R, \omega),(S, \rho)$ be objects in $\mu \RV[n]$ and let $h\colon (R, \omega)\to (S, \rho)$ be an isomorphism between them. Then there exists an isomorphism $f\colon \LL(R, \omega)\to \LL(S,\rho)$ in $\mu \VF[n]$. This holds even if $\cT$ is not effective.

In particular, $\LL$ induces a well-defined morphism
\[
\LL\colon \K_+\mu \RV[n]\to \K_+\mu \VF[n].
\]
The same holds for the bounded category.

If $\cT$ is effective, then the same holds for the categories $\mu_{\Gamma}$ and $\mu_{\Gamma, \bdd}$.
\end{lem}

\begin{proof}
By definition of the morphisms in the category $\mu \RV[n]$, there exists a $\emptyset$-definable map $f\colon \LL R\to \LL S$ with properties as in Lemma~\ref{lem:lift.bijections.mu}, and such that for all $x\in \LL R$ we have
\[
\omega(\rv(x)) = \rho(h(\rv(x))\rv(\Jac f(x)).
\]
This shows exactly that $f$ is an isomorphism $(\LL R, \omega)\to (\LL S, \rho)$. Since $\mu_\bdd\RV[n]$ is a full subcategory, the same automatically holds there.

So assume that we are working in $\mu_{\Gamma}\RV[n]$ and that the theory is effective. Then by Lemma~\ref{lem:lift.bijections.mu}, there exists a $\emptyset$-definable bijection $f\colon \LL R\to \LL S$ such that for $x\in \LL R$ we have
\[
|\Jac f(x)| = |h(\rv(x))| / |\rv(x)|.
\]
But this says exactly that $f$ is in fact a morphism $\LL(R, \omega)\to \LL (S, \rho)$ in the category $\mu_\Gamma \VF[n]$. As above, the bounded category is a full subcategory, so the result also follows there.
\end{proof}

As for surjectivity, we have the following.

\begin{lem}\label{lem:LL.surj}
Let $(X,\omega)$ be an object in $\mu \VF[n]$. Then there exists an object $(R,\rho)$ in $\mu \RV[n]$ such that $\LL(R, \rho)$ is isomorphic to $(X, \omega)$. The same holds for any of the other categories.

In particular, $\LL\colon \K_+\mu \RV[n]\to \K_+\mu \VF[n]$ is surjective, and similarly for the other categories.
\end{lem}

\begin{proof}
The proof is nearly identical to Lemma~\ref{lem:L.surjective}, except that one takes the initial cell decomposition $\chi$ in that proof refining the volume form $\omega$. 
Note that in any of the other categories, the resulting proof remains within that category.
\end{proof}

For the kernel, we follow the same approach as in Sections~\ref{sec:kernel.of.LL.dim.one} and~\ref{sec:int.higher.dim}. 

\begin{lem}\label{lem:kernel.mu.dim.1}
The kernel of the morphism 
\[
\LL\colon \K_+\mu\RV[1]\to \K_+\mu \VF[1]
\]
is the congruence $I_{\mathrm{sp}}^{\mu}[1]$.

A similar statement holds for $\mu_\Gamma, \mu_{\bdd}$ and $\mu_{\Gamma, \bdd}$.
\end{lem}

\begin{proof}
The proof is essentially identical to the non-measured case as in Section~\ref{sec:kernel.of.LL.dim.one}. 
The only change is that if $f\colon \LL R\to \LL S, g\colon \LL S\to \LL R$ are $\emptyset$-definable bijections, one takes the initial cell decompositions $\chi_0, \psi_0$ for compatible domain-image preparation such that $f$ and $g$ have the Jacobian property on the twisted boxes of $\chi_0, \psi_0$ and such that the volume forms are constant on the twisted boxes of the cell decomposition.

For the category $\mu_\Gamma$ the proof is identical, while the bounded variants follow from the fact that these are full subcategories.
\end{proof}

In higher dimensions, we again first show that maps can be made relatively unary.

\begin{lem}\label{lem:unique.volume.form}
Let $(X,\omega)$ be an object in $\mu\VF[n]$ and let $f\colon X\to Y$ be a $\emptyset$-definable bijection, where $Y\subset K^n$. 
Then there exists a $\emptyset$-definable map $\rho\colon Y\to \RV^\times$ such that for almost all $x\in X$ one has
\[
\omega(x) = \rho(f(x))\rv((\Jac f)(x)).
\]
Moreover, if $\rho'$ is another such map then $\rho=\rho'$ almost everywhere.
\end{lem}

In other words, this lemma says that there is an essentially unique volume form $\rho$ on $Y$ such that $f\colon (X, \omega)\to (Y, \rho)$ becomes a morphism in $\mu \VF[n]$.
Note that this lemma also implies a similar version for $\mu_\Gamma \VF[n]$.

\begin{proof}
Let $y\in Y$. If $f$ is $C^1$ at $x=f^{-1}(y)$ with non-zero Jacobian, we define
\[
\rho(y) = \omega(x)/\rv( (\Jac f)(x)). 
\]
The unicity of $\rho$ is clear. 
\end{proof}

\begin{lem}[Relatively unary maps]\label{lem:rel.unary.mu}
Let $f\colon (X, \omega)\to (Y, \rho)$ be a morphism in $\mu \VF[n]$. Then there exists a finite partition of $X$ into $\emptyset$-definable pieces $X_i$ such that the restriction of $f$ to $(X_i, \omega)$ is a finite composition of relatively unary morphisms in $\mu\VF[n]$. 

The same holds in $\mu_\Gamma \VF[n]$.
\end{lem}

\begin{proof}
It suffices to consider the case where $f$ is a bijection. 
The lemma now follows by applying Lemma~\ref{lem:rel.unary} to write $f$ as a finite composition of relatively unary maps, and using Lemma~\ref{lem:unique.volume.form} to turn the resulting maps into morphisms in $\mu \VF[n]$ or $\mu_\Gamma \VF[n]$. 
\end{proof}

Let $n$ be a positive integer. 
For subsets $s \subset \{1,\ldots,n\}$ write $K^{s}$ for the set of functions from $s$ to $K$.
There is an obvious bijection to $K^{\# s}$, but the presentation $K^{s}$ additionally keeps track of an `index' for each coordinate.
The set $\RV^s$ is understood in the same way.

\begin{defn}
Let $n$ be a positive integer and let $k, \ell\subset \{1, \ldots, n\}$ form a partition of $\{1, \ldots, n\}$. Denote by $\mu \VFR[k, \ell]$ the category whose objects are pairs $(X, \omega)$ where $X\subset K^k\times \RV^\ell\times \RV^m$ is $\emptyset$-definable such that the projection onto $K^k\times \RV^\ell$ is finite-to-one, and such that $\omega \colon X\to \RV$ is $\emptyset$-definable and factors through $\RV^\ell\times \RV^m$. 
A morphism $(X, \omega)\to (Y, \rho)$ is a $\emptyset$-definable bijection $f\colon X\to Y$ such that for every $a\in K^k$ the map $f_a\colon (X_a, \omega)\to (Y_a, \omega)$ is a morphism in the category $\mu \RV[\ell]$ in the language $\cL(a)$.

Similarly define the categories $\mu_{\Gamma}\VFR[k, \ell], \mu_{\bdd}\VFR[k, \ell]$ and $\mu_{\Gamma, \bdd}\VFR[k, \ell]$.
\end{defn}

One should think of an element in $\mu \VFR[k, \ell]$ as a motivic function $K^k\to \mu \RV[\ell]$.
In particular, we have that $\mu \VFR[0, \ell]$ is simply $\mu \RV[\ell]$.

\begin{remark}
For technical reasons, we really need to work with subsets $k, \ell\subset \{1, \ldots, n\}$ rather than just integers summing up to $n$ as in the non-measured case.
The reason is that the order of coordinates in $K^k\times \RV^\ell$ is important, and we need to keep track of this order when integrating out several of the variables.
For a concrete example, consider the object $(\rv^{-1}(\xi_1)\times \rv^{-1}(\xi_2), 1)$ in $\mu \VFR[2, 0]$ and suppose that we work simply with integers $k$ and $\ell$.
Integrating the second coordinate and then the first yields the object
\[
((\xi_1, \xi_2), 1)
\] 
in $\mu \VFR[0,2]$, while first integrating the second coordinate and then the first yields
\[
((\xi_2, \xi_1), 1).
\]
These objects might not be isomorphic, since the coordinate permutation $(\xi_1, \xi_2)\mapsto (\xi_2, \xi_1)$ has $\RV$-Jacobian $-1$.
Note that this is not an issue in the category without measures, since this coordinate permutation is certainly a $\emptyset$-definable map.
Also, in the corresponding category $\mu_\Gamma \VFR[k, \ell]$ there is no issue, since the norm of the $\RV$-Jacobian is $1$.
\end{remark}

For the rest of this section we simply work with the category $\mu \VFR[k, \ell]$ and note that the proofs works just as well for the variants $\mu_\Gamma \VFR[k, \ell], \mu_\bdd \VFR[k, \ell]$ and $\mu_{\Gamma, \bdd} \VFR[k, \ell]$.
When dealing with the bounded category, one should use Lemma~\ref{lem:Ispmu.bdd} to control the congruence $I_\mathrm{sp}^\mu$.

If $(X, \omega)$ is an object in $\mu \VFR[k, \ell]$, then for $a\in K^k$, the fibre $(X_a, \omega|_{X_a})$ is an object in $\RV[\ell]$ in the expanded language $\cL(a)$. Define the equivalence relation $I_{\mathrm{sp}}^\mu$ on $\mu \VFR[k, \ell]$ by saying that $(X,Y)\in I_{\mathrm{sp}}^\mu$ if for every $a\in K^k$ we have that $(X_a, Y_a)\in I_{\mathrm{sp},a}^\mu$.

\begin{lem}[$I_\mathrm{sp}^\mu$ in families]
Let $R,S$ be in $\mu \RV[\ell]$ and let $f\colon R\to \RV^N, g\colon S\to \RV^N$ be $\emptyset$-definable. Assume that for every $\eta\in \RV^N$ we have that $(f^{-1}(\eta), g^{-1}(\eta))\in I_{\mathrm{sp}, \eta}^\mu$. Then $(R,S)\in I_{\mathrm{sp}}^\mu$.
\end{lem}

\begin{proof}
This follows from compactness, similar to Lemma~\ref{lem:compactness.Isp}.
\end{proof}

The map $\LL\colon \mu\RV[1]\to \mu \VF[1]$ induces an isomorphism
\[
I_\mu\colon \K_+\mu \VF[1]\to \K_+\mu \RV[1]/I_\mathrm{sp}^\mu.
\]
We define a relative version of $I_\mu$ completely analogous to the non-measured case. 
If $j\in k$, then we may apply this construction of $I_\mu$ uniformly to obtain a morphism
\[
I^j_\mu\colon \K_+ \mu\VFR[k, \ell]\to \K_+ \mu \VFR[k\setminus \{j\}, \ell\cup\{j\}]/I_\mathrm{sp}^\mu
\]
which essentially integrates out the $j$-th variable of $X$.

If $X\subset K^k\times \RV^\ell\times \RV^m$ is an object in $\VFR[k, \ell]$ and $(V_x, \omega_x)_{x\in X}$ is a $\emptyset$-definable family of pairs of objects in $\mu \VFR[k', \ell']$ in the language $\cL(x)$, then we define
\[
\sum_{x\in X} [V_x, \omega_x] \coloneqq [V,\omega] \in \mu \VFR[k\sqcup k', \ell \sqcup \ell']
\]
by 
\[
V = \{(x,y)\in K^k\times \RV^\ell\times \RV^m\times K^{k'}\times \RV^{\ell'}\times \RV^{m'}\mid x\in X, y\in V_x\}
\]
and $\omega(x,y) = \omega_x(y)$. 
This sum is well-defined: for each family $(V_x',\omega_x')_x$ such that for each $x$ one has $[V_x',\omega_x'] = [V_x,\omega_x]$ it holds that $\sum_{x} [V_x',\omega_x'] = \sum_x [V_x,\omega_x]$. 
Indeed, by compactness and dimension theory in $K$ there is a $\emptyset$-definable essential bijection between $(V',\omega')$ and $(V,\omega)$.

For notational simplicity, we shall also denote the classes $[X,\omega] \in \K_+\mu\VFR[k,\ell]$ simply by $X$ in the next few lemmas, if no confusion is possible.

\begin{lem}[Integration, $\RV$ and $\K_+$] \leavevmode
\begin{enumerate}
\item Let $(X,\omega)\in \Ob \mu \VFR[k, \ell]$ and let $f\colon X\to \RV^N$ be $\emptyset$-definable. Then
\[
I_\mu^j X = \sum_{\xi\in \RV^N} I_\mu^j (f^{-1}(\xi), \omega|_{f^{-1}(\xi)}).
\]
\item Let $(X,\omega), (Y, \rho)$ be objects in $\mu \VFR[k, \ell]$. If $X=Y$ in $\K_+\mu\VFR[k, \ell]$ then $I^j_\mu X = I^j_\mu Y$ in $\K_+\mu \VFR[k\setminus \{j\}, \ell\cup \{j\}]/I_\mathrm{sp}^\mu$.
\item Let $(X,\omega), (Y, \rho)$ be objects in $\mu \VFR[k, \ell]$. If $X=Y$ in $ \K_+\mu\VFR[k, \ell]/I_\mathrm{sp}^\mu$ then also $I^j_\mu X = I^j_\mu Y$ in $\K_+\mu \VFR[k\setminus \{j\}, \ell\cup\{j\}] /I_\mathrm{sp}^\mu$.
\end{enumerate}
\end{lem}

\begin{proof}
This is identical to the non-measured case, Lemmas~\ref{lem:integration.commutes.rv}, \ref{lem:I^j.well-defined} and~~\ref{lem:int.Isp}.
\end{proof}

\begin{lem}[Fubini]
Let $j,j'\in k$ be distinct. Let $(X,\omega)$ be an object in $\mu \VFR[k, \ell]$. Then $I^j_\mu I^{j'}_\mu X = I^{j'}_\mu I^j_\mu X$ in $\K_+ \mu \VFR[k\setminus \{j,j'\}, \ell\cup\{j,j'\}]/I_\mathrm{sp}^\mu$.
\end{lem}

\begin{proof}
Following the same strategy as in the proof of Lemma~\ref{lem:Fubini} and using Lemma~\ref{lem:bi.twisted}, we can reduce to the case that $X$ is an $A$-definable bi-twisted box for some $A\subset K\cup \RV$, say of the form
\[
X = \{(x,y)\in K^2\mid \rv(x-c_1) = \xi_1, \rv(y-c_2(x_1)) = \xi_2\},
\]
and $(j,j') = (1,2)$.
We may moreover assume that $\omega$ is equal to some constant $\omega_X$ on $X$. 

Firstly, we have that $I^1_\mu I^{2}_\mu X = [(\xi_1, \xi_2), \omega_X]$ in $\K_+\mu \RV_A[2]$. If $X$ is a box in $K^2$, then similarly $I^{2}_\mu I^1_\mu X = [(\xi_1, \xi_2), \omega_X]$. So assume that $X$ is not a box. Then Lemma~\ref{lem:cell.decomp.bi.twisted} shows that 
\[
X = \{(x,y)\in K^2\mid y\in Y, \rv(x-c_2^{-1}(y)) = \xi_3\},
\]
where $\xi_3$ is $\dcl_{\RV}(A)$-definable and $Y\subset K$ is the projection of $X$ onto the second coordinate. Then $ I^1_\mu X = (\{\xi_3\}\times Y, \omega_X)$, where $\xi_3 = -\rv(c_2')^{-1}\xi_2$. 
The map $c_2$ provides an $A$-definable bijection between $\rv^{-1}(\xi_1)+c_1$ and $Y$. Since $c_2$ has the Jacobian property on $\rv^{-1}(\xi_1)+c_1$, we have in $\K_+\mu \VF_{A\xi_3}[1]$ that
\[ [Y,\omega_X] = [\rv^{-1}(\xi_1),\rv(c_2') \omega_X ].  \]
Working relative to the parameters $A\xi_3$, we see that
\[ I_\mu [Y,\omega_X] =  [\xi_1, \rv(c_2') \omega_X ] \in \K_+ \RV_{A\xi_3}[1]/I_{\mathrm{sp}}^\mu. \]
Now using that $\xi_3 = -\rv(c_2')^{-1} \xi_2$, it follows that
\[ I^{2}_\mu I^{1}_\mu [X,\omega_X] = I^{2}_\mu [ \{\xi_3\} \times Y, \omega_X]  = [(\xi_3,\xi_1),\rv(c_2') \omega_X ] = [(\xi_1,\xi_2),\omega_X], \]
as desired.
%
\end{proof}

Finally, we compute the kernel of $\LL$.

\begin{lem}[The kernel of $\LL$]\label{lem:kernel.LL.mu}
Let $(R, \omega), (S, \rho)$ be objects in $\mu \RV[n]$ and assume that $\LL(R, \omega)$ and $\LL(S,\omega)$ are isomorphic in $\mu \VF[n]$. Then $([R, \omega], [S, \omega])$ is in $I_\mathrm{sp}^\mu$. 
\end{lem}

\begin{proof}
The argument is identical to Lemma~\ref{lem:kernel.LL}. One first uses Lemma~\ref{lem:rel.unary.mu} to reduce to the relatively unary case, at which point one can use Fubini to compute the integral of $\LL (R, \omega)$ and $\LL (S, \rho)$. 
\end{proof}

\subsection{Isomorphisms of Grothendieck semirings with measures}

We deduce isomorphisms of Grothendieck rings similar to the non-measured case.

\begin{theorem}\label{thm:isom.K.mu}
There is an isomorphism of graded semirings
\[
\int\colon \K_+ \mu \VF[*]\to \K_+\mu \RV[*]/I_\mathrm{sp}^\mu,
\]
such that $\int(\LL(R, \omega)) = (R, \omega)$ for any object $(R, \omega)$ in $\mu \RV[n]$. 
\end{theorem}

\begin{proof}
We define $\int$ as the inverse of $\LL$. By Lemma~\ref{lem:LL.surj} $\LL$ is surjective, while Lemma~\ref{lem:kernel.LL.mu} shows that its kernel is exactly $I_\mathrm{sp}^\mu$.
\end{proof}

Note that the above theorem does not need effectivity.
This is simply because the morphisms in $\RV[n]$ are liftable by definition.

We obtain similar results for the variants.
\begin{theorem}\label{thm:isom.K.mu.variants}
The map $\int$ induces the following surjective morphisms of graded semirings, which are bijections if $\cT$ is effective:
\begin{align*}
\K_+ \mu_\Gamma \VF[*]&\to \K_+ \mu_\Gamma \RV[*]/I_\mathrm{sp}^\mu, \\
\K_+ \mu_{\bdd} \VF[*]&\to \K_+ \mu_\bdd \RV[*]/ I_\mathrm{sp}^\mu, \\
\K_+ \mu_{\Gamma,\bdd} \VF[*]&\to \K_+ \mu_{\Gamma,\bdd} \RV[*]/ I_\mathrm{sp}^\mu.
\end{align*}
These in turn induce epimorphisms
\begin{align*}
\K_+ \mu^\vol \VF & \to \K_+ \mu^\vol \RV / I_\mathrm{sp}^\mu, \\
\K_+ \mu^\vol_\Gamma \VF & \to \K_+ \mu^\vol_\Gamma \RV / I_\mathrm{sp}^\mu, \\
\K_+ \mu^\vol_{\Gamma, \bdd} \VF & \to \K_+ \mu^\vol_{\Gamma, \bdd} \RV / I_\mathrm{sp}^\mu. 
\end{align*}
which are again bijective if $\cT$ is effective.
\end{theorem}

\begin{proof}
Clear.
\end{proof}

Using these isomorphisms we can define the motivic volume of a definable set.
\begin{defn}[Motivic volume]
Let $X\subset K^n$ be a $\emptyset$-definable set. Then we define its \emph{motivic volume} as 
\[
\mu(X) = \int [X,1]\in \K_+ \mu^\vol_\Gamma \RV / I_\mathrm{sp}^\mu.
\]
\end{defn}

We conclude this section by showing that if $X$ has bounded motivic volume, then it is in fact isomorphic to a bounded set.
Such a statement makes sense by the following lemma.
\begin{lem} \label{lem:vol.inclusion}
The inclusions
\begin{align*}
	\Ob \mu_{\Gamma,\bdd} \VF[n] &\subset \Ob \mu_{\Gamma} \VF[n], \\
	\Ob \mu_{\Gamma,\bdd}\RV[n] &\subset \Ob \mu_\Gamma\RV[n],
\end{align*}
induce injective semiring morphisms
\begin{align*}
	\K_+ \mu_{\Gamma,\bdd}^\vol\VF &\hookrightarrow \K_+ \mu_{\Gamma}^\vol\VF, \\
	\K_+ \mu_{\Gamma,\bdd}^\vol\RV/I^\mu_{\mathrm{sp}} &\hookrightarrow \K_+ \mu_{\Gamma}^\vol\RV/I_{\mathrm{sp}}^\mu.
\end{align*}    
\end{lem}
\begin{proof}
This follows similarly to the proof of Lemma~\ref{lem:Ispmu.bdd}.
\end{proof}
\begin{prop}
Let $X \subset K^n$ be $\emptyset$-definable. If $\mu(X) \in \K_+ \mu_{\Gamma,\bdd}^\vol\RV/I_{\mathrm{sp}}^\mu$, then $[X,1] \in \K_+ \mu_{\Gamma,\bdd} \VF[n]$. 
\end{prop}
\begin{proof}
By the previous lemma and theorem, there exists some $(Y,\omega) \in \Ob \mu_{\Gamma,\bdd}\VF[*]$ and $A,B,C \in \K_+\mu_{\Gamma}\VF[*]$ such that in $\K_+\mu_{\Gamma} \VF[*]$
\begin{align*}
[X,1] &= A + B[1]_0 + C[\cM_K]_1, \\
[Y,\omega] &= A + B[\cM_K]_1 + C[1]_0.
\end{align*}
Comparing ambient dimensions in the first equation, we find that $A,B \in \K_+\mu_{\Gamma}\VF[n]$ and $C \in {\K_+\mu_{\Gamma}\VF[n-1]}$. 
The second equation yields that $A,B,C$ can be represented by subobjects of bounded objects. 
Hence, the class $[X,1]$ belongs to $\K_+\mu_{\Gamma,\bdd}\VF[n]$.
\end{proof}
%

\subsection{Intrinsic $\RV$-Jacobians}\label{sec:intrinsic.jac}

In familiar theories our notion of $\RV$-Jacobians agrees with the usual notion.
This leads to a more intrinsic definition of the categories $\mu \RV[n]$ from Section~\ref{ssec:volume.forms}.
For V-minimal theories this is in fact how Hrushovski--Kazhdan define $\RV$-Jacobians~\cite[Sec.\,5.5]{HK}. By~\cite[Cor.\,5.23]{HK} their notion coincides with ours.

We first define a notion of derivatives on the residue fields of interest.
Afterwards, we will use translation to define the derivative for maps on $\RV$ itself.
On real closed, almost real closed fields, or $1$-h-minimal fields the notion of derivatives is clear, since these are topological fields.

All other residue fields of interest are algebraically bounded (Definition~\ref{def:alg.bdd}), and so we define a notion of derivative on such a field.
Recall also that a field is algebraically bounded if and only if it is algebraically bounded over $\dcl(\emptyset)$ (\cite[Prop.\,2.17]{JY23}).
Although many algebraically bounded fields are topological in some sense, not all of them are, and so we define a notion of derivative without relying on topology.
The basic idea is to define derivatives via the implicit function theorem.

\begin{defn}\label{def:alg.bdd.der}
Let $k$ be algebraically bounded, let $f\colon k^n\to k$ be a $\emptyset$-definable function and write $\ell = \dcl(\emptyset) \subset k$.
Then we say that $f$ is \emph{differentiable at $a\in k$} (or simply $C^1$ at $a$) if there exists a non-zero polynomial $P\in \ell[x_1, \ldots, x_n,y]$ such that $P$ vanishes on the graph of $f$ and such that
\[
(\partial_y P)(a,f(a))\neq 0.
\]
In that case, we define the \emph{$j$-th partial derivative of $f$ at $a$} to be 
\[
(\partial_j f)(a) = \frac{-(\partial_{x_j} P)(a,f(a))}{(\partial_y P)(a,f(a))}.
\]
\end{defn}

\begin{example}
If $f$ is a polynomial over $\ell$, then one can take $P = y-f(x)$ to see that $\partial_j f(x)$ is just the usual derivative.
\end{example}

We note that our notion of being differentiable at a point may be more restrictive than what one would expect.

\begin{example}
Consider $k=\RR$, which is algebraically bounded in $\cL_\ring(\RR)$, and consider the function
\[
f\colon \RR\to \RR\colon x\mapsto \begin{cases}
x & \text{ if } x < 0, \\
x^2 & \text{ if } x \geq 0.
\end{cases}
\]
If $P\in \RR[x,y]$ is any polynomial vanishing on the graph of $f$, then it must vanish on the Zariski closure of the graph of $f$, and so $(y-x)(y-x^2)$ is a factor of $P$.
But then $P$ is not differentiable at $1$, since $(\partial_y P)(1,1) = 0$.
\end{example}

Our notion of derivative may also be ill-defined.

\begin{example}
Consider again $k = \RR$ which is algebraically bounded in $\cL_\ring(\RR)$, and consider the $\emptyset$-definable function
\[
g\colon \RR\to \RR\colon x\mapsto \begin{cases} 1 & x = 0 \\ 0 & x \neq 0\end{cases}.
\]
Taking $P = y(y-1)$, we see that $f$ is differentiable at $0$ with derivative $0$.
However, we could just as well have taken $P = y(y-x-1)$ which would give $(\partial g)(0) = 1$.
\end{example}

In contrast to the above examples, we will show that definable functions are automatically differentiable away from a lower-dimensional set with well-defined derivative.
This is enough for our purpose, and not an issue when comparing it with $\RV$-Jacobians because of Proposition~\ref{lem:RV.Jacobian.almost.everywhere}.

\begin{remark}
It may be interesting to investigate whether one can use the \'{e}tale-open topology from~\cite{etale-open} to define a more local definition of differentiation.
In the above examples, the \'{e}tale-open topology is simply the usual Euclidean topology on $\RR$, and for this topology the function $f$ is indeed differentiable at $1$, while $g$ is not differentiable at $0$.
\end{remark}

Let $k$ be a field.
Then on $k[x_1, \ldots, x_n,y]$ we consider the monomial order which is the graded lexicographic order with $y > x_1 > \ldots > x_n$.
This determines a notion of leading monomials and leading coefficients for elements of $k[x_1, \ldots, x_n,y]$.
Additionally we obtain a well-founded total preorder $\preceq$ on $k[x_1, \ldots, x_n,y]$ by comparing leading monomials.
We say that an element of $k[x_1, \ldots, x_n, y]$ is \emph{monic} if its leading coefficient with respect to this monomial ordering is $1$.

Also recall that for algebraically bounded fields there is a natural notion of dimension, since $\acl$ satisfies the exchange property.

\begin{lem}\label{lem:alg.bdd.monic.poly}
Let $k$ be algebraically bounded, let $f \colon k^n\to k$ be a $\emptyset$-definable function and write $\ell = \dcl(\emptyset) \subset k$.
Then there exists a unique monic polynomial $P \in \ell[x_1, \ldots, x_n,y]$ which is minimal with respect to $\preceq$ for the following property:
\begin{enumerate}
\item[] There exists a $\emptyset$-definable $Z \subset k^n$ of dimension at most $n-1$ such that for $x\notin Z$ we have that $P(x, \cdot)\in k[y]$ is non-zero and $P(x,f(x)) = 0$.
\end{enumerate}
Moreover, $P$ and $\partial_y P$ are coprime.
\end{lem}

\begin{proof}
Since $k$ is algebraically bounded there exist finitely many polynomials $P_1, \ldots, P_m\in \ell[x_1, \ldots, x_n,y]$ such that for every $x\in k^n$ there is an $i$ such that $P_i(x,\cdot)\in k[y]$ is non-zero and $P_i(x,f(x)) = 0$. 
Note that this uses \cite[Prop.\,2.17]{JY23}.
Then $Q = \prod_i P_i$ has the desired property above, since any $\ell$-definable set is $\emptyset$-definable.
Hence there does indeed exist a polynomial $P\in \ell[x_1, \ldots, x_n, y]$ with the above property which is minimal with respect to $\preceq$.
We can moreover assume that $P$ is monic.

We claim that $P$ is unique, so let $Q$ be another monic polynomial with the same property which is minimal for $\preceq$.
Let $Z$ be as in the property for both $P$ and $Q$. 
Since $P$ and $Q$ have the same leading monomial, it follows by minimality that $P - Q$ is identically zero.

If $P$ and $\partial_y P$ have a common monic factor $G$, then $G^2$ divides $P$.
But then $P/G$ is also as desired, and since $P$ is minimal with respect to $\preceq$ this implies that $G= 1$.
\end{proof}

\begin{lem}\label{lem:alg.bdd.der}
Let $k$ be algebraically bounded, let $f\colon k^n\to k$ be $\emptyset$-definable and write $\ell = \dcl(\emptyset) \subset k$
Then there exists a $\emptyset$-definable set $Z \subset k^n$ such that $\dim(Z) < n$, such that $f$ is differentiable at every $a \notin Z$ and such that $f'(a)$ is independent of the choice of $P$ in Definition~\ref{def:alg.bdd.der}.
\end{lem}

\begin{proof}
Since $k$ is algebraically bounded there exist finitely many $P_1, \ldots, P_m\in \ell[x_1, \ldots, x_n,y]$ such that for every $x\in k^n$ there is an $i$ such that $P_i(x,\cdot)\in k[y]$ is non-zero and $P_i(x,f(x)) = 0$.
Then $P = \prod_i P_i$ is a non-zero polynomial, and we may assume that $P$ and $\partial_y P$ are coprime in $\ell[x,y]$, hence in $k[x,y]$ since $\ell$ is perfect.
Now the set $Z_1$ of $a\in k^n$ for which either $P(a, \cdot) = 0$ in $k[y]$, or $P(a, \cdot)$ and $\partial_y P(a, \cdot)$ are not coprime in $k[y]$, is a proper $\emptyset$-definable Zariski closed subset of $k^n$.
For $a\notin Z_1$, $f$ is differentiable at $a$.

We now show that $f'$ is well-defined away from a lower-dimensional set.
Let $Q$ be minimal with respect to $\preceq$ as in the previous lemma.
Then there exists a $\emptyset$-definable $Z_2\subset k^n$ of dimension at most $n-1$ such that for $a\notin Z_2$, $Q(a, \cdot)\in k[y]$ is non-zero and $Q(a,f(a)) = 0$.
Since $Q$ and $\partial_y Q$ are coprime, there is a $\emptyset$-definable proper Zariski closed set $Z_3\subset k^n$ such that for $a\notin Z_3$ we have that $Q(a, \cdot)\in k[y]$ and $\partial_y  Q(a,\cdot)\in k[y]$ remain coprime.

Take $a\notin Z_1\cup Z_2\cup Z_3$ and let $T \in \ell[x_1, \ldots, x_n, y]$ be such that $T$ vanishes on the graph of $f$ and $\partial_y T(a,f(a)) \neq 0$.
By polynomial division with respect to $\preceq$, there exist $R, G\in \ell[x_1, \ldots, x_n, y]$ with $T = QG + R$ and $R\prec Q$.
We now have for every $b \notin Z_2$ that $R(b,f(b)) = 0$. 
Now either $R(x,y)$ is identically zero, or $R(x,\cdot) \in k[y]$ is nonzero away from a set of dimension at most $n-1$.
But the latter case cannot occur, by minimality of $Q$.
It follows that $R = 0$, whence $Q$ divides $T$.
A straight-forward computation now shows that
\[
\frac{-(\partial_{x_j} T)(a,f(a))}{(\partial_y T)(a,f(a))} = \frac{-(\partial_{x_j} Q)(a,f(a))}{(\partial_y Q)(a,f(a))}. \qedhere
\]
\end{proof}

We will consider three types of residue fields $k$ when defining an intrinsic notion of $\RV$-Jacobian.
Firstly, assume that $k$ is real closed in a language such that the theory of $k$ in $\cL$ is o-minimal. 
Let $f\colon U\subset k^n\to k^n$ be a $\emptyset$-definable map on an open set $U\subset k^n$. 
Then there is an open subset $V\subset U$ for which $\dim(U\setminus V) < n$ and such that $f$ is $C^1$ on $V$. 
We then define $\Jac f\colon V\to k$ as usual.

Secondly, assume that $k$ is a henselian valued field of characteristic zero in a language for which it is $1$-h-minimal. 
If $f\colon U\subset k^n\to k^n$ is a $\emptyset$-definable map on an open set $U\subset k^n$, then by~\cite{CHRV} there exists an open subset $V\subset U$ for which $\dim (U\setminus V) < n$ and such that $f$ is $C^1$ on $V$. 
Then we define $\Jac f\colon V\to k$ in the usual way.

Thirdly, if $k$ is algebraically bounded over $\ell$ in some language $\cL$, and $f\colon k^n\to k^n$ is $\emptyset$-definable, then by Lemma~\ref{lem:alg.bdd.der} there exists some $U\subset k^n$ with $\dim(k^n\setminus U)<  n$ such that $f$ is differentiable on $U$.
We then define $\Jac f\colon U\to k$ using Definition~\ref{def:alg.bdd.der}.

To define the intrinsic $\RV$-Jacobian, we will locally translate $\emptyset$-definable functions $\RV^n \to \RV^n$ to maps $k^n \to k^n$.
In this step we temporarily add additional parameters from $\RV$, so we have to take care that $k$ retains the desired properties. 
A sufficient condition is given by stable embeddedness. Recall that $k$ is \emph{stably embedded} in $\RV$ if each $\RV$-definable subset in $k^n$ is already $k$-definable.
Say that $k$ is \emph{strongly stably embedded} in $\RV$ if for every $\xi \in \RV^m$ each $\xi$-definable subset $X \subset k^n$ is already $\dcl(\xi) \cap k$-definable.

\begin{lem}
	Let $K$ be a 1-h-minimal field of equicharacteristic zero in some language $\cL$. Consider the residue field $k$ with its induced structure and let $\xi \in \RV_K^n$.
	\begin{enumerate}
		\item Assume that $k$ is stably embedded in $\RV$ and that $k$ is either o-minimal or algebraically bounded. Then $k$ remains respectively o-minimal or algebraically bounded, for the induced structure from $\cL(\xi)$.
		\item The same holds if $k$ is 1-h-minimal and strongly stably embedded in $\RV$.
	\end{enumerate}
\end{lem}
\begin{proof}
	Throughout this proof, `definable' and `$\emptyset$-definable' refer to definability with respect to $\cL$.
	
	Suppose first that $k$ is o-minimal. Since every $\xi$-definable subset of $k$ is already $k$-definable, it is still a finite union of points and intervals.
	
	Now assume that $k$ is algebraically bounded. Let $X \subset k^{n+1}$ be a $\xi$-definable subset. We need to show that there exist nonzero polynomials $P_1(x,y),\dots,P_n(x,y) \in k[x,y]$ such that whenever $X_{a} \subset k$ is finite for $a \in k^n$, there is some $i$ such that $P_i(a,y) \in k[y]$ is nonzero and vanishes on $X_a$.
	By stable embeddedness, there exists some $\emptyset$-definable $Y \subset k^{m+n+1}$ and $b \in k^m$ such that $X = Y_b$. By algebraic boundedness, there exist polynomials $Q_i(u,x,y) \in k[u,x,y]$ such that the $Q_i(b,x,y)$ are as desired.
	
	Finally, assume that $k$ is 1-h-minimal and strongly stably embedded in $\RV$. 
	We verify the definition of $1$-h-minimality in the structure induced by $\cL(\xi)$.
	Since $K$ is sufficiently saturated, so is $k$ and hence it suffices to verify the following:
	for parameters $\lambda \in \Gamma_{k}^\times$, $A \subset k$ and $\eta \in \RV_{k,\lambda}$, any $\{\xi\} \cup A \cup \RV_{k} \cup \{ \eta \}$-definable subset of $k$ can be $\lambda\abs{m}$-prepared by a finite $A\xi$-definable set, for some $m \in \NN$.
	Since $A$, $\RV_k$ and $\eta$ all live in (definable quotients of) $k$, strong stable embeddedness provides a $\xi$-definable $g(\xi) \in k^m$ such that $X$ is already $\{g(\xi)\} A \cup \RV_{k} \cup \{ \eta \}$-definable.
	By 1-h-minimality of $k$ in the original structure there is finite a $A \cup \{g(\xi)\}$-definable $C$ which $\lambda\abs{m}$-prepares $X$. 
	Then $C$ is clearly also $A\xi$-definable.
\end{proof}

\begin{defn}[Intrinsic $\RV$-Jacobians]
Let $\cT$ be a $1$-h-minimal theory of equicharacteristic zero. 
Assume that the residue field with its induced structure is o-minimal, algebraically bounded or 1-h-minimal and stably embedded in $\RV$.
If the residue field is 1-h-minimal assume additionally that it is strongly stably embedded.
Let $h\colon R\subset (\RV^\times)^n\to (\RV^\times)^n$ be $\emptyset$-definable and let $\xi\in R$. Consider the map
\[
f_\xi\colon \{\alpha\in (\RF^\times)^n\mid \alpha\xi\in R, |h(\alpha\xi)| = |h(\xi)|\}\to (\RF^\times)^n\colon \alpha\mapsto h(\alpha\xi) / h(\xi),
\]
and for $\eta\in R$ define 
\[
(\operatorname{IJac} h)(\eta) = h(\xi)\cdot (\Jac f_\xi)(\eta/\xi) / \xi \in \RV
\]
for any $\xi$ for which $\eta/\xi$ lies in the domain of $f_\xi$ and such that $f_\xi$ is $C^1$ at $\eta/\xi$. We call $\operatorname{IJac}$ the \emph{intrinsic $\RV$-Jacobian of $h$}.
\end{defn}

Note that in the above definition, the map $f_\xi$ might not be $C^1$ at $\eta/\xi$.
In that case the intrinsic $\RV$-Jacobian is not defined at $\eta$.

\begin{lem}\label{lem:int.jac.well.defined}
In any of the above theories, the intrinsic $\RV$-Jacobian is well-defined away from a set of $\RV$-dimension strictly less than $n$.
\end{lem}

\begin{proof}
Let $h\colon R\subset (\RV^\times)^n\to (\RV^\times)^n$ and $f_\xi$ be as in the definition of the intrinsic $\RV$-Jacobian.
For $\eta\in R$ the definition of the intrinsic $\RV$-Jacobian of $h$ at $\eta$ depends on $\xi$. 
Picking a different $\xi$ leads to a rescaling which cancels out in the definition of the intrinsic $\RV$-Jacobian. 
For each $\xi \in (\RV^\times)^n$, we thus obtain a $\emptyset$-definable set $R_{\xi}  \subset \xi k^n$ of $\RV$-dimension at most $n-1$ away from which $\operatorname{IJac}(h)$ is well-defined.
By compactness we may take $R$ to be $\emptyset$-definable.
It remains to check that $\dim_{\RV} R \leq n-1$.
By Theorem~\ref{thm:dim.theory.RV}(\ref{it:local.dimension}), this follows mediately from the fact that each fibre $R \cap \xi k^n$ has $\RV$-dimension at most $n-1$. 
\end{proof}

We now show that in the relevant theories, the $\RV$-Jacobian and the intrinsic $\RV$-Jacobian agree. 
To prove this for algebraically bounded residue fields we explicitly construct good lifts of maps.
For this we need that effectivity is preserved when adding constants from $\Gamma$.

\begin{lem}\label{lem:eff.Gamma}
Let $\cT$ be an effective $1$-h-minimal theory in a language $\cL$ and let $K$ be a model of $\cT$. Then for any $ \Delta \subset \Gamma_K$ the theory of $K$ in $\cL(\Delta)$ is effective.
\end{lem}

Note that $\Th_{\cL(\Delta)}(K)$ is also $1$-h-minimal, since h-minimality is preserved under adding constants from $\RV^\text{eq}$.

\begin{proof}
Let $L$ be a model of $\Th_{\cL(\Delta)}(K)$, $A\subset L$ and $\xi\in \RV_L$ an $A\gamma$-definable element for some finite tuple $\gamma \in \Delta^n$.
In other words, there exists an $A$-definable map $g \colon \RV^n \to \RV$ such that $g(\RV^n_{=\gamma}) = \{\xi\}$.
In particular, $\xi$ belongs to the $A$-definable set $R$ of all $\zeta \in \RV$ for which $g^{-1}(\zeta)$ has $\RV$-dimension $n$.
We necessarily have $\dim_{\RV}R = 0$, by dimension theory. Hence, $R$ is finite.
Effectivity of $\Th_{\cL(A)}(L)$ now yields a finite $A$-definable lift $C$ of $R$.
The average of $\rv^{-1}(\xi) \cap C$ now lifts $\xi$, showing effectivity.
\end{proof}

We also need the notion of $(\gamma, \mu)$-weighted polynomials. 
This builds upon~\cite[Sec.\,5.5]{HK} but we need a slightly more general notion. 
Let $\gamma = (\gamma_1, \ldots, \gamma_n)\in \Gamma_K^n$ and $\mu\in \Gamma_K^n$ be $\emptyset$-definable.
Then a \emph{$(\gamma, \mu)$-weighted monomial} is an expression of the form $\alpha_iX^i$ where $i = (i_1, \ldots, i_n)\in \NN^n$ and $\alpha_i\in \RV^n$ is $\emptyset$-definable satisfying
\[
|\alpha_i|\prod_{j=1}^n\gamma_j^{i_j} = \mu.
\]
A \emph{$(\gamma, \mu)$-weighted polynomial} is simply a finite sum of $(\gamma, \mu)$-weighted monomials.
Let $\RV[X; \gamma, \mu]$ be the $k$-module of $(\gamma, \mu)$-weighted polynomials.
Note that this is not a ring, since the product of two $(\gamma, \mu)$-weighted polynomials is in general not a $(\gamma, \mu)$-weighted polynomial.
However, there is a multiplication
\[
\RV[X; \gamma, \mu]\times \RV[X; \gamma, \mu']\to \RV[X; \gamma, \mu\mu'],
\]
which shows that the direct sum $\bigoplus_\mu \RV[X; \gamma, \mu]$ is a $k$-algebra, where $\mu$ runs over all $\emptyset$-definable elements of $\Gamma_K$.
In particular, this also shows that $\RV[X; \gamma, 1]$ is a $k$-algebra.
However, to ensure definability of the coefficients in Lemma~\ref{lem:alg.bdd.map.RV_gamma} below, we really need to keep track of $\mu$ as well.

The reason for introducing $(\gamma, \mu)$-weighted polynomials is to transport algebraic geometry from $k^n$ to cosets of $(k^\times)^n$ in $(\RV^\times)^n$.
In more detail, recall that $\RV_{=\gamma}^\times\subset \RV^n$ consists of all $(\xi_1, \ldots, \xi_n)\in \RV^n$ for which $|\xi_i| = \gamma_i$.
We put $\RV_{=\gamma} = \RV_{=\gamma}^\times \cup\{0\}$.
This set is isomorphic to $(k^\times)^n$ by translating with any element $\xi\in \RV_{=\gamma}$, and the resulting Zariski topology on $\RV_{=\gamma}$ is independent of the chosen $\xi$.
Note however that such a translation need not be $\emptyset$-definable.

If $f\in \RV[X; \gamma, \mu]$ is a $(\gamma, \mu)$-weighted polynomial, then we can evaluate $f$ at every element $\xi\in \RV_{=\gamma}^\times$ to obtain an element of $\RV_{=\mu}$. 
In more detail, say that $f$ is of the form $\sum_i \alpha_i X^i$ with $\alpha_i\in \RV$ $\emptyset$-definable. Take $a_i\in K$ such that $\rv(a_i) = \alpha_i$ and consider the polynomial $F = \sum_i a_i X^i$ over $K$.
If $\xi\in \RV_{=\gamma}$ let $x\in K^n$ be such that $\rv(x) = \xi$.
Then we define
\[
f(\xi) = \begin{cases} \rv(F(x)) & \text{ if } |F(x)| = \mu, \\
	0 & \text{ if } |F(x)| < \mu. \end{cases}
\]
This is well-defined and shows that $f$ may be interpreted as a $\emptyset$-definable map $\RV_{=\gamma}^\times\to \RV_{=\mu}$.
Equivalently, this amounts to replacing each occurrence of $+$ in $f$ by the modified partial addition $\oplus'$, as in the proof of Lemma~\ref{lem:lift.fin}.
Now, the vanishing locus of $f$ defines a Zariski closed subset of $\RV_{=\gamma}^\times$.
If $k$ is algebraically bounded, then definable subsets of $\RV_{\gamma}$ are contained in the vanishing loci of $(\gamma, \mu)$-weighted polynomials.

\begin{lem}\label{lem:alg.bdd.map.RV_gamma}
Assume that $k$ is algebraically bounded.
Let $\gamma = (\gamma_1, \ldots, \gamma_n)\in \Gamma_K^n, \lambda\in \Gamma_K$ and $f\colon \RV_{=\gamma}^\times\to \RV_{=\lambda}^\times$ be $\emptyset$-definable.
Then there exists a $\emptyset$-definable $\mu\in \Gamma_K$, a $((\gamma, \lambda), \mu)$-weighted polynomial $P$ in the variables $(x_1, \ldots, x_n,y)$, and a $\emptyset$-definable subset $Z\subset \RV_{=\gamma}^\times$ of dimension at most $n-1$ such that for every $x\in \RV_{=\gamma}^\times\setminus Z$ we have that $P(x,\cdot)$ is non-zero in $y$ and $P(x,f(x)) = 0$.
\end{lem}

The crucial point of this lemma is that $P$ has $\emptyset$-definable coefficients. 
This will be needed to lift them to the valued field.
Note that this lemma would be false working only with $\mu=1$.
Indeed, assume that $\gamma\in \Gamma_K$ is $\emptyset$-definable, but that $\RV_{=\gamma}^\times$ contains no $\emptyset$-definable elements.
Then the identity function $f\colon \RV_{=\gamma}^\times\to \RV_{=\gamma}^\times$ is certainly $\emptyset$-definable, but there is no $((\gamma, \gamma), 1)$-weighted polynomial as in the lemma.
On the other hand, we can clearly take $P = x-y$ which is a $((\gamma, \gamma), \gamma)$-weighted polynomial.

\begin{proof}
Take any $\xi\in \RV_{=\gamma}^\times$ and $\eta\in \RV_{=\lambda}^\times$ and consider the function 
\[
h_{\xi, \eta}\colon k^n\to k\colon x\mapsto f(\xi x)/\eta.
\]
This function is definable (with parameters) and so by Lemma~\ref{lem:alg.bdd.monic.poly} there exists a unique monic polynomial $Q_{\xi,\eta}\in k[x_1, \ldots, x_n, y]$ minimal with respect to $\preceq$ such that for $x\in k^n$ away from a lower-dimensional subset we have that $Q_{\xi, \eta}(x,\cdot)\in k[y]$ is non-zero and $Q_{\xi, \eta}(x,h(x)) = 0$.
Define $P_{\xi, \eta}(x,y) = Q_{\xi, \eta}(x/\xi, y/\eta)$.
After multiplying by an element from $k^\times$, we may assume that $P_{\xi, \eta}$ is monic.

We claim that $P = P_{\xi, \eta}$ is independent of the choice $\xi$ and $\eta$.
In particular this shows that the coefficients of $P$ are $\emptyset$-definable, and so $P$ is as desired.
Take $\xi'\in \RV_{=\gamma}^\times$ and $\eta'\in \RV_{=\lambda}^\times$ and define $\alpha = \xi'/\xi\in (k^\times)^n, \beta = \eta'/\eta\in k^\times$.
For every $x\in k^n$ we have that $h_{\xi', \eta'}(x) = h_{\xi, \eta}(\alpha x) / \beta$.
By the unicity in Lemma~\ref{lem:alg.bdd.monic.poly} we therefore obtain that $Q_{\xi', \eta'}(x,y) = Q_{\xi, \eta}(\alpha x, \beta y)$. 
Hence $P_{\xi, \eta}$ is independent of $\xi$ and $\eta$, as desired.
\end{proof}

\begin{prop}\label{prop:ijac.alg.bdd}
Let $\cT$ be an effective $1$-h-minimal theory of equicharacteristic zero valued fields whose residue field is stably embedded and algebraically bounded.
Let $h\colon R\subset (\RV^\times)^n\to (\RV^\times)^n$ be $\emptyset$-definable. 
Then for any $\RV$-Jacobian $\Jac_\RV h$ of $h$, there exists a $\emptyset$-definable set $R'\subset R$ such that $\dim_\RV(R\setminus R') < n$ and such that $\Jac_\RV h = \operatorname{IJac} h$ on $R'$.
\end{prop}

\begin{proof}
Let $G\subset (\RV^\times)^n\times (\RV^\times)^n$ denote the graph of $h$.
By Lemma~\ref{lem:eff.Gamma} and compactness we may assume that $G$ is a subset of $\RV_{=\gamma}^\times\times \RV_{=\lambda}^\times$ for some $\emptyset$-definable elements $\gamma, \lambda\in \Gamma^n_K$.
Then Lemma~\ref{lem:alg.bdd.map.RV_gamma} implies that there exist $\emptyset$-definable $\mu_1, \ldots, \mu_n$, weighted $((\gamma, \lambda), \mu_i)$-polynomials $\overline{P}_1, \ldots, \overline{P}_n$ with $P_i$ in the variables $x_1, \ldots, x_n, y_i$, and a $\emptyset$-definable $R'\subset R$ of dimension at most $n-1$ such that for $x\in R\setminus R'$, we have that $\overline{P}_i(x,\cdot)$ is non-zero and $\overline{P}_i(x,h_i(x)) = 0$.
By effectivity, there exist for each $i$ a polynomial $P_i\in K[x_1, \ldots, x_n, y_i]$ with $\emptyset$-definable coefficients lifting $\overline{P}_i$.

By enlarging $R'$ we may assume by Lemma~\ref{lem:alg.bdd.der} that $h$ is $C^1$ on $R\setminus R'$ with non-vanishing Jacobian.
Let $F\subset \LL (R\setminus R')\times \LL h(R)$ be the $\emptyset$-definable set given by the common vanishing locus of all the $P_i$.
Since $h$ is $C^1$ with non-vanishing Jacobian on $R\setminus R'$, henselianity shows that $F$ defines the graph of a $\emptyset$-definable function $f\colon\LL(R\setminus R')\to K^n$ which lifts $h$.
Moreover, over this set we have by construction $\rv \Jac f = \operatorname{IJac} h$.
\end{proof}

\begin{remark}\label{rem:0.eff.implies.eff}
The proof of Proposition~\ref{prop:ijac.alg.bdd} works by constructing an explicit lift of $h$ over a $R\setminus R'$, for some $R'$ of dimension at most $n-1$. 
In particular, combining this argument with an induction on $\RV$-dimension reproves Lemma~\ref{lem:lift.bijections.mu} for these theories.
\end{remark}

For coarsenings we have the following.
We follow the notation and set-up from Section~\ref{sec:coarsenings}.

\begin{prop}\label{prop:ijac.coarsening}
Let $\cT'$ be a 1-h-minimal theory in a language $\cL$, possibly of mixed characteristic, and let $\cT$ be the theory of a proper nontrivial equicharacteristic zero coarsening of a model of $\cT'$ in the language $\cL_c$.
Let $R,S \subset (\RV_c^{\times})^n \times \RV_c^m$ and $h \colon R \to S$ be $\cL_c$-definable.
Given any $\RV$-Jacobian $\Jac_\RV h\colon R\to \RV_c^\times$ of $h$, there exists a $\cL_c$-definable $R' \subset R$ such that $\dim_{\RV_c}(R') < n$ and  $\Jac_\RV h = \operatorname{IJac} h$ on $R \setminus R'$.
\end{prop}

\begin{proof}
Let $\Delta\leq \Gamma_K^\times$ be the nontrivial proper subgroup by which we coarsen to obtain $|\cdot|_c$.
By Proposition~\ref{prop:coarsening.eff} the theory $\cT$ is effective.
Hence there exists an $\cL_c$-definable lift $f\colon \LL R\to K^n$ of $h$ as in Lemma~\ref{lem:lift.bijections.mu}.
Let $\chi\colon K^n\to \RV_M^N$ be a cell decomposition of $K^n$ for the original valuation $|\cdot |$ with the following properties:
there exists an $m \in \NN$ such that
\begin{enumerate}
\item $\chi$ refines $\rv$,
\item $\LL R$ is a union of $\abs{m}$-twisted boxes of $\chi$,
\item $f$ is $C^1$ on every $\abs{m}$-twisted box of $\chi$,
\item $\rv_\lambda \partial_i f$ is constant on the $\abs{m}\lambda$-twisted boxes of $\chi$ for every $\lambda\in \Gamma_K^\times, \lambda \leq 1$ and $i = 1, \ldots, n$, and
\item for every $j\in \{1, \ldots, n\}$, every coordinate permutation $\sigma\colon K^n\to K^n$, every $x\in K^{n-1}$ and every $\xi \in \RV_c^m$ the map
\[
y\mapsto f_j(\sigma(x,y), \xi)
\]
satisfies the $\lambda$-Jacobian property on every $\abs{m}\lambda$-twisted box of $\chi$ for every $\lambda\in \Gamma_K^\times, \lambda\leq 1$.
\end{enumerate}
Let $\chi': K^n\to \RV_c^{N'}$ be an $\cL_c$-definable cell decomposition (for $|\cdot|_c$) which refines $\chi$, and let $R'\subset R$ consist of all $\xi\in R$ for which $\rv_c^{-1}(\xi)$ is not a twisted box of $\chi'$.
By Theorem~\ref{thm:dim.theory.RV}(\ref{it:generic.points}) we have that $\dim_{\RV_c}(R') < n$.
After further enlarging $R'$, we may assume that $h$ is differentiable on $R'$ while still having that $\dim_{\RV_c} (R') < n$.

We show that $R'$ is as desired.
For $\xi\in R'$ the value of $\operatorname{IJac} h(\xi)$ is determined by $\rv_{k_c, \lambda}(\partial_i h(\xi))$ for $\lambda\in \Delta, \lambda \leq 1$, $i = 1, \ldots, n$.
By construction we have for $x\in \LL R'$ with $\rv_c x = \xi$ that
\[
\rv_\lambda (\partial_i f(x) ) = \rv_{k_c, \lambda}(\partial_i h(\xi)).
\]
So we conclude that $\Jac_\RV h = \rv \Jac f$ is equal to $\operatorname{IJac} h$ on $R'$.
\end{proof}

We have a similar result for o-minimal residue fields.

\begin{prop}\label{prop:ijac.real.closed}
Let $\cT$ be one of the following:
\begin{enumerate}
\item $\cT$ is the theory of an almost real closed valued field, either in the language of valued fields, or in a language with analytic functions as in~\cite{NSV24},
\item $\cT_{\mathrm{omin}}$ is a power-bounded theory of o-minimal fields and $\cT$ is the corresponding $\cT_{\mathrm{omin}}$-convex theory.
\end{enumerate}
Assume that $\cT$ is effective and let $h\colon R\subset (\RV^\times)^n\to (\RV^\times)^n$ be $\emptyset$-definable.
For any $\RV$-Jacobian $\Jac_\RV h\colon R\to \RV^\times$ of $h$ there exists a $\emptyset$-definable $R'\subset R$ with $\dim_\RV( R\setminus R') < n$ and such that $\Jac_\RV h = \operatorname{IJac} h$ on $R'$.
\end{prop}

The idea of the proof is to add a canonical valuation on the residue field, and then deduce the result from the corresponding result when coarsening.

\begin{proof}
We work in a sufficiently saturated model $K$ of $\cT$.
By effectivity, there exists a $\emptyset$-definable lift $f\colon \LL R\to K^n$ of $h$ as in Lemma~\ref{lem:lift.bijections.mu}.
If $\cT$ is the theory of an almost real closed field, let $\cO_{\mathrm{can}}\subset K$ be the convex closure of $\QQ$ in $K$.
Note that it is a proper subring of $\cO_K$ by saturation.
If $\cT$ is a $\cT_{\mathrm{omin}}$-convex theory, then $\cO_K$ is the convex closure of a proper elementary submodel $K'\preceq K$.
By saturation there exists a further proper elementary submodel $K''\preceq K'$, and we define $\cO_{\mathrm{can}}\subset k$ to be the convex closure of $K''$ in $K$.
In either case $\cO_{\mathrm{can}}$ is a valuation ring of $K$ strictly contained in $\cO_K$.
Let $\cL' = \cL\cup\{\cO_\mathrm{can}\}$ and consider the expansion of $K$ with a predicate for $\cO_{\mathrm{can}}$.
Then the theory of $K$ is $1$-h-minimal in the language $\cL'$ with respect to the finer valuation $\cO_{\mathrm{can}}$.
Indeed, both types of fields are 1-h-minimal with respect to $\cO_{\mathrm{can}}$, and adding a symbol for $\cO_K$ amounts to an $\RV$-expansion, which preserves 1-h-minimality by~\cite[Thm.\,4.1.19]{CHR}.  

Now Proposition~\ref{prop:ijac.coarsening} gives an $\cL'$-definable $S\subset R$ with the desired properties.
Let $Z\subset K^n$ be $\cL'$-definable such that $\rv(Z) = R$ and $\dim Z = \dim_\RV R$. 
By~\cite[Lem.\,4.3.4]{CHR} $Z$ is contained in an $\cL$-definable set $Y$ of the same dimension.
Then $R' = \rv(Y)$ is as desired.
\end{proof}

Using the intrinsic $\RV$-Jacobian, we can give an alternative definition of the categories $\mu \RV[n]$.

\begin{defn}
Let $\cT$ be one of the theories above, in which an intrinsic $\RV$-Jacobian is defined.
Let $n$ be a positive integer.
Denote by $\mu \RV_{\intr}[n]$ the category with objects $(R, \omega)$ where $R\subset (\RV^\times)^n\times \RV^m$ is $\emptyset$-definable and with finite-to-one projection onto $(\RV^\times)^n$, and where $\omega\colon R\to \RV^\times$ is $\emptyset$-definable.
A morphism in this category consists of a $\emptyset$-definable bijection $h\colon R\to S$ such that
\[
\omega(\xi) = \rho(h(\xi))\operatorname{IJac}(h)(\xi) \text{ for all } \xi \text{ away from a set of $\RV$-dimension $<n$},
\]
and such that
\[
|\omega(\xi)||\xi| = |\rho(h(\xi))||h(\xi)| \text{ for all } \xi\in R.
\]
\end{defn}

This definition is closer to the definition of Hrushovski--Kazhdan.
In effective theories, the categories $\mu \RV[n]$ and $\mu_\intr \RV[n]$ coincide.

\begin{prop}
Let $\cT$ be effective and assume that there is an intrinsic $\RV$-Jacobian.
Then the categories $\mu \RV[n]$ and $\mu_\intr \RV[n]$ are equivalent.
\end{prop}

\begin{proof}
This follows from Propositions~\ref{prop:ijac.alg.bdd}, \ref{prop:ijac.coarsening}, and~\ref{prop:ijac.real.closed}.
\end{proof}





%




\printbibliography 

\end{document}